%% file: main.tex
\documentclass[11pt]{amsart}
\pdfoutput=1 

\input{packages}
\input{preamble}

\title{The DPG-star method}
\author[Demkowicz]{Leszek Demkowicz}
\author[Gopalakrishnan]{Jay Gopalakrishnan}
\author[Keith]{Brendan Keith}
\address[Demkowicz]{
	The Institute for Computational Engineering and Sciences (ICES),
	The University of Texas at Austin,
	Austin, 
	TX 78712,
	U.S.A.
}
\address[Gopalakrishnan]{
	Portland State University, Portland, OR 97207-0751, U.S.A.
}
\address[Keith]{
  Technische Universität München,
  Garching,
  85748,
  Germany
}

\thanks{\textbf{Acknowledgments.}
	This work was partially supported with grants by NSF (DMS-1418822), AFOSR (FA9550-17-1-0090), and ONR (N00014-15-1-2496).
	Part of this work was done while the second author was at The University of Texas at Austin on a J. Tinsley Oden Faculty Fellowship.
  The third author was additionally supported in part by the 2017 Graduate School University Graduate Continuing Fellowship at The University of Texas at Austin.
	The authors express their gratitude to Nathan V. Roberts and Socratis Petrides for their assistance during the numerical experiments with certain features of the Camellia and $hp$2D software packages, respectively.
}

\date{}

\begin{document}

\begin{abstract}
  This article introduces the DPG-star (from now on, denoted DPG*)
  finite element method. It is a method that is in some sense dual to
  the discontinuous Petrov--Galerkin (DPG) method.  The DPG
  methodology can be viewed as a means to solve an overdetermined
  discretization of a boundary value problem.  In the same vein, the
  DPG* methodology is a means to solve an underdetermined
  discretization.  These two viewpoints are developed by embedding the
  same operator equation into two different saddle-point problems. The
  analyses of the two problems have many common elements.  Comparison
  to other methods in the literature round out the newly garnered
  perspective.  Notably, DPG* and DPG methods can be seen as
  generalizations of $\mcL\mcL^\ast$ and least-squares methods,
  respectively.  \textit{A priori} error analysis and
  \textit{a posteriori} error control
  for the DPG* method are considered in detail.  Reports of several
  numerical experiments are provided which demonstrate the essential
  features of the new method.  A notable difference between the
  results from the DPG* and DPG analyses is that the convergence rates
  of the former are limited by the regularity of an extraneous
  Lagrange multiplier variable.
\end{abstract}

\maketitle


\section{Introduction} 
\label{sec:introduction}

The \emph{ideal Discontinuous Petrov--Galerkin} (\emph{DPG})
\emph{Method with Optimal Test Functions}
\cite{demkowicz2010class,demkowicz2011class} admits three
interpretations \cite{demkowicz2015encyclopedia}.  First, it can be
viewed as a Petrov--Galerkin (PG) discretization in which {\em
  optimal} test functions are computed on the fly.  Here, the word
``optimal'' refers to the fact that the test functions realize the
supremum in the discrete inf-sup stability condition and, therefore,
the PG discretization automatically inherits the stability of the
continuous method.  The DPG method can also be viewed as a minimum
residual method in which the residual is measured in a dual norm
implied by an underlying test norm.  Finally, the DPG method can be
viewed also as a mixed method \cite{cohen2012adaptivity} wherein one
simultaneously solves for the Riesz representation of the residual
\textemdash{} the so-called {\em error representation function}
\textemdash{} and the approximate solution.  All three equivalent
interpretations involve the inversion of a Riesz operator on the test
space which, in general, cannot be done exactly and has to be
approximated.  This naturally leads to the introduction of an {\em
  enriched} or {\em search} test space \textemdash{} having dimension
larger than that of the latent trial space \textemdash{} and a
discretized Riesz operator.  In this way, the corresponding {\em
  practical DPG method} retains its three interpretations, although
now with {\em approximate} optimal test functions, an {\em
  approximate} residual, and an {\em approximate} error representation
function \cite{gopalakrishnan2014analysis,Nagaraj2015,Carstensen15}.

In the DPG method, the word ``discontinuous'' corresponds to the use
of discontinuous, but conforming, test functions (from broken spaces)
which make the whole methodology computationally efficient.  These
broken spaces also naturally lead to (hybridized) interface solution
variables.  Broken space formulations provide a foundation for the DPG
methodology and can be developed for any well-posed variational
formulation~\cite{Carstensen15}.

Originally motivated by the duality theory in \cite{Keith2017Goal}, in
this article, the DPG methodology is expanded on by reconsidering how
a given operator equation can be embedded into a DPG-type saddle-point
problem.  In turn, the minimum residual principle underpinning DPG
methods can be discarded for a minimum norm principle and a dual class
of methods (i.e. DPG* methods\footnote{DPG* methods are distinct from
  the saddle point least squares methods \cite{BACUTA20152920} which have
  separately been contenders for being named  ``dual'' to DPG methods.}) can be
introduced.
Ultimately, an entire class of stable DPG-type mixed methods can be proposed, simply by changing loads in the saddle-point problem and the interpretation of the corresponding solution variables.
This broad perspective can help relate several different methods, including weakly conforming least squares methods \cite{Ernesti2017} and $\mcL\mcL^\ast$ methods \cite{cai2001first} to the existing DPG theory.

This article provides a number of important \textit{a priori} error
analysis results for DPG* methods.
Unlike standard DPG methods, the convergence rate of a DPG* solution
is not controlled solely by the regularity of the solution itself, but
instead also by the regularity of a Lagrange multiplier variable found
in the corresponding saddle-point formulation.  This article also
expands on the recent \textit{a posteriori} error estimation theory
first introduced in \cite{Keith2017Goal} and re-establishes much of
Repin, Sauter, and Smolianski's abstract \textit{a posteriori} theory
for mixed methods \cite{repin2007two} in the present context.  In
addition, it includes several standard numerical examples to verify
the theory for DPG* methods, including one example employing
$hp$-adaptive mesh refinement.  This work is part of the PhD thesis
\cite{Keith2018thesis}.


\section{The DPG and DPG* methods} 
\label{sec:the_dpg_and_dpg_methods}

\subsection{Operator equations} 
\label{sub:operator_equations}

Central to this paper are the twin relatives of the operator equation
\begin{equation}
  \label{eq:Au=f}
  B u = \ell,
\end{equation}
given in~\eqref{eq:dpgA} and~\eqref{eq:dpg*A} below.  Here
$B : \U \to \V^{\prime}$ is a bounded linear operator from a Hilbert space
$\U$ to the dual of a Hilbert space $\V$, $\ell \in \V^{\prime}$ is given, and
$u \in \U$ is to be found. All spaces here are over $\RRR$, the real
field. In any Hilbert space $X$, the action of a functional
$E \in X^{\prime}$ on $x \in X$ is denoted by $\ip{ E, x}_X$.  When the space
is clear from context, we also use $E(x)$ to denote the same
number. Let $B^{\prime} : \V \to \U^{\prime}$ be the dual of $B$ defined by
$\ip{ B^{\prime} v, u}_U = \ip{Bu, v}_V$ for all $u \in \U$ and $v \in \V$.
The reason for using $\prime$ instead of $\ast$ to denote the dual
operator will become evident when a different, but related, notion of
duality is introduced in
\cref{sub:ultraweak_formulations}.\footnote{Therefore, the asterisk in
  the DPG* method is evocative of the connections to the $\mcL\mcL^\ast$ method \cite{cai2001first}.}

The two reformulations are as follows.
\begin{align}
  & \text{Find } u \in \U \text{ and } \veps \in \V \text{ satisfying } 
  && \left\{
  \begin{alignedat}{3}
  &\scR_\V \veps + B u &&= \ell,
  \\
  &B^{\prime} \veps &&= 0.     
  \end{alignedat}
\right.
  \label{eq:dpgA}
  \\
  & \text{Find } u \in \U \text{ and } \lambda \in \V  \text{ satisfying
    }
  &&
  \left\{
  \begin{alignedat}{3}
    &\scR_\U u - B^{\prime} \lambda &&= 0,
    \\
    &B u &&= \ell.
  \end{alignedat}
\right.
  \label{eq:dpg*A}
\end{align}
Here $\scR_\V: \V \to \V^{\prime}$ is the Riesz operator acting on $\V$, defined using
the inner product $(\cdot, \cdot)_\V$ by
$(\scR_\V v)(\nu) = (v, \nu)_\V$ for all $v, \nu \in \V$. The Riesz operator
$\scR_\U$ is defined similarly. It is immediate that if $u$
solves~\eqref{eq:Au=f}, then with $\veps=0$ it solves~\eqref{eq:dpgA},
revealing a relationship between~\eqref{eq:dpgA} and~\eqref{eq:Au=f}.
The relationship between~\eqref{eq:dpg*A} and~\eqref{eq:Au=f} is also
easy to guess: any solution $(u, \lambda)$ of~\eqref{eq:dpg*A} is such
that the $u$ component solves~\eqref{eq:Au=f}. We shall see below
that, even though related, these formulations are not fully equivalent
to~\eqref{eq:Au=f}. The formulation~\eqref{eq:dpgA} is the one on
which the DPG method is based. The formulation~\eqref{eq:dpg*A}, when
discretized, results in the new DPG* method, as we shall see.


Formulations~\eqref{eq:dpgA} and~\eqref{eq:dpg*A} are structurally similar,
differing mainly in the position where the load $\ell$ is placed.  Due to
the structural similarity, both formulations can be viewed at once as
instantiations of the following general saddle-point problem
\begin{align}
  & \text{Find } v \in \scV \text{ and } w  \in \scU 
    \text{ satisfying}
  &&
  \left\{
    \begin{alignedat}{3}
      \Riesz_\scV &v + \scB w &&= F \,, \\
      \scB^{\prime} & v &&= G \, ,
    \end{alignedat}
  \right.
  \label{eq:Mixed_general}
\end{align}
on some Hilbert spaces $\scU$ and $\scV$, some bounded linear operator
$\scB: \scU \to \scV^{\prime}$, and some given functionals $F \in \scV^{\prime}$ and
$G \in \scU^{\prime}$. Indeed, with
\begin{equation}
  \label{eq:F=f_G=0_setting}
  \scV = V,\; \scU = U, \; \scB = B,\; F = \ell,\; G=0,  
\end{equation}
we obtain~\eqref{eq:dpgA}. If instead, we set 
\begin{equation}
  \label{eq:F=0_G=f_setting}
  \scV = U,\; \scU = V,\; \scB = B^{\prime},\; F=0,\; G=\ell,
\end{equation}
then we obtain~\eqref{eq:dpg*A}. Admittedly, the alternative mixed form
obtained by exchanging $\scB$ and $\scB^{\prime}$ in~\eqref{eq:Mixed_general}
is more natural for studying the DPG* method and even aligns with the
standard notations in mixed method theory~\cite{Boffi2013}. Yet, we
have chosen to work with~\eqref{eq:Mixed_general} to facilitate
comparison with existing DPG literature where the form
of~\eqref{eq:Mixed_general} is more natural.

We proceed under the assumption that $\scB$ is bounded below, i.e.,
there is a $\gamma>0$ such that 
\begin{equation}
\label{eq:Bbddbelow}
  \Vert \scB \mu \Vert_{\scV^{\prime}} \ge \gamma \Vert \mu
  \Vert_\scU,
  \qquad \forall\, \mu\in \scU.
\end{equation}
Note that the maximum of all such $\gamma$ is simply $\|\scB\inv\|$.
Under this assumption, the mixed system~\eqref{eq:Mixed_general} has a
unique solution for any $F \in \scV^{\prime}$ and $G\in \scU^{\prime}$ (see
e.g.~\cite{Boffi2013}). Obviously~\eqref{eq:Bbddbelow} can also be
written out as an $\inf$-$\sup$ condition.
Here and throughout, for any Banach space $X$, the right annihilator of a subset $Y \subset X$ and the
left annihilator of a $Z \subset X^{\prime}$ are defined by
\begin{align}
  \label{eq:rightann}
  Y^\perp
  & = \{ E \in X^{\prime}: \ip{ E, y}_X=0\text{ for all }y\in Y\},
  \\ 
  \label{eq:leftann}
  \lprp Z
  & = \{ x \in X: \ip{ E,x}_X=0\text{ for all }E\in Z\}.
\end{align}
Recall that if $Y\subset X$ is a closed subspace, $\overline{Y} = Y$, then $Y^\perp$ is isomorphic to $(\quotient{X}{Y})^\prime$.

Now consider the mixed system~\eqref{eq:Mixed_general} when $G=0$ and
the related problem of finding $w \in \scU$ satisfying
\begin{equation}
  \label{eq:Bmu=F}
  \scB w = F.
\end{equation}
The regularizing effect of the saddle-point formulation above is already evident:
while~\eqref{eq:Mixed_general} is always solvable
under~\eqref{eq:Bbddbelow}, the related problem~\eqref{eq:Bmu=F} is
solvable {\em provided} $F$ satisfies the compatibility condition
$ F \in (\kernel\,\scB^{\prime})^\perp.$
To reiterate the above observation that~\eqref{eq:Bmu=F} is not fully
equivalent to \eqref{eq:Mixed_general}, we may view~\eqref{eq:Bmu=F} as an
{\em overdetermined system}. Overdetermined systems are solvable only
if they are consistent, i.e., have compatible data. Irrespective of
the data, what the mixed system~\eqref{eq:Mixed_general} solves can be
seen by eliminating $v$ (and recalling that $G=0$), namely
\begin{equation}
  \scB^{\prime} \Riesz_\scV\inv \scB w = \scB^{\prime} \Riesz_\scV\inv F.
\label{eq:normal_equation}
\end{equation}
Equation~\cref{eq:normal_equation} can be immediately identified with what is referred to as a ``normal equation'' in linear algebra.
This is a regularized version of~\eqref{eq:Bmu=F}.
Indeed, whenever~\eqref{eq:Bmu=F} has a solution, it must
be unique due to~\eqref{eq:Bbddbelow}, and that unique solution is
recovered by~\eqref{eq:normal_equation}. However,
\eqref{eq:normal_equation} has a unique solution even
when~\eqref{eq:Bmu=F} does not.

Next, considering the case of $F=0$, we may likewise argue that the mixed
system~\eqref{eq:Mixed_general} also helps us solve {\em
  underdetermined systems}. Under the same
assumption~\eqref{eq:Bbddbelow}, consider 
\begin{equation}
  \scB^{\prime} v = G.
\label{eq:adjoint_problem}
\end{equation}
Assumption~\eqref{eq:Bbddbelow} implies that $\scB^{\prime}$ is surjective,
so~\eqref{eq:adjoint_problem} is always solvable, but its solution
need not be unique in general. Thus,~\eqref{eq:adjoint_problem} may be viewed as an
example of an underdetermined system.
Similar to \cref{eq:normal_equation}, the solution variable $v$ can be readily eliminated from~\cref{eq:Mixed_general} (now recalling that $F=0$):
\begin{equation}
  \scB^{\prime} \Riesz_\scV\inv \scB w = -G.
\label{eq:normal_equation2}
\end{equation}
This equation is in correspondence with a different normal equation (one of the second type \cite{bjorck1996numerical}).
Notice that the left-hand side operator $\scB^{\prime} \Riesz_\scV\inv \scB:\scU\to\scU^\prime$ is the same in both~\cref{eq:normal_equation,eq:normal_equation2} and that the solution $v$ in~\cref{eq:normal_equation2} can be recovered by the relationship $v = -\Riesz_\scV\inv \scB w$.

To reconsider how the mixed
system~\eqref{eq:Mixed_general} converts~\eqref{eq:adjoint_problem} 
into a uniquely solvable problem, we use orthogonal complements in 
Hilbert spaces, which we distinguish from the annihilators in~\eqref{eq:rightann} and \eqref{eq:leftann} by placing the symbol $\perp$ as a
subscript. Thus, while $(\kernel\,\scB^{\prime})^\perp$ is a subspace of
$\scV^{\prime}$, the notation $(\kernel\, \scB^{\prime})_\perp$ indicates the
subspace of $\scV$ defined by
\begin{equation}
  (\kernel\,\scB^{\prime})_\perp = \{ v \in \scV : {(v, \nu_0)_\scV = 0},\; \forall\, \nu_0 \in \kernel\, \scB^{\prime}\}
  .
\end{equation}
One may then decompose any solution
of~\eqref{eq:adjoint_problem} into $\scV$-orthogonal components: 
\begin{equation}
  \label{eq:decomp}
  v = v_0 + v_\perp, \qquad
  v_0 \in \kernel \, \scB^{\prime}, \quad v_\perp \in (\kernel\,\scB^{\prime})_\perp.
\end{equation}
Observe that 
\begin{equation}
  \label{eq:up_low_perp}
  (\kernel \, \scB^{\prime})^\perp = \scR_\scV (\kernel\, \scB^{\prime})_\perp
  .
\end{equation}
Since $F=0$, testing the first equation of~\eqref{eq:Mixed_general}
with $v_0$, we find that what~\eqref{eq:Mixed_general} selects as its
unique solution $v$ is in fact simply $v_\perp.$ 

Returning to the case of general $F$ and $G,$ we collect a few
identities in the next result.
First, note that one may also decompose $F$ into orthogonal components:
\begin{equation}
  \label{eq:Fdecomp}
  F = F^0 + F^\perp, \qquad
  F^0 \in \scR_\scV(\kernel \, \scB^{\prime}), \quad F^\perp \in \scR_\scV(\kernel \, \scB^{\prime})_\perp = (\kernel\,\scB^{\prime})^\perp.
\end{equation}
Second, note that when~\eqref{eq:Bbddbelow}
holds, $\normm{\mu}_{\scU} = \| \scB \mu \|_{\scV^{\prime}}$ generates an
equivalent norm on $\scU$, $\|\scB\inv\|\inv\|\mu\|_\scU \leq \normm{\mu}_{\scU} \leq \|\scB\|\|\mu\|_\scU$, and we may define
\begin{equation}
\label{eq:DualOfEnergyNorm}
  \normm{G}_{\scU^{\prime}} = \sup_{0 \ne \mu \in \scU} 
  \frac{\ip{G,\mu}_{\scU}}{ \normm{\mu}_{\scU}}.
\end{equation}

\begin{proposition}
  \label{prop:identities}
  Suppose $F \in \scV^{\prime}$, $G \in \scU^{\prime}$, $v \in \scV$ and
  $w \in \scU$ solve~\eqref{eq:Mixed_general} and let $v_0$ and
  $v_\perp$ be the unique components of the decomposition of $v$
  in~\eqref{eq:decomp}.
  Similarly, let $F^0$ and $F^\perp$ be the unique components of the decomposition of $F$
  in~\eqref{eq:Fdecomp}. Then the following identities hold:
  \begin{gather}
    \label{eq:identity1}
    \| v_0 \|_\scV^2 + \| \scR_\scV v_\perp + \scB w \|_{\scV^{\prime}}^2
     = \| F \|_{\scV^{\prime}}^2,
    \\
    \label{eq:identity2}
    \| v_0 \|_\scV^2 + \| \scB w\|_{\scV^{\prime}}^2 
    = \| F - \scR_\scV v_\perp\|_{\scV^{\prime}}^2.
  \end{gather}
  Moreover, $v_0 = \scR_\scV\inv F^0$ and
  \begin{gather}
    \label{eq:identityv_0}
    \| v_0 \|_\scV
    = \| F^0 \|_{\scV^\prime}
    ,
    \\
    \label{eq:identityw}
    \| \scB w\|_{\scV^{\prime}} 
    = \| F^\perp - \scR_\scV v_\perp\|_{\scV^{\prime}}.
  \end{gather}
  If in addition, \eqref{eq:Bbddbelow} holds, then for any
  $F \in \scV^{\prime}$, $G \in \scU^{\prime}$, there is a unique $v \in \scV$ and
  $w \in \scU$ satisfying~\eqref{eq:Mixed_general} and 
  the following identities hold:
  \begin{gather}
    \label{eq:identity3}
    \| v_\perp \|_\scV = \normm{G}_{\scU^{\prime}},
    \\
    \label{eq:identity4}
    \| v\|_\scV^2 + \normm{w}_\scU^2 = \| F - \scR_\scV v_\perp
    \|_{\scV^{\prime}}^2 + \normm{G}_{\scU^{\prime}}^2.
  \end{gather}
  If in addition, either $F \in (\kernel\,\scB^{\prime})^\perp$ or $\scB$ is
  a bijection, then $v_0 = 0$ and
  \begin{gather}
    \label{eq:identity5}
    \| v \|_\scV = \normm{G}_{\scU^{\prime}}.
  \end{gather}
\end{proposition}
\begin{proof}
  For any $\nu_0 \in \kernel\, \scB^{\prime}$, we have
  $(\scR_\scV\inv \scB w, \nu_0)_\scV = \ip{\scB w, \nu_0}_\scV =
  \ip{\scB^{\prime}\nu_0, w}_\scU=0.$
  Hence $\scR_\scV\inv \scB w$ is in $(\kernel\,\scB^{\prime})_\perp.$
  Therefore, when the first equation of~\eqref{eq:Mixed_general} is
  rewritten as
  \begin{equation}
    \label{eq:two_orth_terms}
    v_0 + (v_\perp + \scR_\scV\inv \scB w) = \scR_\scV\inv F,
  \end{equation}
  an application of the Pythagorean theorem gives~\eqref{eq:identity1}.
  Rewriting~\eqref{eq:two_orth_terms} as
  $v_0 + \scR_\scV\inv \scB w = \scR_\scV\inv F - v_\perp,$ and
  applying the Pythagorean theorem again, we obtain~\eqref{eq:identity2}.
  Rewriting~\cref{eq:two_orth_terms} instead as
  \begin{equation}
    v_0 - \scR_\scV\inv F^0 = \scR_\scV\inv F^\perp -  (v_\perp + \scR_\scV\inv \scB w),
  \end{equation}
  we note that $v_0 = \scR_\scV\inv F^0$ and
  $\scR_\scV\inv F^\perp = v_\perp + \scR_\scV\inv \scB w$, by
  orthogonality.  Equations~\cref{eq:identityv_0,eq:identityw} are now
  obvious.
  
  Next, if~\eqref{eq:Bbddbelow} holds, then standard mixed
  theory~\cite{Boffi2013} gives existence of a unique
  $(v, w ) \in \scV \times \scU$, and $\normm{\cdot}_\scU$ is an
  equivalent norm on $\scU$. 
  To prove~\eqref{eq:identity3}, we begin by noting that the isometry induced by
  $\scR_\scV$ implies 
  \begin{equation}
    \| v_\perp \|_\scV = 
    \sup_{\nu_\perp \in (\kernel\,\scB^{\prime})_\perp}
    \frac{(\nu_\perp, v_\perp)_\scV }{ \| \nu_\perp\|_\scV}
    = 
    \sup_{\nu_\perp \in (\kernel\,\scB^{\prime})_\perp}
    \frac{\ip{ \scR_\scV \nu_\perp, v_\perp}_\scV }{ \| \scR_\scV
      \nu_\perp\|_{\scV^{\prime}}}
    = 
    \sup_{E^\perp \in \scR_\scV (\kernel\,\scB^{\prime})_\perp}
    \frac{\ip{ E^\perp, v_\perp}_\scV }{ \| E^\perp \|_{\scV^{\prime}}}.
  \end{equation}
  Here and throughout, supremums over spaces are only taken over
  nonzero elements of the space.
  Again, from the identity $\range \,\scB = (\kernel \, \scB^{\prime})^\perp$ and \cref{eq:up_low_perp}, we conclude that
  \[
  \| v_\perp \|_\scV
  =
  \sup_{E^\perp \in \range\, \scB}
  \frac{\ip{ E^\perp, v_\perp}_\scV }{ \| E^\perp \|_{\scV^{\prime}}}
  = 
  \sup_{\mu \in \scU}
  \frac{\ip{ \scB \mu, v_\perp}_\scV }{ \| \scB \mu \|_{\scV^{\prime}}} 
  = 
  \sup_{\mu \in \scU}
  \frac{\ip{  \mu, \scB^{\prime}v_\perp}_\scV }{ \normm{ \mu }_\scU}.
  \]
  Thus,~\eqref{eq:identity3} follows after using the second equation
  in~\eqref{eq:Mixed_general}, namely
  $G = \scB^{\prime} (v_0 + v_\perp) = \scB^{\prime} v_\perp$.
  Identity~\eqref{eq:identity4} now follows by squaring both sides
  of~\eqref{eq:identity3} and adding it to~\eqref{eq:identity2}.

  Finally, when $\scB$ is a bijection or
  $F \in (\kernel\,\scB^{\prime})^\perp$, we conclude that $F^0 = 0$.
  Therefore, $v_0=0$ and~\eqref{eq:identity5} follows
  from~\eqref{eq:identity3}.
\end{proof}

Identities like~\eqref{eq:identity4} have often been referred to by
the name {\em hypercircle identities}~\cite{repin2007two} and their
use in {\it a posteriori} error estimation is now standard. We shall
return to this in \cref{sec:a_posteriori_error_control}.


\subsection{Forms and discretization} 
\label{sub:subsection_name}

It is traditional to write mixed systems
using a bilinear form defined by 
\begin{equation}
  \label{eq:b-form-B}
  b(\mu, \nu) = \ip{ \scB \mu, \nu}_\scV  
\end{equation}
for all $\mu \in \scU, \nu \in \scV$. In terms of $b$, 
the mixed problem~\eqref{eq:Mixed_general} is to
find $v \in \scV$ and $w \in \scU$ satisfying
\begin{equation}
  \left\{
    \begin{alignedat}{5}
      &(v, \nu)_\scV
      + 
      b(w, \nu) &&
      = F(\nu)\,,\quad && \forall\, \nu \in \scV,
      \\
      &b(\mu, v) &&= G(\mu)\,,\quad &&
      \forall\, \mu\in\scU.
    \end{alignedat}
  \right.
\label{eq:Mixed-General-form}
\end{equation}
Suppose $b$ arises from a weak formulation of a PDE on a domain
$\om \subset \RRR^d$, which is partitioned into a mesh $\oh$ of
finitely many open connected elements $K$ with Lipschitz boundaries
$\d K$,  such that $\shoverbar{\om}$ is the union of
the closures of all mesh elements $K$ in $\oh$. In this scenario, 
if there are Hilbert spaces $\scV(K)$ on each mesh element $K$ such that 
\begin{equation}
  \label{eq:broken_space}
  \scV = \prod_{K\in \oh} \scV(K),  
\end{equation}
then the system~\eqref{eq:Mixed-General-form}, in the case $G=0$, is
called a {\em DPG formulation}.
In the case $F=0$, it is called a {\em DPG* formulation}. Spaces of the form~\eqref{eq:broken_space} are called
{\em broken spaces} \cite{Carstensen15}.

A discrete method based on~\eqref{eq:Mixed-General-form} would require
a pair of discrete finite-dimensional spaces $\scU_h \subset \scU$ and
$\scV_h \subset \scV$, not necessarily of the same dimension. The
discrete problem would then read as the problem of finding $v_h \in
\scV_h$ and $w_h \in \scU_h$ satisfying 
\begin{equation}
  \left\{
    \begin{alignedat}{5}
      &(v_h, \nu)_\scV
      + 
      b(w_h, \nu) &&
      = F(\nu)\,,\quad && \forall\, \nu \in \scV_h,
      \\
      &b(\mu, v_h) &&= G(\mu)\,,\quad &&
      \forall\, \mu\in\scU_h.
    \end{alignedat}
  \right.
\label{eq:Mixed-General-form-discrete}
\end{equation}
When $\scV$ is a broken space of the form~\eqref{eq:broken_space}, 
$\scV_h$ can be chosen to consist of functions with no continuity
constraints across mesh element interfaces. Then the case $G=0$ delivers
{\em DPG methods} and the case $F=0$ delivers {\em DPG* methods}.  In
both cases, we must typically find $\scV_h$ with
$\text{dim}(\scV_h) > \text{dim}(\scU_h)$ with provable discrete
stability.

A key feature of~\eqref{eq:Mixed-General-form-discrete} is that the
the top left form, $(v, \nu)_\scV$, being an inner product, is always
coercive. Hence the discrete stability
of~\eqref{eq:Mixed-General-form-discrete} is guaranteed solely by a
discrete $\inf$-$\sup$ condition, which is often {\em easy to obtain} in practice
since we can increase $\text{dim}(\scV_h)$ without violating the coercivity of the top
left term. This $\inf$-$\sup$ condition has been analytically
established for various DPG methods through the construction of local
\cite{gopalakrishnan2014analysis,Carstensen15,Nagaraj2015} or global
\cite{carstensen2016low} Fortin operators on generously large test
spaces. The same $\inf$-$\sup$ condition also confirms the stability of the
corresponding DPG* methods. An alternative characterization of the
methods above can be found in a Petrov--Galerkin form in
\cite[Section~4.1]{Keith2017Goal}.

Upon the choice of bases $\{v_i\}$ and $\{w_j\}$ for the discrete spaces $\scV_h$ and $\scU_h$, \cref{eq:Mixed-General-form-discrete} can be identified with the following system of matrix equations:
\begin{equation}
    \begin{bmatrix}
        \sfG\phantom{\T} & \sfB\\
        \sfB\T & 0
    \end{bmatrix}
    \begin{bmatrix}
        \sfv\\
        \sfw
    \end{bmatrix}
    =
    \begin{bmatrix}
        \sff\\
        \sfg
    \end{bmatrix}
    .
\label{eq:Mixed-General-form-matrix}
\end{equation}
Here, $\sfB$ is a rectangular matrix with coefficients determined by the bilinear form, $\sfB_{ij} = b(w_j,v_i)$, and, by conventional notation, $\sfG$ is a Gram matrix governed by the chosen inner product, $\sfG_{ik} = (v_i, v_k)_\scV$.
Naturally, the vectors $\sff_i=F(v_i)$, $\sfg_j=G(w_j)$ are identified with the two loads in \cref{eq:Mixed-General-form-discrete} and the vectors $\sfv$ and $\sfw$ correspond to the coefficients of the chosen basis functions.
In the broken space setting~\eqref{eq:broken_space}, the Gram matrix can be block-diagonal.
In that case, inverting $\sfG$ is computationally feasible and the Schur complement of \cref{eq:Mixed-General-form-matrix} (cf. \cref{eq:normal_equation,eq:normal_equation2}) may be used to solve for the vector $\sfw$ in a much smaller system, independent of $\sfv$:
\begin{equation}
  \sfB\T\sfG\inv\sfB
  \sfw
  =
  \sfB\T\sfG\inv\sff
  -
  \sfg
  \,.
\label{eq:SchurComplement}
\end{equation}
Notice that the DPG stiffness matrix, $\sfB\T\sfG\inv\sfB$, is always symmetric and positive-definite and that after solving for $\sfw$ via~\cref{eq:SchurComplement}, $\sfv$ can always be recovered with only local cost, i.e., $\sfv = \sfG\inv(\sff-\sfB\sfw)$.
Construction of the stiffness matrix $\sfB\T\sfG\inv\sfB$ with broken spaces is considered in detail in \cite{Keith2017Discrete}.


\subsection{Ultraweak formulations} 
\label{sub:ultraweak_formulations}

Many PDEs originate in the following strong form:
\begin{equation}
  \LL u
    =
    f
    \,,
\label{eq:StrongFormulation}
\end{equation}
where $\LL$ is a linear differential operator and $f$ is a prescribed
function.  It is possible to give many general DPG and DPG* formulations for
such operator equations using the framework
of~\cite[Appendix~A]{demkowicz2016spacetime} (which generalizes the
Friedrichs systems framework in~\cite{ErnGuermCapla07,bui2013unified,wieners2016skeleton}).
Let $d, k, m\ge 1$ be integers and let $\om \subseteq \RRR^d$ be a bounded
open set.  We use multiindices $\alpha=(\alpha_1,\ldots \alpha_d)$ of
length $|\alpha| = \alpha_1 + \cdots + \alpha_d \le k$.  Suppose we
are given functions $a_{ij\alpha}: \om \to \RRR$ for each $i = 1, \ldots, l$,
$j = 1, \ldots, m$, and each $|\alpha| \le k$. Let $\LL$ be the
differential operator acting on functions $u: \om \to \RRR^m$ such that
\begin{equation}
  \label{eq:Ldef}
  [\LL u]_i = \sum_{j=1}^m \sum_{|\alpha|\le k} \partial^\alpha (
  a_{ij\alpha} u_j), 
  \qquad
  i\in\{1,\ldots,l\}
  .
\end{equation}
Wherever appropriate, let $L^2$ denote \emph{either} the $l$- or $m$-fold Cartesian product of $L^2(\om)$. Likewise, let $\D$ denote \emph{either} the $l$- or $m$-fold Cartesian product of $\D(\om)$, where $\D(\om)$ is the space of infinitely
differentiable functions that are compactly supported on $\om$ (and
accordingly, $\D^{\prime}$ denotes distributional vector fields). Let $\LL^*$
be the formal adjoint differential operator of $\LL$, i.e., it
satisfies $(\LL \phi, \psi)_{L^2} = (\phi, \LL^* \psi)_{L^2}$ for all
$\phi, \psi \in \D$.
From now on, we will simply denote all such $L^2$-inner products on $\om$ as $(\cdot,\cdot)_\som = (\cdot,\cdot)_{L^2}$.
Likewise, all $L^2$-inner products restricted to a measurable subset $K\subset \Omega$ will be denoted $(\cdot,\cdot)_{\selement}$.

The action of $\LL^*$ on $v: \om \to \RRR^l$ is given by
\begin{equation}
  \label{eq:L*defn}
  [\LL^*v ]_j = \sum_{i=1}^l \sum_{|\alpha| \le k} (-1)^{|\alpha|}
  {a_{ij\alpha}}\, \partial^\alpha v_i,
  \qquad
  j\in\{1,\ldots,m\}
  .
\end{equation}
We assume that the coefficients $a_{ij\alpha}$ are such that both 
$\LL u$ and $\LL^* v$ are well-defined distributions for all $u,v \in L^2$, i.e.,
\begin{subequations}
  \label{eq:UWasm}
\begin{equation}
  \label{eq:LL*assume}
  \text{$\LL u$
    and $ \LL^* v$ are in $\D^\prime$ for all
    $ u,v \in L^2$.}
\end{equation}
(This holds e.g., if $a_{ij\alpha}$ are constant.)

We may now define
Sobolev-like graph spaces by virtue of~\eqref{eq:LL*assume}. On any
nonempty open subset $K \subseteq \om,$ define the Hilbert spaces
$H(\mcL,K) = \{ u \in L^2(K)^m: \LL u \in L^2(K)^l \}$ and, likewise,
$H(\mcL^\ast,K) = \{ v \in L^2(K)^l: \LL^* v \in L^2(K)^m \}.$
(E.g., if we let $\mcL = \grad$, the canonical gradient operator, then $\mcL^\ast = -\div$ and $H(\mcL,K) = H(\grad,K) = H^1(K)$ and $H(\mcL^\ast,K) = \bmH(\div,K)$.)
To simplify notation, we abbreviate $H(\mcL) = H(\mcL,\om)$ and $H(\mcL^\ast) = H(\mcL^\ast,\om)$.
Also define linear operators $\bdryOp: H(\mcL) \to H(\mcL^\ast)^\prime$ and $\bdryOpAdj: H(\mcL^\ast) \to H(\mcL)^\prime$ such that
\begin{equation}
  \ip{\bdryOp u, v}_{H(\mcL^\ast)}
  = (\LL u, v)_\som - (u, \LL^* v)_\som,
  \qquad
  \ip{ \bdryOpAdj v, u}_{H(\mcL)}
  = (\LL^* v, u)_\som - (v , \LL u)_\som,
\end{equation}
for all $u\in H(\mcL)$ and $v \in H(\mcL^\ast)$.
Note that $\bdryOpAdj = -\bdryOp^\prime$, by these definitions.
These graph spaces are equipped with natural \emph{graph norms}:
\begin{equation}
  \|u\|_{H(\mcL)}^2 = \|\mcL u\|_{L^2}^2 + \|u\|_{L^2}^2,
  \qquad
  \|v\|_{H(\mcL^\ast)}^2 = \|\mcL^\ast v\|_{L^2}^2 + \|v\|_{L^2}^2.
\end{equation}
With these norms, notice that both $\bdryOp$ and $\bdryOpAdj$ are bounded.
Indeed, $|\ip{\bdryOp u, v}_{H(\mcL^\ast)}| \leq \|\LL u\|_{L^2} \| v\|_{L^2} + \|u\|_{L^2} \|\LL^* v\|_{L^2} \leq \|u\|_{H(\mcL)}\|v\|_{H(\mcL^\ast)}$.

Finally, we may incorporate homogeneous boundary conditions. 
Recall the definition of the left annihilator in~\eqref{eq:leftann}.
Define $H_0(\mcL)\subset H(\mcL)$ and $H_0(\mcL^\ast)\subset H(\mcL^\ast)$ to be two subspaces satisfying
\begin{equation}
  \label{eq:assumptionVDV}
  H_0(\mcL) = \lprp{\bdryOpAdj(H_0(\mcL^\ast))},
  \qquad 
  H_0(\mcL^\ast) = \lprp{\bdryOp(H_0(\mcL))}.
\end{equation}
\end{subequations}
Observe that~\cref{eq:assumptionVDV} does not uniquely characterize either $H_0(\mcL)$ or $H_0(\mcL^\ast)$.
These definitions permit many different so-called ``mixed'' homogeneous boundary conditions.

We will consider two boundary value problems: Given $f,g \in L^2$,
\begin{subequations}
\begin{align}
  \label{eq:bvp}
  &\text{ find $u \in H_0(\mcL)$ satisfying } 
  \LL u  = f, 
  \\
  \label{eq:bvp*}
  &\text{ find $v\in H_0(\mcL^\ast)$ satisfying } 
  \LL^* v  = g.
\end{align}
\end{subequations}
To derive a broken ``ultraweak formulation'' for~\cref{eq:bvp,eq:bvp*}, we focus on
the scenario where $\om$ is partitioned into a mesh $\oh$ of finitely
many open disjoint elements $K$ such that $\shoverbar{\om}$ is the union of closures
of all mesh elements $K$ in $\oh$.
For functions $u$ and $v$, we denote by $\LL_h u$ and $\LL^* _h v$ the functions obtained by applying $\LL$ and $\LL^* $ to $u|_K$ and $v|_K$, respectively, element by element, for all $K \in \oh$.
With this is mind, define the broken spaces 
\begin{align}
  \label{eq:WhWh*}
  H(\mcL_h)
  &=\prod_{K\in \oh} H(\mcL,K),
  &
  H(\mcL^\ast_h)
  &=\prod_{K\in\oh} H(\mcL^\ast,K)
  ,
\end{align}
which naturally conform to~\eqref{eq:broken_space}.

Clearly, $H(\mcL_h)$ and $H(\mcL_h^\ast)$ are inner product product spaces with corresponding graph norms.
The natural inner products on these spaces, induced by these graph norms, are defined
\begin{equation}
  (u,\tilde{u})_{H(\mcL_h)}
  =
  (\mcL_h u,\mcL_h \tilde{u})_\som + (u,\tilde{u})_\som
  ,
  \qquad
  (v,\tilde{v})_{H(\mcL^\ast_h)}
  =
  (\mcL^\ast_h v,\mcL^\ast_h \tilde{v})_\som + (v,\tilde{v})_\som
  ,
\label{eq:GraphInnerProducts}
\end{equation}
for all $u,\tilde{u} \in H(\mcL_h), v,\tilde{v}\in H(\mcL_h^\ast)$.
Now define the corresponding bounded linear operators $\hbdryOp: H(\mcL_h) \to H(\mcL^\ast_h)^\prime$ and
$\hbdryOpAdj: H(\mcL^\ast_h) \to H(\mcL_h)^\prime$ by
\begin{equation}
  \ip{\hbdryOp u, v}_{H(\mcL_h^\ast)}
  = (\LL_h u, v)_\som - (u, \LL_h^* v)_\som,
  \qquad
  \ip{ \hbdryOpAdj v, u}_{H(\mcL_h)}
  = (\LL_h^* v, u)_\som - (v , \LL_h u)_\som,
\end{equation}
for all $u \in H(\mcL_h), v\in H(\mcL_h^\ast)$.
From now on, when using the operators $\hbdryOp$ and $\hbdryOpAdj$, we will simply denote $\ip{\hbdryOp\cdot,\cdot}_{h} = \ip{\hbdryOp\cdot,\cdot}_{H(\mcL_h^\ast)}$ or, likewise, $\ip{\hbdryOpAdj\cdot,\cdot}_{h} = \ip{\hbdryOpAdj\cdot,\cdot}_{H(\mcL_h)}$, since the meaning can easily be deduced from context.
Finally, let
\begin{align}
  Q(\mcL_h)
    & = \{ p \in H(\mcL_h)^\prime: \text{ there is a $v \in H_0(\mcL^\ast)$ such that } 
      p = \hbdryOpAdj v\},
  \\
    Q(\mcL_h^\ast) & = \{ q \in H(\mcL_h^\ast)^\prime: \text{ there is a $u \in H_0(\mcL)$ such that } 
      q = \hbdryOp u\}.
\end{align}
These are Hilbert spaces when normed by the so-called \textit{minimum energy extension} (quotient) norm, i.e.,
$\| q\|_{Q(\mcL_h^\ast)} = \inf\{ \| u \|_{H(\mcL)}: u\in H(\mcL)$ satisfying $\hbdryOp u = q\}$.

Multiplying~\cref{eq:bvp} by a function $\nu \in H(\mcL_h^\ast)$ and
applying the definition of $\hbdryOp$, we get
$(u, \LL^* _h \nu )_\som + \ip{ \hbdryOp  u, \nu}_h = (f,\nu)_\som$ for all
$\nu$ in $H(\mcL_h^\ast)$.
Setting $\hbdryOp  u$ to $q$, a new unknown in $Q(\mcL_h^\ast)$, we
obtain the following {\em ultraweak formulation} with
$F(\nu) = (f, \nu)_\som$.
Given any $F \in H(\mcL_h^\ast)^\prime,$\, find $u\in L^2$
and $q \in Q(\mcL_h^\ast)$ such that
\begin{subequations}
\begin{equation}
  \label{eq:uwprob}
  (u, \LL^*_h \nu)_\som + \ip{ q,\nu }_h = F(\nu)
  \qquad 
  \forall\, \nu \in H(\mcL_h^\ast).
\end{equation}
Similarly proceeding with~\cref{eq:bvp*} and setting $F(\nu) = (g, \nu)_\som$, we obtain an ultraweak formulation of the dual problem: Given any $F \in H(\mcL_h)^\prime,$\, find $u \in L^2$ and $p \in Q(\mcL_h)$ such that
\begin{equation}
  \label{eq:uwprob*}
  (v, \LL_h \nu)_\som + \ip{ p, \nu }_h = F(\nu)
  \qquad \forall\, \nu \in H(\mcL_h).
\end{equation}
\end{subequations}
The next result shows that~\eqref{eq:uwprob} is uniquely solvable
whenever~\eqref{eq:bvp} is, as well as similar connection
between~\eqref{eq:uwprob*} and~\eqref{eq:bvp*}.

\begin{theorem}[Wellposedness of broken forms]
  \label{thm:wellposed-mesh*}
  Suppose~\eqref{eq:UWasm} holds. Then 
  \begin{enumerate}
  \item Whenever $\LL: H_0(\mcL) \to L^2$ is a bijection,
    problem~\eqref{eq:uwprob} is well-posed.  Moreover, if
    $F(\nu) = (f,\nu)_\som$ for some $f \in L^2,$ then the unique
    solution~$u$ of \eqref{eq:uwprob} is in~$H_0(\mcL)$,
    solves~\eqref{eq:bvp}, and satisfies $q= \hbdryOp u$.

  \item Whenever $\LL^*: H_0(\mcL^\ast) \to L^2$ is a bijection,
    problem~\eqref{eq:uwprob*} is well-posed.  Moreover, if
    $F(\nu) = (g,\nu)_\som$ for some $g \in L^2,$ then the unique
    solution~$v$ of \eqref{eq:uwprob*} is in~$H_0(\mcL^\ast)$,
    solves~\eqref{eq:bvp*}, and satisfies $ p= \hbdryOpAdj u$.
  \end{enumerate}
\end{theorem}
\begin{proof}
  The first statement is exactly the statement of
  \cite[Theorem~A.5]{demkowicz2016spacetime}. The second statement
  also follows from \cite[Theorem~A.5]{demkowicz2016spacetime} when 
  $\LL$ is replaced by $\LL^*$.
\end{proof}

Naturally, formulations~\cref{eq:uwprob,eq:uwprob*} also have adjoints.
For instance, the adjoint of the ultraweak formulation~\cref{eq:uwprob} is the following: Given any $G \in (L^2\times Q(\mcL_h^\ast))^\prime$, find $v \in H(\mcL_h^\ast)$ such that
\begin{subequations}
\begin{equation}
  \label{eq:uwprob_adj}
  (\mu, \LL^*_h v)_\som + \ip{ \rho,v }_h = G(\mu,\rho)
  \qquad \forall\, \mu \in L^2,\, \rho \in Q(\mcL_h^\ast).
\end{equation}
Similarly, the adjoint of~\cref{eq:uwprob*} is: Given any $G \in (L^2\times Q(\mcL_h))^\prime$, find $u \in H(\mcL_h)$ such that
\begin{equation}
  \label{eq:uwprob*_adj}
  (\mu, \LL_h u)_\som + \ip{ \rho,u }_h = G(\mu,\rho)
  \qquad \forall\, \mu \in L^2,\, \rho \in Q(\mcL_h).
\end{equation}
\end{subequations}
Under similar conditions to \cref{thm:wellposed-mesh*}, these variational formulations are also well-posed, as the following theorem demonstrates.

\begin{theorem}[Wellposedness of the adjoint problems]
  \label{thm:wellposed-adj}
  Suppose~\eqref{eq:UWasm} holds. Then 
  \begin{enumerate}
  \item Whenever $\LL: H_0(\mcL) \to L^2$ is a bijection,
    problem~\eqref{eq:uwprob_adj} is well-posed.  Moreover, if
    $G(\mu,\rho) = (g,\mu)_\som$ for some $g \in L^2$, then the unique
    solution~$v$ of \eqref{eq:uwprob_adj} is in~$H_0(\mcL^\ast)$ and
    solves~\eqref{eq:bvp*}.

  \item Whenever $\LL^*: H_0(\mcL^\ast) \to L^2$ is a bijection,
    problem~\eqref{eq:uwprob*_adj} is well-posed.  Moreover, if
    $G(\mu,\rho) = (f,\mu)_\som$ for some $f \in L^2$, then the unique
    solution~$u$ of \eqref{eq:uwprob*_adj} is in~$H_0(\mcL)$ and
    solves~\eqref{eq:bvp}.

  \end{enumerate}
\end{theorem}
\begin{proof}
  Both claims are closely related and follow similarly from \cref{thm:wellposed-mesh*}.
  Therefore, we prove only the first statement.

  Let the operator $\scB : L^2 \times Q(\mcL_h^\ast) \to H(\mcL_h^\ast)^\prime$ be defined
  $\ip{ \scB (\mu, \rho), \nu}_{H(\mcL_h^\ast)} = (\mu, \LL^*_h \nu)_\som + \ip{ \rho,\nu }_h$, for all $\nu \in H(\mcL_h)$ and $ (\mu, \rho) \in L^2 \times Q(\mcL_h^\ast)$.
  Recall that $F \in H(\mcL_h^\ast)^\prime = \range\,\scB$ in~\cref{eq:uwprob} was arbitrary.
  Therefore, as a consequence of the first statement in \cref{thm:wellposed-mesh*}, we conclude that $\scB$ is both bounded below (cf.~\cref{eq:Bbddbelow}) \emph{and} surjective.
  That is, $\scB$ is a bijection and, by the Closed Range Theorem, $(\kernel\,\scB^\prime)_\perp = \{0\}$.
  Hence, we conclude that~\cref{eq:uwprob_adj} is well-posed.

  Next, suppose $G((\mu, \rho)) = (g, \mu)_\som$.
  Then~\eqref{eq:uwprob_adj} yields
  \begin{gather}
    \label{eq:LLhu}
    (\mu, \LL_h^\ast v)_\som = (g, \mu)_\som, 
    \\
    \label{eq:rhotu}
    \ip{\rho, v}_h = 0,    
  \end{gather}
  for all $\mu \in L^2$ and $\rho \in Q(\mcL_h^\ast)$.  Equation~\eqref{eq:LLhu}
  yields $\LL_h^\ast v = g$ since $H(\mcL_h^\ast)$ is continuously embedded in $L^2$.
  It remains to show that $v$ is in $H_0(\mcL^\ast)$. Note that for all
  $\phi \in \mcD$, by the distributional definition of $\LL$ and
  the definition of $\hbdryOp$, 
  \[
  \begin{aligned}
  \ip{ \LL^\ast v, \phi}_{\D} 
  & =   (\LL \phi, v)_\som
  = (\LL_h^\ast v, \phi)_\som + \ip{\hbdryOp \phi, v}_h.   
  \end{aligned}
  \] 
  Since $\mcD$ is contained in $H_0(\mcL)$, $\hbdryOp \phi$ is in $Q(\mcL_h^\ast)$ and the last term vanishes by virtue of \eqref{eq:rhotu}.
  Moreover, since $\mcD$ is densely contained in $L^2$, this shows
  that $\LL^\ast v = \LL_h^\ast v = g$.
  Thus $v \in H(\mcL^\ast)$.
  Using \eqref{eq:rhotu} again, observe (cf. \cite[Lemma~A.3]{demkowicz2016spacetime}) that 
  \begin{equation}
    0 = \ip{\rho, v }_h = \ip{\hbdryOp \mu, v }_h = \ip{ \bdryOp \mu, v}_{H(\mcL^\ast)}
    \,,
  \end{equation}
  for all $\rho = \hbdryOp \mu \in Q(\mcL_h^\ast)$, where $\mu \in H_0(\mcL)$.
  Therefore, $v \in \lprp{\bdryOp(H_0(\mcL))}$.
  Finally, $v$ is in $H_0(\mcL^\ast)$ simply by~\eqref{eq:assumptionVDV}.
\end{proof}

Evidently, this result gives a class of examples where DPG* methods can be
formulated.
Letting $\scU = L^2 \times Q(\mcL_h)$ and $\scV = H_0(\mcL_h)$, define the
bilinear form $b: \scU \times \scV \to \RRR$ as follows
\begin{equation}
  \label{eq:b-form-defn}
  b((\mu, \rho), \nu) =   (\mu, \LL_h \nu)_\som + \ip{ \rho,\nu }_h
  \qquad \forall\,(\mu, \rho) \in \scU, \; \nu \in \scV.  
\end{equation}
We may now consider the DPG* formulations of~\cref{eq:bvp}.
Treatment of the dual problem~\cref{eq:bvp*} is similar.


\begin{theorem}[Ultraweak DPG* formulation of~\eqref{eq:bvp}]
  \label{thm:dpgstaruw}  
  Let $(\cdot,\cdot)_{\scV}$ be any inner product on $H(\mcL_h)$ equivalent to $(\cdot,\cdot)_{H(\mcL_h)}$.
  Suppose~\eqref{eq:UWasm} holds, $\LL^*: H_0(\mcL^\ast) \to L^2$ is a bijection,
  and $b$ is as in \eqref{eq:b-form-defn}.  Then, given a
  $G \in (L^2 \times Q(\mcL_h))^\prime$, the problem of finding a function $u \in H(\mcL_h)$
  satisfying
  \begin{equation}
    \label{eq:1uw}
    \left\{
    \begin{aligned}
      & (u, \nu)_\scV \;-\; b( (\lambda, \sigma), \nu) &&= 
      0  
      \quad && \forall\, \nu \in H(\mcL_h),       
      \\
      & b( (\mu, \rho), u) &&= 
      G((\mu, \rho))
      \quad && \forall\, (\mu, \rho) \in L^2 \times Q(\mcL_h),
    \end{aligned}
  \right.
  \end{equation}
  is well-posed.  Moreover, if $G((\mu, \rho)) = (f, \mu)_\som$ for
  some $f \in L^2$, then the unique solution $u$ is in $H_0(\mcL)$ and satisfies
  $\LL u = f$, i.e., $u$ solves \eqref{eq:bvp}.
\end{theorem}
\begin{proof}
  Define the operator
  $B : H(\mcL_h) \to (L^2 \times Q(\mcL_h))^\prime$, by
  $\ip{ B\nu, (\mu, \rho)}_{L^2 \times Q(\mcL_h)} = b( (\mu, \rho),
  \nu)$ for all $(\mu, \rho) \in L^2 \times Q(\mcL_h)$ and
  $\nu \in H(\mcL_h)$.  As in the proof of~\cref{thm:wellposed-adj},
  the operator $\scB = B^\prime$ is a bijection.  Hence,~\cref{eq:1uw}
  is a problem of form~\cref{eq:Mixed_general} (also~\cref{eq:dpg*A})
  with $\scB$ satisfying~\cref{eq:Bbddbelow} and we conclude
  that~\eqref{eq:1uw} has a unique solution (see,
  e.g.,~\cref{prop:identities}).  The remaining statements immediately
  follow from \cref{thm:wellposed-adj}.
\end{proof}

\begin{remark}
  Note that the notation in~\cref{thm:dpgstaruw} expresses the DPG* solution $u\in \scV$ and $(\lambda,\sigma) \in\scU$ in the form of~\cref{eq:dpg*A}.
  That is, the first solution component is denoted by the symbol $u$.
  From now on, in order to more closely follow the abstract notation used in~\cref{eq:Mixed_general}, we will only use the symbol $v$ for the $\scV$-solution in all DPG* problems.
\end{remark}

\begin{remark}
  
  A very general broken space theory applicable of both DPG and DPG* methods has been established in the literature.
  This theory encompasses more traditional weak formulations and has been applied to a wide variety different boundary value problems \cite{demkowicz2013primal,broersen2015petrov,Carstensen15,Keith16,fuentes2016coupled}.
  For brevity, we will not expand on the intricate details here, but simply act to remind the reader that ultraweak variational formulations are not a prerequisite for any DPG-type method coming from~\cref{eq:Mixed-General-form-discrete}.

\end{remark}

\begin{example}[Poisson equation]
  \label{eg:poisson}
  In this example, which resurfaces throughout the document, $\bmv=(\bmp, v)$ will denote the DPG* solution variable.
  Similarly, $\blambda = (\bzeta, \lambda, \hat{\zeta}_n, \hat\lambda)$ will come to denote the associated Lagrange multiplier (see \cref{sub:Application_to_the_Poisson_example}).

  On a bounded open set $\om \subseteq \RRR^d$ with connected Lipschitz
  boundary, set $m=d+1$ and 
  \begin{equation}
    \LL (\bmp, v) = ( \bmp - \grad v, -\div \bmp)
    \,,
  \label{eq:DefinitionOfLPoisson}
  \end{equation}
  where $\bmp: \om \to \RRR^d$ represents flux variable and
  $v:\om \to \RRR$ represents solution variable.  Note that the
  equation $\LL(\bmp, v) = (\vec{0}, f)$, after elimination of $\bmp$, results in
  the well-known Poisson equation $-\Delta v = f$.

  We want to write
  out the DPG* formulation studied in \cref{thm:dpgstaruw} for
  this $\LL$. Begin by observing that $\LL^*$ given by~\eqref{eq:L*defn}
  can be written as
  \begin{equation}
    \label{eq:L*_Poisson}
    \LL^* (\bsigma, \mu)  = ( \bsigma + \grad \mu, \div \bsigma).
  \end{equation}
  Obviously~\eqref{eq:LL*assume} is satisfied. By the triangle
  inequality, we immediately see that both $H(\mcL)$ and $H(\mcL^\ast)$ in this case
  coincide with $\bmH(\div,\om) \times H^1(\om)$. By integration by
  parts, 
  \begin{align}
    \label{eq:DDtPoisson}
    \ip{\bdryOpAdj(\bsigma,\mu),(\bmp,v)}_h
    = \ip{ \bsigma\cdot \bmn, v}_{H^{1/2}(\partial \om)}
    + \ip{\bmp\cdot \bmn, \mu}_{H^{1/2}(\partial \om)},
  \end{align}
  where $\bmn$ denotes the unit outward normal on $\d\om$.
  Put 
  \begin{equation}
    \label{eq:V_Vt_Poisson}
    H_0(\mcL^\ast) = \bmH(\div, \om) \times H^1_0(\om).
  \end{equation}
  This choice corresponds to the Dirichlet problem, $v=0$ on $\partial\Omega$, as we shall see.
  From~\eqref{eq:DDtPoisson}, it is immediate
  that~\eqref{eq:assumptionVDV} holds.   
  Along the lines of~\eqref{eq:DDtPoisson}, we also have 
  \begin{align}
    \label{eq:DDtPoisson-h}
    \ip{\hbdryOpAdj(\bsigma,\mu),(\bmp,v)}_h 
    & = 
      \sum_{K \in \oh} 
      \bigg[\ip{ \bsigma\cdot \bmn, v}_{H^{1/2}(\partial K)}
      + \ip{\bmp\cdot \bmn, \mu}_{H^{1/2}(\partial K)}
      \bigg].
  \end{align}
  The range of $\hbdryOpAdj|_{H_0(\mcL^\ast)}$ is $Q(\mcL_h)$. In this
  example, this can characterized using standard trace operators.  The
  domain-dependent trace operators $\tr^\selement u = u|_{\bdry K}$
  and
  $\tr_{n}^\selement \vec{\sigma} = \vec{\sigma}|_{\bdry K}\cdot\bmn$
  for smooth functions are well-known to be continuously extendable to
  bounded linear maps
  $\tr^{\selement}:H^1(K)\to H^{\onehalf}(\bdry K)$ and
  $\tr^{\selement}_n:H(\div,K)\to H^{\minusonehalf}(\bdry K)$.  Let
  $\tr=\prod_{K\in\mesh}\tr^\selement$ and
  $\tr_{n}=\prod_{K\in\mesh}\tr_{n}^\selement$.  Then set 
  that
  \begin{equation}
    \label{eq:InterfaceSpaces}
    H^{\minusonehalf}(\partial\oh) = \tr_n( H(\div, \om)), 
    \qquad
    H^{\onehalf}_0(\partial\oh) = \tr( H_0^1(\om)).
  \end{equation}
  Clearly,
  $ Q(\mcL_h) = H^{\minusonehalf}(\partial\oh)\times
  H^{\onehalf}_0(\partial\oh)$.  
  
  Applying the abstract setting with these definitions, the DPG*
  bilinear form in~\eqref{eq:b-form-defn}, for this example becomes
  \begin{equation}
    \label{eq:DPG*Poissonform}
    \begin{aligned}
    b( (\bsigma, \mu, \hat{\sigma}_n, \hat{\mu}), (\btau, \nu)) 
    & = \sum_{K \in \oh}
    \bigg[
      (\bsigma, \btau - \grad \nu)_{\selement} - (\mu, \div \btau)_{\selement}
    \bigg]
    \\
    & +\sum_{K \in \oh} 
    \bigg[ \ip{ \btau\cdot \bmn, \hat \mu}_{H^{1/2}(\d K)}
    + \ip{ \hat \sigma_n,  \nu}_{H^{1/2}(\d K)}
    \bigg].
    \end{aligned}
  \end{equation}
  Here, $(\cdot, \cdot)_{\selement}$ denotes the inner product in $L^2(K)$ or
  its Cartesian products, $\bsigma \in L^2(\om)^d$,
  $\mu \in L^2(\om)$, $(\hat\sigma_n, \hat\mu)$ is in the space $Q(\mcL_h)$
  defined above, and the solution variable $(\bsigma, \mu)$ is in the broken space
  $H(\mcL_h) = \bmH(\div,\oh) \times H^1(\oh),$ where
  \begin{equation}
    H(\div, \oh) = \prod_{K\in \oh} \bmH(\div, K), \qquad 
    H^1(\oh) = \prod_{K \in \oh} H^1(K).
  \label{eq:H1andHdivBroken}
  \end{equation}
  Finally, the bijectivity of $\LL: H_0(\mcL) \to L^2$ can be proved by
  standard techniques (see e.g.~\cite{demkowicz2011analysis}). Hence
  \cref{thm:dpgstaruw} yields that this DPG* formulation is
  well posed.
  
  We shall revisit this example later. In order to shorten the
  notation for later discussions, we shall denote
  \[
    (\grad_h \mu, \vec{\sigma})_\som + (\mu, \div_h \vec{\sigma})_\som
  \]
  by $\ip{\mu, \vec\sigma \cdot \vec n}_h$ or
  $\ip{\vec\sigma \cdot \vec n, \mu}_h$ without explicitly indicating
  the application of the trace maps $\tr_n$ and $\tr$. Accordingly,
  the bilinear form in~\eqref{eq:DPG*Poissonform} may be abbreviated  as
  $ b( (\bsigma, \mu, \hat{\sigma}_n, \hat{\mu}), (\btau, \nu)) =
  (\bsigma, \btau - \grad \nu)_{\om} - (\mu, \div \btau)_{\om} + \ip{
    \btau\cdot \bmn, \hat \mu}_h + \ip{ \hat \sigma_n, \nu}_h.$ 
  \hfill$\Box$
\end{example}

\subsection{Related methods} 
\label{sub:related_methods}

Let $\scV = H(\mcL_h^\ast)$.
For any $F\in H(\mcL_h^\ast)^\prime$, the ultraweak DPG formulation defined by \cref{eq:1uw} can be restated as the following system of variational equations:
\begin{subequations}
\begin{equation}
    \left\{
    \begin{alignedat}{5}
      & (\veps, \nu)_\scV \;+\; (u, \LL_h^\ast \nu)_\som + \ip{ p,\nu }_h &&=
      F(\nu)\,,
      \quad && \forall\, \nu \in H(\mcL_h^\ast),       
      \\
      & (\mu, \LL_h^\ast \veps)_\som &&=
      0\,,
      \quad && \forall\, \mu \in L^2,
      \\
      & \ip{ \rho,\veps }_h &&=
      0\,,
      \quad && \forall\, \rho \in Q(\mcL_h^\ast).
    \end{alignedat}
  \right.
\label{eq:DPGformulationExpanded}
\end{equation}
Likewise, letting $\scV = H(\mcL_h)$, an ultraweak DPG formulation corresponding to~\cref{eq:1uw} may be defined for any $G = G_\om \times G_h\in(L^2\times Q(\mcL_h))^\prime$:
\begin{equation}
    \left\{
    \begin{alignedat}{5}
      & (v, \nu)_\scV \;-\; (\lambda, \LL_h \nu)_\som - \ip{ \sigma,\nu }_h &&=
      0\,,
      \quad && \forall\, \nu \in H(\mcL_h),       
      \\
      & (\mu, \LL_h v)_\som &&=
      G_\om(\mu)\,,
      \quad && \forall\, \mu \in L^2,
      \\
      & \ip{ \rho,v }_h &&=
      G_h(\rho)\,,
      \quad && \forall\, \rho \in Q(\mcL_h).
    \end{alignedat}
  \right.
\label{eq:DPG*formulationExpanded}
\end{equation}
\end{subequations}
Both of the formulations defined above relate to the primal problem
\cref{eq:bvp} with $u=v$.  Clearly, the role of $\mcL_h$ and
$\mcL_h^\ast$ can be interchanged if a solution of the dual problem
\cref{eq:bvp*} is of interest.

The link between DPG and least-squares methods is well established in
the literature (see e.g. \cite{Keith2017Discrete}).  DPG* methods, as
it turns out, can be readily identified with the category of so-called
$\mcL\mcL^\ast$ methods \cite{cai2001first}.  In this subsection, we
briefly illustrate this and a couple of other notable relationships in
the context of the mixed problems introduced in
\cref{sub:operator_equations}.

\subsubsection{Least-squares methods} 
\label{ssub:_mclmcl_ast_methods}

Let $\scV = L^2$ and $\scU = H_0(\mcL)$.
It is well-known that least-squares finite element methods \cite{bochev2009least} follow from the following saddle-point formulation (cf.~\cref{eq:dpgA,eq:DPGformulationExpanded}):
\begin{equation}
  \left\{
    \begin{alignedat}{5}
      &(\veps, \nu)_\som  + (\mcL\fku,\nu)_\som &&= F(\nu)\,,\quad &&\forall\, \nu\in L^2\,, \\
      &(\mcL\mu,\veps)_\som &&= 0\,,\quad &&\forall\, \mu\in H_0(\mcL)\,.
    \end{alignedat}
  \right.
\label{eq:LeastSquares}
\end{equation}
This may be identified with a mixed problem akin to~\cref{eq:dpgA} using the strong formulation of~\cref{eq:bvp}, rather than the ultraweak formulation, as in~\cref{eq:DPGformulationExpanded}.
Indeed, let $\Riesz_{L^2}$ be the $L^2$ Riesz operator appearing in each term in~\cref{eq:LeastSquares} and recall identity~\cref{eq:normal_equation}, where $(\scB\mu)(\cdot) = (\mcL\mu,\cdot)_\som$.
Then observe that $\scB^\prime\Riesz_\scV\inv\scB = (\Riesz_{L^2}\mcL)^\prime\Riesz_{L^2}\inv(\Riesz_{L^2}\mcL) = \mcL^\prime\Riesz_{L^2}\mcL$ and $\scB^\prime\Riesz_\scV\inv F = \mcL^\prime F$.
That is,
\begin{equation}
  \langle \scB^\prime\Riesz_\scV\inv\scB\fku,\mu\rangle_\scU
  =
  \langle\scB^\prime\Riesz_\scV\inv F,\mu\rangle_\scU
  \qquad
  \iff
  \qquad
  (\mcL\fku,\mcL\mu)_\som
  =
  F(\mcL\mu)
  \,.
\end{equation}
In the case $F(\cdot) = (f,\cdot)_\som$, observe that $F(\mcL\nu) = (f,\mcL\nu)_\som$.
Therefore, the variational equation above can be readily identified with the first-order optimality condition for the functional $\mcJ:\fku\mapsto \|\mcL\fku-f\|_{L^2}^2$, $\partial_\fku\mcJ = 0$.

\subsubsection{\texorpdfstring{$\mcL\mcL^\ast$}{LL*} methods} 

Let $\scV = L^2$ and $\scU = H_0(\mcL^\ast)$.
Contrary to \cref{eq:LeastSquares}, so-called $\mcL\mcL^\ast$ methods \cite{cai2001first} relate to the following system (cf.~\cref{eq:dpg*A,eq:DPG*formulationExpanded}):
\begin{equation}
  \left\{
    \begin{alignedat}{5}
      &(\fkv,\nu)_\som  - (\mcL^\ast\lambda,\nu)_\som &&= 0\,,\quad &&\forall\, \nu\in L^2\,, \\
      &(\mcL^\ast\mu,\fkv)_\som &&= G(\mu)\,,\quad &&\forall\, \mu\in H_0(\mcL^\ast)\,.
    \end{alignedat}
  \right.
\label{eq:LLstar}
\end{equation}
Likewise, consider \cref{eq:normal_equation2}, where $(\scB\mu)(\cdot) = (\mcL^\ast\mu,\cdot)_\som$ and $G(\cdot) = (f,\cdot)_\som$.
In this case, we see that $\mcL\mcL^\ast$ formulations may again be identified with \cref{eq:bvp}, in this case using a saddle-point expression akin to~\cref{eq:dpg*A}.
Indeed, observe that
\begin{equation}
  \langle \scB^\prime\Riesz_\scV\inv\scB\lambda,\mu\rangle_\scU
  =
  \langle G,\mu\rangle_\scU
  \qquad
  \iff
  \qquad
  (\mcL^\ast\lambda,\mcL^\ast\mu)_\som
  =
  (f,\mu)_\som
  \,.
\end{equation}
The variational equation above indicates, in a weak sense, that $\mcL\mcL^\ast\lambda = f$.
Recalling that the solution is determined by the transformation $\fkv = \Riesz_\scV\inv\scB\lambda = \mcL^\ast\lambda$, we have $\mcL\fkv = f$ weakly, as well.

\subsubsection{Weakly conforming least-squares methods} 
\label{ssub:weakly_conforming_least_squares_methods}

A weakly conforming least squares method \cite{Ernesti2017} for the primal problem \eqref{eq:bvp} seeks a minimizer of the least squares functional
\begin{equation}
  \fkw\mapsto\Vert \mcL \fkw - f \Vert^2_{L^2}\, ,
\end{equation}
under the conformity constraint
\begin{equation}
  \langle \fkw, \rho \rangle_h = 0\,, \quad \forall\, \rho \in Q(\mcL_h)
  \,.
\end{equation}
Here, of course, the operator $\mcL$ is understood element-wise so we may saliently replace it by $\mcL_h$.
This leads to the following saddle-point problem for the two solution components $\fkw$ and $\sigma$:
\begin{equation}
  \left\{
    \begin{alignedat}{5}
      &(\mcL_h\fkw, \mcL_h \nu)_\som  + \langle \sigma, \nu \rangle_h &&= (f, \mcL_h \nu)_\som\,,\quad &&\forall\, \nu \in H(\mcL_h)\,, \\
      &\langle \fkw, \rho \rangle_h &&= 0\,,\quad &&\forall\, \rho \in Q(\mcL_h)\,.
    \end{alignedat}
  \right.
\label{eq:WeaklyConformingLeastSquares}
\end{equation}
If we use an ultraweak DPG* formulation \cref{eq:DPG*formulationExpanded} with its corresponding graph inner product \cref{eq:GraphInnerProducts}, scaled by an arbitrary constant $\alpha>0$, we arrive at
\begin{equation}
  \left\{
    \begin{alignedat}{5}
      &(\mcL_h  v,\mcL_h \nu)_\som + \alpha ( v,\nu)_\som  - (\lambda,\mcL_h\nu)_\som  - \langle \sigma,\nu \rangle_h &&= 0\,,\quad && \forall\, \nu \in H(\mcL_h)\,, \\
      &(\mu, \mcL_h v)_\som &&= ( f, \mu )_\som\,,\quad && \forall\, \mu \in L^2\,,  \\
      &\langle  v, \rho\rangle_h &&= 0\,,\quad && \forall\, \rho \in Q(\mcL_h)\,.
    \end{alignedat}
  \right.
\label{eq:DPG*ultraweak}
\end{equation}
From the second equation in \cref{eq:DPG*ultraweak}, observe that $f = \mcL_h v$.
Therefore, the first equation can be rewritten as
\begin{equation}
  (f-\lambda,\mcL_h \nu)_\som + \alpha ( v,\nu)_\som  = \langle \sigma,\nu \rangle_h \,,\quad \forall\, \nu \in H(\mcL_h)\,.
\end{equation}
Testing only with $\nu\in H(\mcL)$, so that the term $\langle \sigma,\nu \rangle_h$ vanishes, it can now be seen that $\lambda \to f$ as $\alpha \to 0$. 
Consequently, this particular DPG* formulation can be viewed as a regularization of the weakly conforming least-squares formulation \cref{eq:WeaklyConformingLeastSquares}.

\subsection{Solving the primal and dual problems simultaneously} 
\label{sub:jointly_solving_for_the_primal_and_dual_solutions}

In \cref{eq:Mixed_general}, we may hypothetically consider any $F\in\scV^\prime$ and $G\in\scU^\prime$ we wish:
\begin{equation}
  \left\{
    \begin{alignedat}{3}
      \Riesz_\scV &\fkv + \scB w &&= F \,, \\
      \scB^\prime &\fkv &&= G \, .
    \end{alignedat}
  \right.
\label{eq:Mixed_general2}
\end{equation}
Let $\scB$ to be an isomorphism and define $F = \Riesz_\scV(\scB^\prime)\inv G + \ell$, for some fixed $\ell \in (\kernel\,\scB^\prime)^\perp = \scV^\prime$.\footnote{In the case of an injective but \textit{not} surjective $\scB$, consider $F = \scB(\scB^\prime\Riesz_\scV\inv\scB)\inv G + \ell \in (\kernel\,\scB^\prime)^\perp\subsetneq \scV^\prime$.}
Noting that $\fkv = (\scB^\prime)\inv G$, by the second equation in \cref{eq:Mixed_general2}, it is readily seen that $\scB w = \ell$.
Therefore, with this choice of loads, $w=u$ solves the primal problem \cref{eq:Au=f} and $\fkv$ solves the dual problem \cref{eq:adjoint_problem}, simultaneously.

Introducing the load $F$, as proposed above, involves the inversion the linear operator $\scB^\prime$.
In practice, this is usually not feasible and, therefore, precludes the construction of any such load in most circumstances.
Nevertheless, consider the following system of equations:
\begin{equation}
  \left\{
    \begin{alignedat}{5}
      &(\mcL^\ast  v,\mcL^\ast \nu)_\som + ( w,\mcL^\ast\nu)_\som  &&= (g,\mcL^\ast \nu)_\som + (f,\nu)_\som\,,\quad && \forall\, \nu \in H(\mcL^\ast)\,, \\
      &(\mcL^\ast v, \mu)_\som &&= (g,\mu )_\som\,,\quad && \forall\, \mu \in L^2\,.
    \end{alignedat}
  \right.
\label{eq:PrimalDualDPG}
\end{equation}
This corresponds to a system like \cref{eq:Mixed_general2} with $\scV = H(\mcL^\ast)$ and $\scU = L^2$.
In \cref{eq:PrimalDualDPG}, $ v$ clearly satisfies $(\mcL^\ast v, \mu)_\som = ( g, \mu )_\som$.
That is, $ v$ solves the dual problem~\cref{eq:bvp*}, $\mcL^\ast v = g$, in a strong sense.
Substituting $\mu = \mcL^\ast\nu$ into \cref{eq:PrimalDualDPG} and canceling terms in the first equation, we immediately find that $( w,\mcL^\ast\nu)_\som  = (f,\nu)_\som$.
That is, $w=u$ solves the primal problem~\cref{eq:bvp}, $\mcL u = f$, in the ultraweak sense.

To avoid solving the mixed problem for both $ v$ and $w$ at the same time, upon discretization, broken test spaces spaces can be used.
In this setting, we must consider the following related system with solution $( w,\sigma)\in \scU = L^2\times Q(\mcL^\ast_h)$ and $ v\in \scV = H(\mcL^\ast_h)$:
\begin{equation}
  \left\{
    \begin{alignedat}{5}
      &(\mcL_h^\ast  v,\mcL_h^\ast \nu)_\som + \alpha ( v,\nu)_\som  + ( w,\mcL_h^\ast\nu)_\som  + \langle \sigma,\nu \rangle_h &&= (g,\mcL_h^\ast \nu)_\som + (f,\nu)_\som\,,\quad\!\! && \forall\, \nu \in H(\mcL^\ast_h), \\
      &(\mcL_h^\ast v, \mu)_\som &&= (g,\mu )_\som,\quad && \forall\, \mu \in L^2,  \\
      &\langle  v, \rho\rangle_h &&= 0,\quad && \forall\, \rho \in Q(\mcL^\ast_h).
    \end{alignedat}
  \right.
\label{eq:PrimalDualDPGbroken}
\end{equation}
The consequent manipulations are inspired by \cite[Lemma~7]{fuhrer2017superconvergent}.
First, notice that the last two equations in \cref{eq:PrimalDualDPGbroken} uniquely determine $ v$.
Therefore, after testing the middle equation with $\mu = \mcL_h^\ast\nu$, observe that the first equation can be rewritten
\begin{equation}
  ( w,\mcL_h^\ast\nu)_\som  + \langle \sigma,\nu \rangle_h
  =
  (f,\nu)_\som - \alpha ( v,\nu)
  \,.
\end{equation}
By linearity, $( w,\sigma)_\som = ( u,q)_\som + \alpha (e,r)$, where $( u,q)\in\scU$ solves the ultraweak primal problem $( u,\mcL_h^\ast\nu)_\som  + \langle q,\nu \rangle_h = (f,\nu)_\som$ (cf. \cref{eq:bvp}) and $(e,r) = (e(v),r(v))$ is a \emph{pollution term} defined by the equation $(e,\mcL_h^\ast\nu)_\som  + \langle r,\nu \rangle_h = -( v,\nu)$.
Clearly, $ w\to u$ as $\alpha\to 0$.


\section{A priori error analysis} 
\label{sec:a_priori_error_analysis}

\subsection{General results}

Having explained the connections between the DPG* method and the mixed
formulation~\eqref{eq:Mixed_general}, it should not be a surprise that
its error analysis reduces to standard mixed theory. To state the
result, let $v \in \scV$ and $\lambda \in \scU$ satisfy 
\begin{equation}
    \left\{
    \begin{alignedat}{5}
      &(v, \nu)_\scV - b(\lambda, \nu) &&= 0\,,\quad &&  \nu \in \scV,
      \\
      &b(\mu, v) &&= G(\mu)\,,\quad &&  \mu\in\scU.
    \end{alignedat}
  \right.
\label{eq:DPGstarExact}
\end{equation}
The DPG* approximation  $(v_h, \lambda_h) \in \scV_h \times \scU_h$
satisfies
\begin{equation}
    \left\{
    \begin{alignedat}{5}
      &(v_h, \nu)_\scV - b(\lambda_h, \nu) &&= 0\,,\quad &&  \nu  \in\scV_h,
      \\
      &b(\mu, v_h) &&= G(\mu)\,,\quad &&  \mu\in\scU_h.
    \end{alignedat}
  \right.
\label{eq:DPGstarApprox}
\end{equation}
We assume that $b(\cdot, \cdot) : \scV \times \scU \to \RRR$ is
generated by a bounded linear operator $\scB$ (as in~\eqref{eq:b-form-B}) that
satisfies~\eqref{eq:Bbddbelow}.  The only further assumption we need
for error analysis is the existence of a Fortin operator. Namely,
we assume that there is a continuous linear operator
$\vpi_h: \scV \to \scV_h$ such that
\begin{equation}
  \label{eq:Fortin}
  b( \mu, \nu - \vpi_h \nu) = 0, \qquad \mu \in \scU_h, \; \nu \in \scV.  
\end{equation}
Under these assumptions, the standard theory of mixed methods
\cite{Boffi2013} yields the following {\it a priori} estimate.

\begin{theorem}
  \label{thm:apriori}
  Suppose~\eqref{eq:Bbddbelow} and~\eqref{eq:Fortin} hold.  Then there
  is a constant $C$ such that the complete DPG* solution $(v,\lambda)\in\scV\times\scU$ satisfies the error
  estimate
  \[
    \| v - v_h \|_\scV + \| \lambda - \lambda_h \|_\scU
    \le 
    C 
    \bigg[
    \inf_{\nu \in \scV_h} 
    \| v - \nu \|_\scV +
    \inf_{\mu \in \scU_h} \| \lambda - \mu \|_\scU
    \bigg].
  \]  
\end{theorem}

At times, it is possible to get an improvement using the Aubin-Nitsche
duality argument. Suppose $F$ is a functional in $(\kernel\,\scB^\prime)^\perp$ and we
are interested in bounding $F(v-v_h)$, a functional of the error
$v-v_h$.  Consider $\veps \in \scV$ and $u \in \scU$
satisfying
\begin{subequations} 
  \label{eq:DPGexact}
  \begin{align}
    \label{eq:DPGexact-1}
    (\veps, \nu)_\scV + b(u, \nu) 
    &= F( \nu),
      \\
    \label{eq:DPGexact-2}
    b(\mu, \veps)
    &= 0,    
  \end{align}
\end{subequations}
for all $\nu \in \scV$, $\mu\in\scU$.
To conduct the duality argument, we suppose that there is a positive
$c_0(h) $ that goes to 0 as $h\to 0$ satisfying
\begin{equation}
  \label{eq:reg}
  \inf_{\mu \in \scU_h} \| u - \mu \|_\scU
    \le\; c_0(h)\, \| F \|_{\scV^\prime}.
\end{equation}
This usually holds when the solution
of~\cref{eq:DPGexact} has sufficient
regularity.

\begin{theorem}
  \label{thm:duality}
  Suppose~\eqref{eq:reg} holds in addition to the assumptions of \cref{thm:apriori}.
  Then there exists a positive function $c_0(h)$, which goes to 0 as $h\to 0$, such that the error in the DPG* solution component $v_h$ satisfies
  \[
    F(v - v_h ) 
    \le \; c_0(h)\,\|\scB\|\, \| F\|_{\scV^\prime}
    \bigg[
    \inf_{\nu \in \scV_h} 
    \| v - \nu \|_\scV^2 +
    \inf_{\mu \in \scU_h} \| \lambda - \mu \|_\scU^2
    \bigg]^{1/2}.
  \]  
\end{theorem}
\begin{proof}
  By \cref{prop:identities}, $\veps = 0$ since $F\in(\kernel\,\scB^\prime)^\perp$.
  Put $\nu = v-v_h$ in~\eqref{eq:DPGexact-1}.
  Then for any $\mu \in \scU_h$, we have
  \[
    \begin{aligned}
      F( v - v_h )
      & 
      = ( \veps, v-v_h)_\scV + b( u, v-v_h)
      \quad
      && \text{by \eqref{eq:DPGexact-1},}
      \\
      &
      = b( u, v-v_h) 
      && \text{since $\veps = 0$,}
      \\
      &
      = b( u -\mu, v-v_h) 
      && \text{by~\eqref{eq:DPGstarExact} and~\eqref{eq:DPGstarApprox},}
      \\
      & \le  c_0(h) \|\scB\| \| F \|_{\scV^\prime}
      \| v - v_h \|_\scV
      && \text{by~\eqref{eq:reg}}.
    \end{aligned}
  \]
  The proof is completed by applying \cref{thm:apriori}.
\end{proof}

It is interesting to note that the duality argument for the DPG* method
uses a DPG formulation: the system~\eqref{eq:DPGexact} is
clearly a DPG formulation. Vice versa, the duality argument for DPG
methods uses DPG* formulations, as can be seen from the duality
arguments in~\cite{BoumaGopalHarb14,fuhrer2017superconvergence,
  fuhrer2017superconvergent}.  Even though these references did not
use the name ``DPG*,'' one can see DPG* formulations at work within
their proofs.

At this point, an essential difficulty in DPG* methods (that was not
present in DPG methods) becomes clear.  Consider using a DPG* form
$b(\cdot, \cdot)$ given by \cref{thm:dpgstaruw} for solving the 
primal problem $\LL v = f$.
Then the error in $v_h$ computed by the
DPG* method not only depends on the regularity of the solution $v$,
but also on the regularity of an extraneous Lagrange multiplier
$\lambda$. This is evident from the best approximation error bounds
appearing in \cref{thm:apriori,thm:duality}. The
following example will clarify this observation further.

\subsection{Application to the Poisson example}
\label{sub:Application_to_the_Poisson_example}

Given $f \in L^2(\om)$, consider approximating the Dirichlet solution $v$
\begin{equation}
  \label{eq:Poisson}
-\Delta v = f \quad \text {in } \om,
\qquad 
v=0 \quad \text{ on } \d\om,
\end{equation}
by the DPG* method. We follow the setting of \cref{eg:poisson}.
Accordingly, we set
\[
\scU = L^2(\om)^d \times L^2(\om) \times H^{\minusonehalf}(\partial\oh)\times H^{\onehalf}_0(\partial\oh),
\qquad \scV = H(\div,\oh) \times H^1(\oh),
\]
where $H^{\minusonehalf}(\partial\oh)$ and $H^{\onehalf}_0(\partial\oh)$ are defined in \cref{eq:InterfaceSpaces}.
This DPG* formulation~\eqref{eq:DPGstarExact} of~\cref{eq:Poisson} characterizes two variables,
\[
\bmv = (\bmp, v) \in \scV, \qquad \blambda = (\bzeta, \lambda, \hat{\zeta}_n,
\hat\lambda) \in \scU
,
\]
satisfying 
\begin{subequations}
  \label{eq:dpgstar-poisson}
  \begin{align}
    \label{eq:dpgstar-poisson-a}
    ((\bmp,v), (\btau, \nu))_\scV
    - b((\bzeta, \lambda, \hat{\zeta}_n,\hat\lambda), (\btau, \nu)) 
    & = 0,
    \\
    \label{eq:dpgstar-poisson-b}
    b((\bsigma, \mu, \hat{\sigma}_n,\hat\mu), (\bmp, v)) 
    & = (f, \mu)_\som,
  \end{align}
\end{subequations}
for all $\bnu = (\btau, \nu)$ in $\scV$ and all
$\bmmu = (\bsigma, \mu, \hat{\sigma}_n,\hat\mu)$ in $\scU$.
Here, $b(\cdot, \cdot)$ as given by~\eqref{eq:DPG*Poissonform} and, as
before, $(\cdot,\cdot)_\som$ denotes the inner product in $L^2(\om)$
(or its Cartesian products).

By \cref{thm:dpgstaruw}, $\scB$ is a bijection, so
obviously~\eqref{eq:Bbddbelow} holds.  Let $\oh$ be a shape regular
mesh of simplices and let $P_p(K)$ denote the space of polynomials of
degree at most $p$ on a simplex $K$.  Define
$P_p(\partial K) = \{ \mu : \mu|_E \in P(E)\;  \text{ for all 
  codimension-one sub-simplices } E \text{ of } K \}$ and
$\tilde{P}_p(\partial K) = P_p(\partial K) \cap C^0(\partial K)$,
where $C^0(D)$ denotes the set of all continuous functions on a domain
$D$.  Set
\begin{gather}
  \label{eq:Ppoh}
  P_p(\oh) = \prod_{K \in \oh} P_p(K), \qquad 
  P_p(\d\oh) =  \prod_{K \in \oh} P_p(\d K),
  \\
  \label{eq:tr-fl}
  \tilde{P}_p(\d\oh) = \tr (P_p(\oh) \cap H_0^1(\om)),
  \qquad
  \hat{P}_p(\d\oh) = \tr_n (P_p(\oh)^d \cap H(\div, \om)).
\end{gather}
Clearly $\tilde{P}_p(\d\oh)$ is a subspace of $H_0^{\onehalf}(\d\oh)$
which consists of single-valued functions on mesh interfaces.  Its also
obvious that the space $\hat{P}_p(\d\oh)$ is a subspace of
$H^{-\onehalf}(\d\oh)$. Every $\hat{\sigma}_n \in \hat{P}_p(\d\oh)$
has a corresponding $\hat\sigma$ in $P_p(\oh) \cap H (\div, \om)$ such
that $\hat{\sigma}_n = \hat {\sigma} \cdot n.$ In particular, the
value of $\hat{\sigma}_n|_{\d K^+}$ has the opposite sign of the value
of $\hat{\sigma}_n|_{\d K^-}$ at every point of a mesh interface
$E = \d K^+ \cap \d K^-$ (so $\hat{\sigma}_n$ is not single-valued on
$E$).  To remember this orientation dependence, we often write
a $\hat{\sigma}_n$ in $\hat{P}(\oh)$ as $\hat{\sigma} \cdot \vec n$.
A Fortin operator satisfying~\eqref{eq:Fortin} for the case
\begin{align}
  \scU_h
  &=
    \{ (\bsigma, \mu,\hat{\sigma}_n,\hat \mu) \in \scU :
    \bsigma \in P_p(\oh)^d, \;
    \mu \in P_p(\oh),\;
    \hat{\sigma}_n \in \hat{P}_p(\d\oh), \;
    \hat{\mu} \in \tilde{P}_{p+1}(\d\oh) \}
   \\
  \scV_h
  &=
    \{(\btau,\nu) \in \scV :
    \btau
    \in P_{p+d}(\oh)^d,\, \nu\in P_{p+d}(\oh) \}  .
\end{align}
was constructed in~\cite{gopalakrishnan2014analysis}.

To understand the practical convergence rates, we must understand the
regularity of $\blambda$. One way to do this is to write down the
boundary value problem that $\blambda$ satisfies, as done
in~\cite{BoumaGopalHarb14,fuhrer2017superconvergence}. An alternate
technique can be seen in~\cite{fuhrer2017superconvergent}, which
directly manipulates the variational
equation~\eqref{eq:dpgstar-poisson-a} using the information
in~\eqref{eq:dpgstar-poisson-b}.  We follow the latter approach
in the next proof.

\begin{proposition}
  \label{prop:LM}
  The solution components $\bzeta, \lambda, \hat{\zeta}_n,\hat\lambda$ of the
  system~\eqref{eq:dpgstar-poisson} can be characterized using 
  the remaining solution components, $\bmp$ and $v$, and the function $f$ as 
  \begin{equation}
  \label{eq:LMSolutions}
    \begin{aligned}
      \bzeta & = \bmp + \bmr, && \qquad \hat{\zeta}_n = 2\bmp \cdot \bmn + \bmr \cdot \bmn,
      \\
      \lambda & = f + e,    && \qquad \hat\lambda = e,
    \end{aligned}
  \end{equation}
  where $(\bmr, e)$ is in the space $H_0(\mcL^\ast)$ defined
  in~\eqref{eq:V_Vt_Poisson} and satisfies the Dirichlet problem
  $\LL^*(\bmr, e) = (\vec{0}, v+2f)$ where $\LL^*$ is as
  in~\eqref{eq:L*_Poisson}.
  Specifically, $e\in H_0^1(\Omega)$ satisfies $-\Delta e = v+2f$ and $\bmr = - \grad e$.
\end{proposition}
\begin{proof}
  By \cref{thm:dpgstaruw}, we know
  that~\eqref{eq:dpgstar-poisson-b} implies that $\bmp, v$ satisfies
  $\LL (\bmp, v) = (\vec{0}, f)$, i.e.,
  \begin{equation}
    \label{eq:q_u_char}
    \bmp - \grad v = \vec{0}, \qquad -\div \bmp = f.
  \end{equation}
  Next, we manipulate the first term of~\eqref{eq:dpgstar-poisson-a}
  as follows:
  \begin{align}
    ((\bmp,v), (\btau, \nu))_\scV
    & = (\bmp, \btau)_\som + ( \div \bmp, \div \btau)_\som 
      + (v,\nu)_\som + (\grad v, \grad \nu)_\som
      \\
    & = (\bmp, \btau - \grad \nu)_\som + ( \div \bmp, \div \btau)_\som 
      + (v,\nu)_\som + 2(\grad v, \grad \nu)_\som
    \\
    & = (\bmp, \btau-\grad \nu)_\som + (f, -\div \btau)_\som + (v,\nu)_\som 
      \\
    & \qquad + 
      2\sum_{K \in \oh} 
      \bigg[ 
      \ip{\bmn \cdot \grad v, \nu}_{H^{1/2}(\d K)}-(\Delta v, \nu)_{\selement} 
      \bigg]
    \\
    & = b( (\bmp, f, 2\bmp\cdot \bmn, 0 ), (\btau, \nu)) + (v+2f, \nu)_\som
    ,
  \end{align}
  where we have used~\eqref{eq:q_u_char} twice. Now, let
  $\bmr \in L^2(\om)^d, e \in L^2(\om)$ and
  $(\hat{r}_n, \hat{e}) \in Q(\mcL_h)$ satisfy
  \[
  b( (\bmr, e, \hat{r}_n, \hat{e}), (\btau, \nu)) = (v+2f, \nu)_\som
  \]
  for all $(\btau, \nu) \in H(\mcL_h) = \scV.$ This is a variational equation of
  the form~\eqref{eq:uwprob}. Hence, by the first item of
  \cref{thm:wellposed-mesh*}, $\bmr$ and $e$ are unique.
  Moreover, $(\bmr, e)\in H_0(\mcL^\ast)$ satisfies $\LL^* (\bmr,e) = (0, v+2f)$, and
  $\bmr\cdot \bmn|_{\d K} = \hat{r}_n|_{\d K}, e|_{\d K} = \hat
  e|_{\d K}$ on all mesh element boundaries. Thus,
  \[
  ((\bmp,v), (\btau, \nu))_\scV
  = 
  b( (\bmp+ \bmr, f+ e, (2\bmp + \bmr)\cdot \bmn, e ), (\btau, \nu)).
  \]
  Comparing this with~\eqref{eq:dpgstar-poisson-a}, the result
  follows.
\end{proof}

We may now apply \cref{thm:apriori} along with standard Bramble-Hilbert arguments (see~\cite[Corollary~3.6]{gopalakrishnan2014analysis} for details) to obtain convergence rates dictated by the following corollary to \cref{thm:apriori}.

\begin{corollary}
\label{cor:DPG*PoissonConvergenceRates}
  Let $h= \max_{K\in\oh} \mathrm{diam}(K)$, $d=2,3$, and let the assumptions of \cref{thm:apriori,prop:LM} hold.
  Let $\bmv_h = (\bmp_h, v_h) \in \scV_h$ and  $\blambda_h = (\bzeta_h, \lambda_h, \hat{\zeta}_{n,h}, \hat{\lambda}_h) \in \scU_h$  be the DPG* solutions to \cref{eq:DPGstarApprox}, with $G(\bmmu) = (f,\mu)$.
  Let $e\in H_0^1(\Omega)$ satisfy $-\Delta e = v+2f$.
  Then
  \begin{align}
  \| \bmv - \bmv_h \|_\scV + \| \blambda - \blambda_h \|_\scU
  &\leq C h^s
  \big(\|v\|_{H^{s+2}(\Omega)} + \|e\|_{H^{s+2}(\Omega)}\big)
  \,,
  \label{eq:DPG*PoissonConvergenceRatesIneq}
  \end{align} 
  for all $1/2 < s < p+1$.
\end{corollary}

\begin{proof}
  Note that the left-hand side of the inequalities in \cref{thm:apriori,cor:DPG*PoissonConvergenceRates} coincide.
  Therefore, our proof proceeds by showing that $\inf_{\vec{\nu} \in \scV_h}\| \vec{v} - \vec{\nu} \|_\scV^2 + \inf_{\vec{\mu} \in \scU_h} \| \vec{\lambda} - \vec{\mu} \|_\scU$ is bounded from above by the right-hand side of~\cref{eq:DPG*PoissonConvergenceRatesIneq}.

  First notice that $\| \vec{v} - \vec{\nu} \|_\scV^2 = \|\vec{p}-\vec{\tau}\|_{H(\div,\Omega_h)}^2 + \|v-\nu\|_{H^1(\Omega_h)}^2$.
  Therefore, the following inequality, $\inf_{\bnu \in \scV_h}\| \bmv - \bnu \|_\scV^2 \leq C h^{s} (\|\bmp\|_{H^{s+1}(\Omega)} + \|v\|_{H^{s+1}(\Omega)})$, is immediate upon substituting ${\bnu = \prod_{K\in\mesh}(\Pi_{\div}^K\bmp,\Pi_{\grad}^K v)}$, where $\Pi_{\div}^K: \bmH(\div,K) \to \vec{x}P_{p}(K) + P_{p}(K)^d$ and $\Pi_{\grad}^K:H^1(K)\to P_{p+1}(K)$ are the local Raviart--Thomas and nodal interpolation operators, for each element $K\in\mesh$.
  Since $\bmp = \grad v$, we see that $\inf_{\bnu \in \scV_h}\| \bmv - \bnu \|_\scV^2 \leq C h^{s} \|v\|_{H^{s+2}(\Omega)}$.

  To handle the term $\inf_{\vec{\mu} \in \scU_h} \| \vec{\lambda} - \vec{\mu} \|_\scU$, we remark that it is well known (cf. \cite{demkowicz2012polynomial}) that there also exist global interpolants $\Pi_{\grad}v\in H^1_0(\Omega)$, $\Pi_{\div}\bmp\in \bmH(\div,\Omega)$, and $\Pi\lambda\in L^2(\Omega)$ such that $\Pi_{\grad}v|_K\in P_{p+1}(K)$, $\Pi_{\div}\bmp|_K\in \vec{x}P_{p}(K) + P_{p}(K)^d$, and $\Pi\lambda|_K \in P_p(K)$, for all $K\in\Omega$.
  Moreover, there exists constants $C$, depending on the polynomial degree $p$ and the shape of the domain $\Omega$, such that
  \begin{subequations}
    \begin{align+}
      \|v-\Pi_{\grad}v\|_{H^1(\Omega)} &\leq C h^s|v|_{H^{s+1}(\Omega)},&&\hspace{-30pt}(1/2< s \leq p+1),
      \label{eq:H1Interpolant}
      \\
      \|\bmp-\Pi_{\div}\bmp\|_{\bmH(\div,\Omega)} &\leq C h^s|\bmp|_{H^{s+1}(\Omega)}, &&\hspace{-30pt}(0 < s \leq p+1),
      \label{eq:HdivInterpolantAPriori}
      \\
      \|\lambda-\Pi\lambda\|_{L^2(\Omega)} &\leq C h^s|\lambda|_{H^s(\Omega)}, &&\hspace{-30pt}(0 < s \leq p+1).
      \label{eq:L2Interpolant}
    \end{align+}
  \end{subequations}
  Notice that $\| \blambda - \vec{\mu} \|_\scU^2 = \|\bzeta - \sigma\|_{L^2(\Omega)}^2 + \|\lambda - \mu\|_{L^2(\Omega)}^2 + \|\hat{\zeta}_n - \hat{\sigma}_{n}\|_{H^{\minusonehalf}(\partial\mesh)} + \|\hat{\lambda} - \hat{\mu}\|_{H^{\onehalf}(\partial\mesh)}$.
  Consider $\inf_{\vec{\sigma}\in L^2(\Omega)^d}\|\bzeta - \vec{\sigma}\|_{L^2(\Omega)} \leq C h^s|\bzeta|_{H^s(\Omega)}$ and $\inf_{\mu\in L^2(\Omega)}\|\lambda - \mu\|_{L^2(\Omega)} \leq h^s|\lambda|_{H^s(\Omega)}$, both by \cref{eq:L2Interpolant}.
  Now, by \cref{eq:LMSolutions} in \cref{prop:LM}, $|\bzeta|_{H^s(\Omega)} \leq |\bmp|_{H^s(\Omega)} + |\bmr|_{H^s(\Omega)}$.
  Similarly, by invoking the identity $f = -\Delta v$, we have $|\lambda|_{H^s(\Omega)} \leq 2|f|_{H^s(\Omega)} + |e|_{H^s(\Omega)} \leq 2 |v|_{H^{s+2}(\Omega)} + |e|_{H^s(\Omega)}$.
  Next, recall that $\|\tr_n{\vec{\sigma}}\|_{H^{\minusonehalf}(\partial\mesh)} \leq C\|\bsigma\|_{\bmH(\div,\mesh)}$, for any $\bsigma\in\bmH(\div,\mesh)$, by continuity of the normal trace operator $\tr_n: \bmH(\div,\mesh) \to H^{\minusonehalf}(\partial\mesh)$.
  Therefore, invoking \cref{eq:HdivInterpolantAPriori} and \cref{eq:LMSolutions}, we see that $\inf_{\hat{\sigma}_n\in H^{\minusonehalf}(\partial\mesh)}\|\hat{\zeta}_n - \hat{\sigma}_{n}\|_{H^{\onehalf}(\partial\mesh)} \leq C h^s ( 2|\bmp|_{H^{1+s}(\Omega)} + |\bmr|_{H^{1+s}(\Omega)})$.
  Likewise, using the trace theorem, \cref{eq:H1Interpolant}, and \cref{eq:LMSolutions}, we find $\inf_{\hat{\mu}\in H^{\onehalf}(\partial\mesh)}\|\hat{\lambda} - \hat{\mu}\|_{H^{\onehalf}(\partial\mesh)} \leq C h^s|e|_{H^{1+s}(\Omega)}$.
  Recalling that $\bmp = \grad v$ and $\bmr = - \grad e$ completes the proof.
\end{proof}

The conclusion from \cref{cor:DPG*PoissonConvergenceRates} is that even if the solution $v$ has high regularity throughout the entire domain and up to the boundary, the convergence rate of the DPG* method is still controlled by a pollution variable $e$, which may happen to be less regular than $v$.
Indeed, by elliptic regularity \cite{MR2597943}, $e$, which satisfies $-\Delta e = v+2f$, will be at least as regular as $v$ in the interior of the domain, but may not be as regular up to the boundary.

To illustrate how to get higher order convergence rates in weaker norms using
duality, we want to apply \cref{thm:duality}. To this
end, we require sufficient regularity in the solution of the dual problem.
Consider the case of full regularity,
where, for any $g\in L^2(\om)$, the solution $u \in H^1_0(\om)$ of
the Dirichlet problem $-\Delta u = g$ satisfies
\begin{equation}
  \label{eq:full_reg}
  \| u \|_{H^2(\om)} \le C \| g \|_{L^2(\om)}.  
\end{equation}
The inequality above is well known to hold on convex polygonal domains.
In this case, we apply \cref{thm:duality} with $F \in \scV^\prime$ defined
\begin{equation}
  F( (\vec{\tau},\nu)) = (v-v_h, \nu)_\som.
\label{eq:LoadL2ErrorDPG*argument}
\end{equation}
Note that $F$ only sees the error in the first solution component of $\vec{v}_h = (\vec{p}_h, v_h)$ and that
\begin{equation}
  \label{eq:3}
  \| F \|_{\scV^\prime} \le \| v - v _h \|_{L^2(\om)}.
\end{equation}

We now need to verify~\cref{eq:reg}, so let us consider the present analog of~\cref{eq:DPGexact}, with the functional $F$ in~\cref{eq:LoadL2ErrorDPG*argument}:
\begin{equation}
  \left\{
  \begin{alignedat}{3}
    &(\vec{\veps}, \vec{\nu})_\scV + b(\vec{u}, \vec{\nu}) 
    &&= F( \vec{\nu}),
    \quad
    &&
    \forall\, \vec{\nu} \in \scV,
    \\
    &b(\vec{\mu}, \vec{\veps})
    &&= 0,
    \quad
    &&
    \forall\, \vec{\mu} \in \scU.
  \end{alignedat}
  \right.
\label{eq:DPGPoissonSuperconvergence}
\end{equation}
First, observe that~\cref{thm:wellposed-mesh*} implies $\scB$ is a bijection,
so~\eqref{eq:DPGexact-2} implies that
$\vec{\veps}=0$. Therefore~\eqref{eq:DPGexact-1} reduces to finding $\vec{u} = (\vec{q},u) \in \scU$, where
$b(\bmu, \bnu) = F(\bnu)$ for all $\bnu \in \scV$.
This is an equation of the form~\cref{eq:uwprob*}.
Hence, the second item of \cref{thm:wellposed-mesh*} implies that
$\LL^* \bmu = (\vec{0}, v-v_h)$; i.e., we may write
$\bmu = (-\grad u, u)\in H_0(\mcL^\ast) = \bmH(\div, \om) \times H^1_0(\om)$ such that $u \in H^1_0(\om)$
satisfies $-\Delta u= v-v_h$. Now, due to the full regularity
estimate~\eqref{eq:full_reg} applied to $u$, we have
$\| u \|_{H^2(\om)} \le C \| v - v_h \|_{L^2(\om)}$.

We may now invoke the complement of \cref{cor:DPG*PoissonConvergenceRates}, \cite[Theorem~6]{fuhrer2017superconvergent}:
\begin{theorem}
  Let $p\in\N_0$, let $\bmu$ be the solution to~\cref{eq:DPGPoissonSuperconvergence} for some arbitrary $F\in\scV^\prime$.
  Let $\bmu_h$ be the corresponding DPG solution.
  Then there exists a constant $C$, depending only on $p$ and the shape regularity of $\mesh$, such that
  \begin{equation}
    \inf_{\vec{\mu} \in \scU_h} \| \vec{u} - \vec{\mu} \|_\scU
    \leq
    Ch^{p+1}\big(\|u\|_{H^{p+2}(\Omega)} + \|\bmq\|_{H^{p+1}(\mesh)}\big)
    .
  \end{equation}
\end{theorem}

Using the fact that $\bmq = -\grad u$, we now have the estimate
\begin{equation}
  \inf_{\vec{\mu} \in \scU_h} \| \vec{u} - \vec{\mu} \|_\scU
  \leq
  Ch\|u\|_{H^{2}(\Omega)}
  \leq
  Ch\|v-v_h\|_{L^2(\Omega)}
  .
\end{equation}
Which is of the same form as~\cref{eq:reg}.
Finally, it is clear that assumption~\eqref{eq:reg} holds with $c_0(h) = h$.
Then,~\eqref{eq:3} and \cref{thm:duality} imply
\[
\| v - v_h \|_{L^2(\om)}^2 
 = F(v-v_h) 
\le C h \| v - v_h \|_{L^2(\om)}     \bigg[
    \inf_{\nu \in \scV_h} 
    \| \bmv - \bnu \|_\scV^2 +
    \inf_{\bmmu \in \scU_h} \| \blambda - \bmmu \|_\scU^2
    \bigg]^{1/2},
\]
which provides one higher order of convergence in the $L^2$ norm for the solution component $v_h$.

Ultimately, the upshot of the entire \emph{a priori} error analysis above is that poor {\it a priori} convergence rates are possible with this method, even for infinitely smooth solutions $v$, due to the Lagrange multiplier $\lambda$, which may not be as smooth (cf. \cref{sub:square_domain_dirichlet}).
Thus, without an adaptive algorithm that helps one capture irregular solutions, the DPG* method is generally impractical for high-order methods.
We therefore proceed by studying \textit{a posteriori} error control.


\section{A posteriori error control} 
\label{sec:a_posteriori_error_control}

In this section, we will present an abstract {\it a posteriori} error
estimator valid for all ultraweak DPG* formulations (see
\cref{sub:ultraweak_formulations}). We then proceed to work out the
example of the Poisson problem in full detail.  Note that abstract
ultraweak formulations encompass many physical models besides the
Poisson example given above. Other important examples with similar functional settings include convection-dominated diffusion \cite{demkowicz2013robust,chan2014robust}, Stokes flow
\cite{roberts2014dpg,carstensen2018low}, linear elasticity
\cite{Bramwell12,carstensen2016low}, and acoustics
\cite{gopalakrishnan2014dispersive,petrides2016ices,astaneh2018perfectly}.

\subsection{Designing error estimators for general ultraweak DPG* formulations}
\label{sub:abstract_stability_analysis}

Consider the general setting of \cref{sub:ultraweak_formulations} and
the broken ultraweak DPG* formulation which is proved to be well posed
in \cref{thm:dpgstaruw}. Namely, with $\LL$ set to the general partial
differential operator in~\cref{eq:StrongFormulation}, the problem of
finding a $v \in H_0(\LL)$ satisfying $\LL v = f$ is reformulated as
\cref{eq:1uw}, where
\[
\ip{\scB (\mu, \rho), \nu}_{\scV} = 
b((\mu, \rho), \nu) = (\mu, \LL_h \nu)_\som + \ip{ \rho,\nu }_h,
\] for all
$ (\mu, \rho) \in \scU = L^2 \times Q(\LL_h)$ and
$ \nu \in \scV = H(\LL_h).$ The DPG* method produces an approximation
to $v$ using two finite-dimensional subspaces $\scU_h \subset \scU$
and $\scV_h \subset \scV$.
Let 
\begin{equation}
  \label{eq:abstracteta}
\eta(\nu) = \sup_{\rho \in Q(\LL_h)} 
\frac{ \ip{\rho, \nu}_h }{\;\| \rho \|_{Q(\LL_h)}}, \qquad \nu \in \scV_h.  
\end{equation}
This quantity can usually be interpreted as a ``jump term'' in
applications. The message of the next theorem is that, notwithstanding
the generality of the operators considered, the design of {\it a posteriori} error estimators for all ultraweak DPG* formulations
reduces to obtaining upper and lower bounds for $\eta$.

\begin{theorem}
\label{thm:DualDPGAPosterioriErrrorEstimator-0}
Consider the ultraweak DPG* formulation~\cref{eq:1uw} with
$G((\mu, \rho)) = (f, \mu)$, for some $f \in L^2$.
Suppose~\cref{eq:UWasm} holds and $\LL^\ast: H_0(\mcL^\ast) \to L^2$ is a
bijection.  Then, for any $v_h \in \scV_h$ (not necessarily equal to
the DPG* solution),
\begin{equation}
  \|\scB\|\inv \|G - \scB' v_h\|_{\scU^\prime}
  \leq
  \| v - v_h \|_\scV
  \leq
  \|\scB\inv\| \|G - \scB' v_h\|_{\scU^\prime}
\label{eq:EquivalenceErrorResiduals}
\end{equation}
and, moreover,
\begin{equation}
  \|G - \scB' v_h\|_{\scU^\prime}^2
  =
  \| \LL_h v_h - f \|_\som^2 + \eta(v_h)^2
  .
\label{eq:ResidualCharacterization}
\end{equation}
\end{theorem}
\begin{proof}
  In view of~\cref{eq:1uw}, $v-v_h$ satisfies
  \begin{equation}
    \label{eq:v_vh}
    \left\{
      \begin{alignedat}{3}
        &R_\scV  (v-v_h) + \scB ((\lambda, \sigma)) && =  -R_\scV v_h ,
        \\
        &\scB' (v-v_h) &&  = G - \scB' v_h.
      \end{alignedat}
    \right.
  \end{equation}
  By \cref{thm:dpgstaruw}, $\scB$ is a bijection.
  Hence, applying the identity \cref{eq:identity5} of~\cref{prop:identities} to the system~\cref{eq:v_vh}, we arrive at
  \begin{equation}
    \| v - v_h \|_\scV
    =
    \normm{G - \scB' v_h}_{\scU^\prime}
    =
    \sup_{\mu\in\scU}
    \frac{\ip{G - \scB' v_h,\mu}_{\scU}}{ \normm{\mu}_{\scU}}
    \,,
  \end{equation}
  where the second equality comes directly from definition~\cref{eq:DualOfEnergyNorm}.
  Next, recall that the norms $\|\cdot\|_\scU$ and $\normm{\,\cdot\,}_\scU$ are equivalent.
  Indeed, $\|\scB\inv\|\inv\|\mu\|_\scU \leq \normm{\mu}_{\scU} \leq \|\scB\|\|\mu\|_\scU$, for all $\mu\in\scU$.
  The first result,~\cref{eq:EquivalenceErrorResiduals}, now follows immediately.

  To arrive at~\cref{eq:ResidualCharacterization}, simply observe that
  \[
  \begin{aligned}
  \| G - \scB' v_h \|_{\scU'}^2
  &= 
  \sup_{\mu  \in L^2, \;\rho \in Q(\LL_h)}
  \frac{\big| (f, \mu)_\som - b((\mu, \rho), v_h) \big|^2}
  { \| (\mu, \rho) \|_{\scU}^2} 
  \\
  & 
  =
  \sup_{\mu  \in L^2, \;\rho \in Q(\LL_h)}
  \frac{ \big| (f - \LL_h v_h, \mu)_\som  - \ip{ \rho, v_h}_h \big|^2}
  {\| \mu\|_\som^2 + \| \rho \|_{Q(\LL_h)}^2}
  \\
  & 
  =
  \sup_{\mu  \in L^2}
  \frac{ \big| (f - \LL_h v_h, \mu)_\som \big|^2}{ \| \mu \|_\som^2}
  + 
  \sup_{\rho \in Q(\LL_h)}
  \frac{\big| \ip{ \rho, v_h}_h \big|^2}
  {\| \rho \|_{Q(\LL_h)}^2}
  \end{aligned}
  \]
  and the second result also follows. 
\end{proof}

\subsection{A posteriori error analysis for the Poisson example} 
\label{sub:analysis_of_a_complete_error_estimator}

In this subsection, we develop a computable quantity that is
equivalent to the function $\eta$ defined above, for the example of the DPG*
method for Poisson equation. We then provide a complete analysis of
reliability and efficiency of the resulting error estimator.

Recall the variational formulation derived in \cref{eg:poisson} for
Poisson's equation. Its bilinear form (see~\cref{eq:DPG*Poissonform})
is
\begin{equation}
  b((\vec{\sigma},\mu,\hat{\sigma}_n,\hat{\mu}),
  (\vec{\tau},\nu))
  =
  (
  (\vec{\sigma},\mu),\mcL_h (\vec{\tau},\nu))_\sOmega
  +
  \langle\hat{\mu},\vec{\tau}\cdot\bmn\rangle_h + \langle\hat{\sigma}_n,\nu\rangle_h
  \,,
\label{eq:DPGstarGeneral}
\end{equation}
where
$\mcL_h(\nu,\vec{\tau}) = (\vec{\tau} - \grad_h \nu,\, -\div_h
\vec{\tau})$.  In this subsection, we proceed by assuming, for
simplicity, that $\om \subset \RRR^2$, and that $\mesh$ is a
geometrically conforming triangular shape-regular mesh.  Let $\mcE$
denote the set of all mesh edges and let $\mcE_{\interior}\subset\mcE$
be the set of all interior edges of $\mesh.$ Let $h_E$ be the length
of any edge $E\in\mcE$.  Any $E \in \mcE_{\interior}$ has two adjacent
elements $K^+$ and $K^-$ such that $E = \bdry K^+ \cap \bdry K^-$.
Let
\begin{equation}
  \jump{\vec{\tau} \cdot \vec{n}}
  =
  \vec{\tau}_{K^+}\cdot\vec{n}_{K^+} + \vec{\tau}_{K^-}\cdot\vec{n}_{K^-}
  \,,
  \qquad
  \jump{\vec{\tau} \cdot \vec{n}^\perp}
  =
  \vec{n}_{K^+}^\perp\cdot\vec{\tau}_{K^+} + \vec{n}_{K^-}^\perp\cdot\vec{\tau}_{K^-}
  \,.
\end{equation}
Here, $\vec{n}_K^\perp$ is the tangential unit vector; i.e., if
$\vec{n}_K = (n_1,n_2)$ then $\vec{n}_K^\perp = (-n_2,n_1)$.  If
$E\in \mcE \setminus \mcE_{\interior}$ is an exterior edge on the
boundary of an element $K$, then with $\vec{n}$ equal to the outward unit
normal on $\d \om$, we simply set
$ \jump{\vec{\tau}\cdot \vec{n}} = \vec{\tau}_{K}\cdot\vec{n}$ and
$ \jump{\vec{\tau} \cdot \vec{n}^\perp} =
\vec{n}^\perp\cdot\vec{\tau}_{K}.$ Similarly, for any scalar function $\nu$
that may be discontinuous across an interface
$E \in \mcE_{\interior}$, we define 
\begin{equation}
  \jump{\nu \vec n}
  =
  \nu_{K^+}\vec{n}_{K^+} + \nu_{K^-}\vec{n}_{K^-}
\end{equation}
and set $\jump{\nu} = \nu_{K}\vec{n}$ on boundary edges $E \in \mcE \setminus \mcE_{\interior}$.

The setting for the DPG* method for the Poisson equation (of
\cref{eg:poisson}), including its discrete spaces $\scU_h, \scV_h$, is
described in \cref{sub:Application_to_the_Poisson_example}.  We
continue with these settings in this subsection.  Let $h$ denote the
maximum of $h_E$ over all $E \in \mcE$. When mesh dependent quantities
$A$ and $B$ satisfy $A \le C B$, with a positive constant $C$
independent of $h$, then we write $ A \lesssim B$. When $A \lesssim B$
and $B \lesssim A$, then we write $A \eqsim B$.  The main result of
this subsection is the next theorem.

\begin{theorem}  \label{thm:PossionDualDPGAPosterioriErrrorEstimator}
  Under the above settings, suppose $(\vec p, v) \in \scV$ and
  $ (\vec{p}_h, v_h) \in \scV_h$ are the exact solution and the
  discrete DPG* solution of the Laplace problem, respectively. Then
  \[
    \| (\vec p, v) - (\vec{p}_h, v_h) \|_\scV \;\eqsim \;
    \eta_i(\vec{p}_h, v_h)
  \]
  for  $i\in\{1,2\}$, where  the computable error  estimators $\eta_i$
  are defined by
  \begin{subequations}
    \begin{align}
      \eta_1(\vec{p}_h, v_h)^2
      & =
      \big\|\mcL (\vec{p}_h, v_h) - f \big\|_\sOmega^2
      +
      \sum_{E\in\mcE_\interior}
      h_E
      \big\|\jump{\vec{p}_{h} \cdot \vec n}\big\|_{L^2(E)}^2
      +
      \sum_{E\in\mcE}
      h_E
      \big\|\jump{v_h \vec n}\big\|_{H^1(E)}^2
      \,,
    \label{eq:etastarH1}
      \\
      \eta_2(\vec{p}_h, v_h)^2
      & =
      \big\|\mcL (\vec{p}_h, v_h) - f \big\|_\sOmega^2
      +
      \sum_{E\in\mcE_\interior}
      h_E
      \big\|\jump{\vec{p}_{h} \cdot \vec n }\big\|_{L^2(E)}^2
      +
      \sum_{E\in\mcE}
      h_E^{-1}
      \big\|\jump{v_h  \vec n}\big\|_{L^2(E)}^2
      \,.
    \label{eq:etastar}
    \end{align}
    \end{subequations}
\end{theorem}

The remainder of this subsection is devoted to proving this
theorem. The main idea is to apply
\cref{thm:DualDPGAPosterioriErrrorEstimator-0}, but some
intermediate results need to be established first. To this end,
recall that there exists a $H^1$-norm minimum energy extension
$\calliE$ that provides a continuous right inverse of
$\tr: H^1(\om) \to H^{\onehalf}(\d\oh)$.
Indeed, for all $\hat{w}\in H^{\onehalf}(\bdry\mesh)$, $\calliE$ is defined using the pre-image set $\tr^{-1}\{ \hat w\}$ simply as
\begin{equation}
  \calliE(\hat{w})
  =
  \argmin_{w \in \tr^{-1} \{ \hat w\}}
  \|w\|_{H^1(\om)}.
  \label{eq:EnergyExtensionH1}
\end{equation}
Similarly, for all $\hat{q}_n\in H^{\minusonehalf}(\bdry\mesh)$, the continuous right inverse of $\tr_n: H(\div,\om) \to H^{\minusonehalf}(\d\oh)$ is defined
\begin{equation}
  \Ed(\hat{q}_n)
  =
  \argmin_{\vec{q}_n \in \tr_n^{-1} \{\hat{q}\}}
  \|\,\vec{q}\,\|_{H(\div,\om)}
  .
  \label{eq:EnergyExtensionHdiv}
\end{equation}
Clearly these operators can also also be applied element by element.

Before establishing a number of lemmas, we pause to construct two helpful observations.
First, for any
$\hat{q}_n \in \hat{P}_p(\d\oh)$, let $\hat{q}_E$ denote the function in
$\hat{P}_p(\d\oh)$ that vanishes on all edges of $\mcE$, except on the edge
$E$ where it equals $\hat{q}|_E$. Observe that $\Ed(\hat{q}_E)$ is supported
only on
$\om_E = \bigcup\{ K\in\mesh : \mathrm{meas}(\bdry K \cap
E)\neq\emptyset\}$.  Second, let $b_E$ denote the edge bubble of $E$
(i.e., the product of the barycentric coordinates of the endpoints of
$E$) and define
$\tilde{P}^0_{p+2}(\d\oh) = \{ \mu \in \tr( P_{p+2}(\oh) \cap
H_0^1(\om)): $ on any edge $E\in \mcE,$ $ \mu|_E = r_p b_E$ for some
$r_p \in P_p(E)\}.$ Likewise, for any $\hat{\mu}\in\tilde{P}^0_{p+2}(\d\oh)$,
the function $\calliE(\hat{\mu})$ is supported only on $\Omega_E$.

Our first lemma may be
thought of as an inf-sup condition involving the space of edge bubbles $\tilde{P}^0_{p+2}(\d\oh)$.

\begin{lemma}
  \label{lem:inf-H1-sup}
  For any degree $p\ge 0$ and for any $\vec{q}_h \in P_p(\oh)^2$,
  \[
    \sum_{E \in \Eint} h_E \Big\| \jump{\vec{q}_h \cdot \vec n } \Big\|_{L^2(E)}^2
    \lesssim
    \sup_{\hat{\mu} \in \tilde{P}^0_{p+2}(\d\oh)}
    \frac{\ip{ \vec{q}_h \cdot \vec n, \hat{\mu}}_h^2 }{\| \Eg(\hat{\mu}) \|_{H^1(\om)}^2}.
  \]
\end{lemma}
\begin{proof}
  We shall use the following two estimates that can be proved by scaling
  arguments using finite dimensionality (see
  e.g.~\cite{verfurth1996review}).  For all $\hat{w} \in \tr( P_p(\oh)
  \cap H_0^1(\om))$, 
  \begin{gather}
    \|{b}_E \hat{w}\|_{L^2(E)}^2
    \leq
    \|\hat{w}\|_{L^2(E)}^2
    \lesssim
    ({b}_E\hat{w}, \hat w)_E
    \,,
    \label{eq:BubbleBounds1}
    \\
    |\calliE({b}_E \hat{w}_E)|_{H^1({\Omega_E})}
    \lesssim
    h_E^{\minusonehalf} \|\hat{w}\|_{L^2(E)},    
    \label{eq:BubbleBounds2}
  \end{gather}
  where $(\cdot, \cdot)_E$ denotes the inner product of $L^2(E)$ and
  $\hat{w}_E$ is as defined above. For any $\vec{q}_h \in P_p(\oh)$ with
  nontrivial $\jump{\vec{q}_h \cdot \vec n}_E$, we have
  \begin{equation}
    \label{eq:4}
  \begin{aligned}
    h_E \big\|\jump{\vec{q}_{h} \cdot \vec n}\big\|_{L^2(E)}^2
    & \lesssim
    h_E (b_E \jump{\vec{q}_h \cdot \vec n}, \jump{\vec{q}_h \cdot n})_E
    && \text{ by }~\eqref{eq:BubbleBounds1}
    \\ 
    & \lesssim 
      \frac{ (b_E \jump{\vec{q}_h \cdot \vec n}, \jump{\vec{q}_h \cdot \vec{n}})_E}
      {    |\calliE({b}_E \jump{\vec{q}_h \cdot \vec{n}}_E )|_{H^1({\Omega_E})}}
      h_E^{\onehalf}
      \| \jump{\vec{q}_h \cdot \vec{n}}_E \|_{L^2(E)}
    && \text{ by }~\eqref{eq:BubbleBounds2},
  \end{aligned}    
  \end{equation}
  which, after canceling
  $ h_E^{\onehalf} \| \jump{\vec{q}_h \cdot \vec{n}}_E \|_{L^2(E)}$
  from both sides, provides a local version of the result we want
  to prove.

  To get to the global estimate, we will accumulate the contributions
  of jumps across each edge. For this, its useful to observe that for all
  $\hat{\mu} \in \tilde{P}^0_{p+2}(\d\oh)$, we have
  \begin{equation}
    |\calliE(\hat{\mu})|_{H^1(\Omega)}^2
    \lesssim
    \sum_{E\in\mcE_\interior} |\calliE(\hat{\mu}_E)|_{H^1({\Omega_E})}^2
    \,.
  \label{eq:BubbleNormBound}
  \end{equation}
  This can be seen beginning from the linearity of $\calliE$
  and the fact that $\mcE_K = \{ E \in \mcE:
  \mathrm{meas}(E\cap\partial K)\neq 0\}$ has fixed finite cardinality:  
\begin{align}
  \lefteqn{|\calliE  (\hat{\mu})|_{H^1(\Omega)}^2
   =
    \bigg|\sum_{E\in\mcE_\interior}\calliE(\hat{\mu}_E)\bigg|_{H^1(\Omega)}^2
  =
  \sum_{K\in\mesh}
    \Big|\sum_{E\in\mcE_\interior}\calliE(\hat{\mu}_E)\Big|_{H^1(K)}^2}
  \\
  \quad & =
  \sum_{K\in\mesh}
  \Big|\sum_{E \in \mcE_K}\calliE(\hat{\mu}_E)\Big|_{H^1(K)}^2
  \lesssim
  \sum_{K\in\mesh}
  \sum_{E \in \mcE_K}
    |\calliE(\hat{\mu}_E)|_{H^1(K)}^2
  \lesssim
  \sum_{E\in\mcE_\interior}
  |\calliE(\hat{\mu}_E)|_{H^1({\Omega_E})}^2,
\end{align}
which proves~\eqref{eq:BubbleNormBound}.

We can now complete the proof as follows. Starting from~\eqref{eq:4},
\begin{align}
  \sum_{E\in\mcE_\interior}
  h_E &
  \Big\|\jump{\vec{q}_h \cdot \vec{n}}\Big\|_{L^2(E)}^2
        \lesssim
    \sum_{E\in\mcE_\interior}
      \frac{ (b_E \jump{\vec{q}_h \cdot \vec n}, \jump{\vec{q}_h \cdot \vec{n}})_E^2}
      {    |\calliE({b}_E \jump{\vec{q}_h \cdot \vec{n}}_E )|_{H^1({\Omega_E})}^2}
  \\
  &
    \le
    \sum_{E\in\mcE_\interior}
    \sup_{\hat{\mu} \in \tilde{P}^0_{p+2}(\d\oh)}
    \frac{ ( \hat{\mu}_E, \jump{\vec{q}_h \cdot \vec{n}})_E^2}
    {    |\calliE( \hat{\mu}_E )|_{H^1({\Omega_E})}^2}
=
    \sup_{\hat{\mu} \in \tilde{P}^0_{p+2}(\d\oh)}
    \,
    \frac
    {\Big(\sum_{E\in\mcE_\interior}(\hat{\mu}_E,\jump{\vec{q}_{h}\cdot\vec n})_E\Big)^2}
    {\sum_{E\in\mcE_\interior}|\calliE(\hat{\mu}_E)|_{H^1({\Omega_E})}^2}
    ,
  \end{align}
  where, in the equality, we have exploited a property of suprema over components of a
  Cartesian product space (noting that the space $\tilde{P}^0_{p+2}(\d\oh)$ is
  the Cartesian product of $b_E P_p(E)$ over all interior edges
  $E$). Now, the result follows by noting that the numerator above
  equals $\ip{ \mu, \vec{q}_h \cdot \vec n}_h^2$ and by bounding the
  denominator using \eqref{eq:BubbleNormBound} and the Poincar\'e
  inequality.
\end{proof}

\begin{lemma}
  \label{lem:inf-Hdiv-sup}
  For any degree $p\ge 1$ and for any $w_h \in P_p(\oh)$,
  \[
    \sum_{E \in \Eint} h_E \Big| \jump{ w_h \vec n } \Big|_{H^1(E)}^2
    \lesssim
    \sup_{\hat{\sigma}\cdot \vec n \in \hat{P}_p(\d\oh)}
    \frac{\ip{ \hat{\sigma}\cdot \vec n, w_h }_h^2 }
    {\| \Ed(\hat{\sigma}\cdot \vec n) \|_{H(\div, \om)}^2},
  \]
  where $\hat{P}_p(\d\oh)$ is as defined in~\eqref{eq:tr-fl}.
\end{lemma}
\begin{proof}
  For any $w_h \in P_p(\oh)$ and any $E \in \mcE$, the function
  $\jump{\vec{n}^\perp \cdot \grad w_h }$ represents the tangential
  derivative of the jump of $w_h$ across $E$. Then
  $\phi_E = \Eg( \jump{\grad w_h \cdot \vec{n}^\perp}_E b_E)$ is
  supported on $\om_E$ and the trace of $\phi_E$ vanishes on all edges except
  $E$.  Let $\ohE = \{ K \in \oh: K \subseteq \om_E\}$.  Using
  the vector curl of the scalar function $\phi_E$, by an application 
  of \eqref{eq:BubbleBounds1}, we have  
  \begin{align}
    \big|\jump{w_h \vec n}\big|_{H^1(E)}^2
    & \lesssim (b_E \jump{\vec{n}^\perp \cdot \grad w_h },
      \jump{\vec{n}^\perp \cdot \grad w_h })_E
    \\
    & = \int_E \phi_E \jump{\vec{n}^\perp \cdot \grad w_h }
      = \sum_{K \in \ohE} \int_{\d K} \phi_E
      \vec{n}^\perp \cdot \grad w_h
    \\
    & =
      \sum_{K \in \ohE} 
      (\curl \phi_E, \grad w_h)_K
      =
      \sum_{K \in \ohE}
      \int_{\d K} \vec{n} \cdot \curl \phi_E w_h
     = (\curl \phi_E,  \jump{w_h \vec n})_E,
  \end{align}
  where we have also used the Stokes and the divergence theorems in
  succession.  Now, noting that
  $\| \curl \phi_E\|_{H(\div, \om_E)} = |\phi_E|_{H^1(\om_E)} = |\Eg(
  \jump{\grad w_h \cdot \vec{n}^\perp}_E b_E)|_{H^1(\om_E)}$ we deduce
  using~\eqref{eq:BubbleBounds2} that
  $\| \curl \phi_E\|_{H(\div, \om_E)} \lesssim h_E^{\minusonehalf}
  |\jump{w_h \vec n}|_{H^1(E)}.$ Hence
  \begin{align}
     \big|\jump{w_h \vec n}\big|_{H^1(E)}^2
    & \lesssim
      \frac{( \curl \phi_E, \jump{w_h \vec n})_E }
      {\| \curl \phi_E \|_{H(\div, \om_E)}}
      \| \curl \phi_E \|_{H(\div, \om_E)}
      \lesssim
            \frac{( \curl \phi_E, \jump{w_h \vec n})_E }
      {\| \curl \phi_E \|_{H(\div, \om_E)}}
     h_E^{\minusonehalf} \big|\jump{w_h \vec n}\big|_{H^1(E)}.
  \end{align}
  Together with the minimal extension property
  $\| \Ed(\vec{n} \cdot \curl \phi_E) \|_{H(\div, \om)} \le \| \curl \phi_E
  \|_{H(\div, \om)},$ this implies that
  \begin{align}
    h_E \Big|\jump{w_h \vec n}\Big|_{H^1(E)}^2
    & \lesssim
      \frac{( \curl \phi_E, \jump{w_h \vec n})_E }
      {\| \Ed(\vec n \cdot \curl \phi_E) \|_{H (\div, \om_E)}^2}
      \le
      \sup_{\hat{\sigma}_E \cdot \vec n \in \hat{P}_p(E)}
      \frac{ \ip{ \hat{\sigma}_E \cdot \vec n, w_h}_h }
      {\| \Ed( \hat{\sigma}_E \cdot \vec n) \|_{H (\div, \om_E)}^2}
      ,
  \end{align}
  where $\hat{P}_p(E)$ denotes the subspace of functions in
  $\hat{P}_p(\oh)$ supported on $E$. We have thus arrived at a local
  version of the desired inequality. By proving an analogue of
  \eqref{eq:BubbleNormBound} for $\Ed$, and following along the lines of the
  proof of \cref{lem:inf-H1-sup}, we finish the proof.
\end{proof}

Since the suprema in \cref{lem:inf-H1-sup,lem:inf-Hdiv-sup} are related to the function $\eta$ in
\eqref{eq:abstracteta}, these lemmas can be thought of providing lower
bounds, often called {\em efficiency} estimates in the analysis of
{\it a posteriori} estimators. To prove upper bounds, also called {\em
  reliability} estimates, we need some additional tools. Recall that
any $\vec{\sigma}\in H(\div,\Omega)$ may be decomposed using the
so-called {\em regular decomposition} as
$\vec{\sigma} = \curl(\varphi_{\vec{\sigma}}) +
\vec{\psi}_{\vec{\sigma}}$ such that
\begin{equation}
  \|\varphi_{\vec{\sigma}}\|_{H^1(\Omega)}
  +
  \|\vec{\psi}_{\vec{\sigma}}\|_{H^1(\Omega)}
  \lesssim
  \|\vec{\sigma}\|_{H(\div,\Omega)}
  \,.
  \label{eq:RegularDecomposition}
\end{equation}
We shall also need low-regularity \textit{commuting
  quasi-interpolators} of \cite{demlow2014posteriori}, built using
refinements of earlier ideas in~\cite{clement1975approximation,schoberl2008posteriori}.
Namely, there exist operators $\Igrad: H^1(\Omega)\to P_1(\oh) \cap H^1(\om)$ and
$\Idiv :H(\div,\Omega) \to \mcR\mcT_0(\mesh)\cap H(\div,\om)$, such
that $\curl\circ\Igrad = \Idiv \circ\curl$.  Here,
$\mcR\mcT_0(\mesh) = P_0(\oh)^2 + \vec{x} P_0(\oh)$, the lowest order
Raviart--Thomas space. In addition, the following inequalities
\cite[Lemma~6]{demlow2014posteriori} hold for all $\mu \in H^1(\om)$
and $\vec{\sigma} \in H(\div, \om)$:
\begin{subequations}
  \begin{gather}
    \sum_{E \in \mcE} 
    h_E^{-1}\|\mu-{\Igrad} \mu\|_{L^2(E)}^2
    \lesssim
     \|\mu\|_{H^1({\Omega})}^2
    \,,
    \label{eq:Clement}
    \\
    \sum_{E \in \mcE}
    h_E^{-1} \|\varphi_{\vec{\sigma}}-{\Igrad} \varphi_{\vec{\sigma}}\|_{L^2(E)}^2
    +
    h_E^{-1} \|(\vec{\psi}_{\vec{\sigma}}-{ {\Idiv}} \vec{\psi}_{\vec{\sigma}})
    \cdot\vec{n}\|_{L^2(E)}^2
    \lesssim
    \|\vec{\sigma}\|_{H(\div,{\Omega})}^2
    \,.
    \label{eq:HdivInterpolant}
  \end{gather}
\end{subequations}
These results also hold with $H^1(\Omega)$ replaced by $H^1_0(\Omega)$ and $H(\div,\Omega)$ replaced by $H_0(\div,\Omega)$.

\begin{lemma}
  \label{lem:FluxBounds}
  For any degree $p \ge 0$ and any $\vec{q}_{h} \in P_p(\oh)^2$
  satisfying $\langle\hat{\mu}_1,\vec{q}_{h}\cdot\vec{n}\rangle_h = 0$
  for all $\hat{\mu}_1 \in \tr( P_1(\oh) \cap H_0^1(\om)),$ we have
  \begin{equation}
    \sum_{E\in\mcE_\interior}
    h_E
    \Big\|\jump{\vec{q}_{h}\cdot\bmn} \Big\|_{L^2(E)}^2
    \,\eqsim\,
    \sup_{\mu\in H^1_0(\Omega)}
    \frac{\langle \mu ,\bmq_{h}\cdot\bmn\rangle_h^2}
    {\quad \|\mu\|_{H^1(\Omega)}^2}.
  \label{eq:FluxEquivalence}
  \end{equation}
\end{lemma}
\begin{proof}
  By \cref{lem:inf-H1-sup},
  \begin{align}
    \sum_{E\in\mcE_\interior}
    h_E
    \Big\|\jump{\vec{q}_{h}\cdot\bmn} \Big\|_{L^2(E)}^2
    & \lesssim
    \sup_{\hat{\mu}_h \in \tilde{P}^0_{p+2}(\d\oh)}
      \frac{\ip{ \vec{q}_h \cdot \vec n,
      \hat{\mu}_h}_h^2 }{\| \Eg(\hat{\mu}_h) \|_{H^1(\om)}^2}
    \lesssim
      \sup_{\hat\mu\in H_0^{\onehalf}(\d\oh)}
      \frac{\langle \hat{\mu}
      ,\vec{q}_{h}\cdot\bmn\rangle_h^2}{\|\hat\mu\|_{H^{\onehalf}(\d\oh)}^2}
    \\
    & =
      \sup_{\mu\in H_0^1(\om)}
      \frac{\langle \mu
      ,\vec{q}_{h}\cdot\bmn\rangle_h^2}{\|\mu\|_{H^1(\om)}^2}
      \,,
  \end{align}
  where the last identity follows from previous works (see
  \cite{demkowicz2011analysis} or \cite[Theorem~2.3]{Carstensen15}).

  Hence, it suffices to prove the reverse inequality.
  Since
  $\langle \Igrad \mu,\vec{q}_{h}\cdot\vec{n}\rangle_h = 0,$
  \begin{align}
    \langle \mu &,\vec{q}_{h}\cdot\bmn\rangle_h   
     =
      \langle
      \mu - \Igrad \mu ,\vec{q}_{h}\cdot\bmn\rangle_h
      =
      \sum_{E \in \Eint} 
    (\mu - \Igrad \mu, \jump{\vec{q}_{h}\cdot\bmn})_E
    \\
    & 
      \lesssim
      \sum_{E \in \Eint}
      \!
      h_E^{\minusonehalf}
      \| \mu - \Igrad \mu \|_{L^2(E)}
      \,
      h_E^{\onehalf}
      \Big\| \jump{\vec{q}_{h}\cdot\bmn} \Big\|_{L^2(E)}
      \lesssim
      \|\mu\|_{H^1(\om)}
      \bigg(\sum_{E \in \Eint}
      \!
      h_E
      \Big\| \jump{\vec{q}_{h}\cdot\bmn} \Big\|_{L^2(E)}^2\bigg)^{\onehalf}.
  \end{align}
  Here, in the final line, we have used the Cauchy--Schwarz inequality and \eqref{eq:Clement}.
  This completes the proof of~\eqref{eq:FluxEquivalence}.
\end{proof}

\begin{lemma}
  \label{lem:TraceBounds}
  For any degree $p \ge 1$ and any $w_h \in P_p(\oh)$ satisfying
  $\int_E \jump{w_h \bmn } = 0$ on all edges  $E\in\mcE,$
  \begin{equation}
    \sum_{E\in\mcE}
    h_E
    \Big\|\jump{w_h \bmn }\Big\|_{H^1(E)}^2
    \eqsim
    \sum_{E\in\mcE}
    h_E^{-1}
    \Big\|\jump{w_h\bmn} \Big\|_{L^2(E)}^2
    \eqsim
    \sup_{\bsigma\in \bmH(\div,\Omega)}
    \frac{\langle \bsigma \cdot \bmn,w_h\rangle_h^2}
    {\quad \|\bsigma\|_{\bmH(\div,\Omega)}^2}
    \,.
  \label{eq:TraceEquivalence}
  \end{equation}
\end{lemma}
\begin{proof}
  The first equivalence in~\eqref{eq:TraceEquivalence} immediately
  follows from the Poincar\'e inequality since the mean value of
  $\jump{w_h \bmn}$ vanishes. To prove the remaining equivalence,
  first observe that \cref{lem:inf-Hdiv-sup} implies
  \begin{align}
    \sum_{E \in \Eint} h_E \Big| \jump{ w_h \vec n } \Big|_{H^1(E)}^2
    & \lesssim
    \sup_{\hat{\sigma}\cdot \vec n \in \hat{P}_p(\d\oh)}
    \frac{\ip{ \hat{\sigma}\cdot \vec n, w_h }_h^2 }
      {\quad \| \Ed(\hat{\sigma}\cdot \vec n) \|_{H(\div, \om)}^2}
    \\
    &
      \le
    \sup_{\hat{\sigma}\cdot \vec n \in H^{\minusonehalf}(\d\oh)}
    \frac{\ip{ \hat{\sigma}\cdot \vec n, w_h }_h^2 }
    {\quad \| \hat{\sigma}\cdot \vec n
      \|_{H^{\minusonehalf}(\d\oh)}^2}
      =
      \sup_{\sigma \in H(\div,\om)}
      \frac{\ip{ {\sigma}\cdot \vec n, w_h }_h^2 }
      {\quad \| {\sigma}      \|_{H(\div,\om)}^2 }
      ,
  \end{align}
  where the last identity is well known (see~\cite[Theorem~2.3]{Carstensen15}).
  This proves one side of the stated equivalence.

  To prove the remaining inequality, we start by decomposing any given
  $\vsig \in H(\div, \om)$ using above-mentioned regular
  decomposition: $\vsig = \curl \varphi_\vsig + \vpsi_\vsig$.  Then,
  since the jump of $w_h$ has zero mean value on every edge,
  we observe that
  $\ip{ \bmn \cdot \Idiv \vpsi_\vsig, w_h}_h = \ip{ \bmn\cdot
    \Idiv(\curl \varphi_\vsig), w_h }_h= 0.$ Therefore, by the
  commutativity property of the quasi-interpolators,
  \[
    \ip{ \vsig \cdot \bmn, w_h}_h
    =
    \ip{ \bmn \cdot ( \curl \varphi_\vsig - \Igrad \varphi_\vsig),
      w_h }_h
    +
    \ip{\bmn \cdot (\vpsi_\vsig - \Idiv \vpsi_\vsig), w_h}_h.
  \]
  We proceed labeling the terms on the right as $t_1$ and $t_2$,
  respectively.  Using the divergence theorem and the Stokes theorem
  in succession,
  \begin{align}
    t_1
    & = (\curl( \varphi_\vsig - \Igrad \varphi_\vsig), \grad w_h)_\som
    \\
    & = \ip{  \varphi_\vsig - \Igrad \varphi_\vsig, \vec{n}^\perp\cdot  \grad w_h}_h
      =
      \sum_{E \in \mcE} (\varphi_\vsig - \Igrad \varphi_\vsig,
      \jump{\vec{n}^\perp\cdot  \grad w_h})_E
    \\
    &
      \lesssim
      \sum_{E \in \mcE}  h_E^{-\onehalf}
      \| \varphi_\vsig - \Igrad \varphi_\vsig \|_{L^2(E)}
      \; h_E^{\onehalf}
      \Big| \jump{ w_h \vec{n}} \Big|_{H^1(E)}
      \lesssim
      \|\vsig\|_{H(\div,\om)}
      \bigg( \sum_{E \in \mcE}  h_E
      \; \Big| \jump{ w_h \vec{n}} \Big|_{H^1(E)}^2
      \bigg)^{\onehalf}.
  \end{align}
  Here, in the final line, we have also used the Cauchy--Schwarz inequality and \eqref{eq:HdivInterpolant}.
  The term $t_2$ can be estimated
  similarly:
  \begin{align}
    t_2
    & = 
      \sum_{E \in \mcE} ( \vpsi_\vsig - \Idiv \vpsi_\vsig, \jump{w_h \bmn} )_E
      \le
      \sum_{E \in \mcE} h_E^{-\onehalf}
      \| \vpsi_\vsig - \Idiv \vpsi_\vsig \|_{L^2(E)}
      h_E^{\onehalf} \Big\| \jump{w_h \bmn} \Big\|_{L^2(E)}
    \\
    & \lesssim
      \| \vsig \|_{H(\div, \om)}
      \bigg(
      \sum_{E \in \mcE}
      h_E \Big\| \jump{w_h \bmn} \Big\|_{L^2(E)}^2
      \bigg)^{\onehalf}.
  \end{align}
  Thus, the proof of the remaining inequality is complete:
  \reqnomode
  \begin{equation}
    \frac{\ip{ \vsig\cdot \bmn, w_h}_h^2 }{ \quad \| \vsig \|_{H(\div, \om)}^2}
    =
    \frac{(t_1 + t_2)^2}{ \quad \| \vsig \|_{H(\div, \om)}^2}
    \lesssim
    \sum_{E \in \mcE}
      h_E \Big\| \jump{w_h \bmn} \Big\|_{L^2(E)}^2.
  \tag*{\qedhere}
  \end{equation}
  \leqnomode
\end{proof}

\begin{proof}[Proof of \cref{thm:PossionDualDPGAPosterioriErrrorEstimator}]
  The DPG* solution $(\vec{p}_h, v_h)$ satisfies the equations
  of~\eqref{eq:dpgstar-poisson} for all $(\vec{\tau}, \nu) \in \scV_h$
  and all $(\vsig, \mu, \hat{\sigma}_n. \hat{\mu}) \in \scU_h$. In
  particular, \eqref{eq:dpgstar-poisson-b} implies that
  $\ip{ \vec{p}_h\cdot \bmn, \hat{\mu}}_h = \ip{ v_h, \hat{\sigma}_n}_h
  =0$. Hence the conditions of \cref{lem:FluxBounds,lem:TraceBounds} are satisfied. The conclusions of these
  lemmas show that the $\eta$ in
  \cref{thm:DualDPGAPosterioriErrrorEstimator-0} satisfies
  \[
    \eta( (\vec{p}_h, v_h) ) \eqsim
    \sum_{E\in\mcE_\interior}
    h_E
    \big\|\jump{\vec{p}_{h} \cdot \vec n}\big\|_{L^2(E)}^2
    +
    \sum_{E\in\mcE}
    h_E
    \big\|\jump{v_h \vec n}\big\|_{H^1(E)}^2
  \]
  and, moreover, the last term may be replaced by
  $ h_E^{-1} \big\|\jump{v_h \vec n}\big\|_{L^2(E)}^2$, if we
  please. Hence an application of
  \cref{thm:DualDPGAPosterioriErrrorEstimator-0} completes the
  proof.
\end{proof}


\section{Numerical experiments} 
\label{sec:numerical_experiments}

In order to verify the mathematical theory developed above, we conducted several standard numerical verification experiments using two finite element software packages which have been used extensively for implementing DPG methods.
In our first set of experiments, we used Camellia \cite{roberts2014camellia,CamelliaManual}, a user-friendly C++ toolbox developed by Nathan V. Roberts which itself relies on Sandia’s Trilinos library of packages \cite{1089021}.
Specifically, Camellia was used for the \textit{a priori} convergence rate verification on the model square domain $\Omega_\square = [0,1]^2$ reported on in \cref{sub:square_domain_dirichlet}.
In our second set of experiments, we used $hp$2D, a sophisticated suite of Fortan routines with support for 2D local hierarchical and anisotropic $h$- and $p$-refinements on hybrid meshes \cite{hpbook} and corresponding oriented embedded shape functions for both quadrilateral and triangular elements in each of the canonical 2D de Rham sequence energy spaces \cite{Fuentes2015}:
\begin{equation}
  H^1(K) \xrightarrow{\,\,\grad\,\,} \bmH(\rot,K) \xrightarrow{\,\,\rot\,\,} L^2(K)
  \qquad
  \text{and}
  \qquad
  H^1(K) \xrightarrow{\,\,\curl\,\,} \bmH(\div,K) \xrightarrow{\,\,\div\,\,} L^2(K)
  \,.
\label{eq:ExactSequence}
\end{equation}
$hp$2D was used to implement a simple $hp$-adaptive algorithm for a singular solution to Poisson's equation on the canonical L-shaped domain, $\Omega_{\Lshaped{}} = (-1,1)^2 \setminus [0,1]\times[-1,0]$; see \cref{sub:l_shaped_domain}.
This experiment parallels a similar study with the analogous DPG method in \cite{demkowicz2011analysis}.

\subsection{Set-up} 
\label{sub:set_up}
Let $\Omega\in\{\Omega_\square,\Omega_{\Lshaped{}}\}$ and let $\Gamma_\Dirichlet,\Gamma_\Neumann$ be disjoint and relatively open subsets comprising $\bdry \Omega$; $\Gamma_\Dirichlet\cap\Gamma_\Neumann = \emptyset$, $\overline{\Gamma_\Dirichlet\cup\Gamma_\Neumann} = \bdry \Omega$.
All of our experiments investigate some form of Poisson's equation:
\begin{equation}
  \left\{
    \begin{aligned}
      -\Delta v &= f &&\text{in } \Omega\,,\\
      v &= v_0 &&\text{on } \Gamma_\Dirichlet\,,\\
      \frac{\partial v}{\partial n} &= p_n &&\text{on } \Gamma_\Neumann\,,
    \end{aligned}
  \right.
\label{eq:ModelProblem}
\end{equation}
where the load $f\in L^2(\Omega)$ and the boundary data $v_0$ and $p_n$ are appropriately smooth.

As before, let $\mesh$ denote the mesh subordinate to $\Omega$ and let $\mcE$ denote the corresponding collection of edges.
In each of our experiments, we only considered piecewise-affine two-dimensional domains $\Omega\in\{\Omega_\square,\Omega_{\Lshaped{}}\}$ subdivided into quadtree meshes consisting of either fully geometrically conforming or 1-irregular quadrilateral elements $K\in\mesh$.\footnote{Although many of the preceding results are proven only for triangular meshes, the numerical experiments documented in this section verify alternative results in the setting of quadrilateral elements, which we understand to be similar.}
During stiffness matrix assembly, the degrees of freedom of every element edge with a hanging node was constrained by its common edge.
Alternatively, because of the ultraweak variational formulation we considered, we could have incorporated each edge independently \cite{roberts2014camellia,vaziri2017high}.

For each quadrilateral element $K\in\mesh$, we associated a unique (anisotropic) polynomial order $p_K,q_K\ge 1$, respectively.
Each associated polynomial order can be naturally related to a $(p_K,q_K)$-order conforming finite element de Rham sequence.
For instance, begin with the standard N\'ed\'elec spaces of the first type,
\begin{equation}
  \mcQ^{p_K,q_K}(K) \xrightarrow{\,\,\curl\,\,} \mcQ^{p_K,q_K-1}\times\mcQ^{p_K-1,q_K}(K) \xrightarrow{\,\,\div\,\,} \mcQ^{p_K-1,q_K-1}(K)
  \,,
\end{equation}
where $\mcQ^{p_K,q_K}(K)$ is the space of bivariate polynomials over $K$ with degree at most $p_K$ horizontally and $q_K$ vertically.
Now, consider the mesh-dependent sequence ${W_{hp} \xrightarrow{\,\,\curl\,\,} \bmV_{\!hp} \xrightarrow{\,\,\div\,\,} Y_{hp}}$, where
  \begin{align}
    W_{hp}
    &=
    \{w\in H^1_0(\Omega) : w|_K \in\mcQ^{p_K,q_K}(K)~\forall K\in\mesh \}
    \,,
    \\
    \bmV_{\!hp}
    &=
    \{\bmq\in \bmH(\div,\Omega) : \bmq|_K \in\mcQ^{p_K,q_K-1}(K)\times\mcQ^{p_K-1,q_K}(K)~\forall K\in\mesh \}
    \,,
    \\
    Y_{hp}
    &=
    \{w\in L^2(\Omega) : w|_K \in\mcQ^{p_K-1,q_K-1}(K)~\forall K\in\mesh \}
    \,.
  \end{align}
We define the (isotropic) uniform-$p$ trial space to be $\scU_{h} = Y_{hp} \times Y_{hp}^2 \times \tr(W_{hp}) \times \tr_n(\bmV_{\!hp})$, where $p_K=q_K=p$ is fixed for all $K\in\mesh$.
Similarly, the corresponding (anisotropic) $hp$ trial space is defined $\scU_{h} = Y_{hp} \times Y_{hp}^2 \times \tr(W_{hp}) \times \tr_n(\bmV_{\!hp})$, where $p_K$ and $q_K$ are allowed to vary freely throughout the mesh.
With the latter definition, notice that the polynomial order of an $hp$ interface function, when restricted to a single shared edge $E\in\mcE$, $E= \bigcap_{\overbar{K}\cap E\neq\emptyset}\overbar{K}$, will naturally be restricted by the lowest polynomial order of all elements $\overbar{K}\cap E\neq\emptyset$ sharing the edge.\footnote{Such a mesh obeys the so-called \emph{minimum rule}.}

For the test functions, define the spaces
\begin{align}
  \tilde{W}_{hp,\dd\!p}
  &=
  \{v\in H^1(\mesh) : v|_K \in\mcQ^{p_K+\dd\!p,q_K+\dd\!p}(K)~\forall K\in\mesh \}
  \,,
  \\
  \tilde{\bmV}_{hp,\dd\!p}
  &=
  \{\bmq\in \bmH(\div,\mesh) : \bmq|_K \in\mcQ^{p_K+\dd\!p,q_K+\dd\!p-1}(K)\times\mcQ^{p_K+\dd\!p-1,q_K+\dd\!p}(K)~\forall K\in\mesh \}
  \,.
\end{align}
In all of our numerical experiments, we used $\scV_h = \tilde{W}_{hp,\dd\!p} \times \tilde{\bmV}_{hp,\dd\!p}$ where $\dd\!p \in\{0,1,2\}$.


\subsection{Adaptive mesh refinement} 
\label{sub:adaptive_mesh_refinement}

In our experiments with $h$- and $hp$-adaptive mesh refinement, we used a standard isotropic $h$-subdivision rule.
Namely, at each refinement step, each element marked for $h$-refinement was uniformly subdivided into four equal-order quadrilateral elements.
Afterward, a standard so-called ``mesh closure'' algorithm was called to induce a small number of additional isotropic $h$-subdivisions of neighboring elements in order to ensure 1-irregularity of the mesh.
Alternatively, at each refinement step, the polynomial order of any $p$-refinement marked element was isotropically incremented by one, $(p_K,q_K)\mapsto(p_K+1,q_K+1)$, and then the order of all elements neighboring a $p$-refinement marked element was also isotropically incremented by one.

Recall that $\mcE_{K} = \{ E\in \mcE : \mathrm{meas}(\bdry K \cap E)\neq\emptyset \}$ and define $\mcE_{K,\interior} = \mcE_{K} \cap \mcE_\interior$.
Recalling the global error estimator $\eta_1(\vec{v}_h)$ appearing in \cref{thm:PossionDualDPGAPosterioriErrrorEstimator}, define a \emph{refinement indicator} $\eta_K\in \R_{\ge 0}$ for each $K\in\Omega_h$, \textit{viz.},
\begin{equation}
  \eta_K
  =
  \bigg(
  \|\mcL\vec{v}_h-f\|_{L^2(K)}^2 + \sum_{E\in\mcE_{K,\interior}} h_E
      \big\|\jump{\vec{p}_{h} \cdot \vec n}\big\|_{L^2(E)}^2 + \sum_{E\in\mcE_{K}} h_E \big\|\jump{v_h \vec n}\big\|_{H^1(E)}^2
  \bigg)^{\onehalf}
  \,.
\end{equation}
In element marking, we followed the so-called ``greedy'' algorithm.
That is, at each refinement step, all elements $K\in\mesh$ whose refinement indicator $\eta_K$ was above $50\%$ of the maximum over all elements in the mesh, $\eta_{\mathrm{max}} = \max_{K\in\mesh} \eta_K$, was marked for refinement.
In the case of what we call $h$-adaptive mesh refinement, every marked element was $h$-refined, as described above, i.e. no elements were $p$-refined.
Alternatively, in the case of $hp$-adaptive mesh refinement, a common flagging strategy \cite{ainsworth1997aspects} was used to decide whether to $h$ or $p$ refine; see \cref{sub:l_shaped_domain}.


\subsection{Pure Dirichlet boundary conditions on a square domain} 
\label{sub:square_domain_dirichlet}

Recall \cref{eq:ModelProblem}.
In this first example, $\Omega = \Omega_\square$ and $\Gamma_\Dirichlet = \bdry\Omega$.
We considered two seemingly benign cases for the loads: (i) $f=2\pi^2\sin(\pi x)\sin(\pi y)$ and $v_0 = 0$; and (ii) $f=0$ and  $v_0 = 1$.
In both cases, the exact solution is infinitely smooth.
Indeed, in case (i), $v=\sin(\pi x)\sin(\pi y)$ and, in case (ii), $v=1$.

Recall from \cref{thm:apriori} that the best approximation error of a DPG* method involves the Lagrange multiplier $\blambda = (\bzeta, \lambda, \hat{\zeta}_n, \hat\lambda)$ as well as the DPG* solution variable $\vec{v} = (\bmp,v)$.
Assume that $v$ is smooth.
With the norm $\|(\btau,\nu)\|_\scV^2 = \|\btau\|_{\bmH(\div,\mesh)}^2 + \|\nu\|_{H^1(\mesh)}^2$, $\lambda$ solves
\begin{equation}
  \left\{
    \begin{aligned}
      -\Delta \lambda &= g &&\text{in } \Omega\,,\\
      \lambda &= 0 &&\text{on } \bdry\Omega
      \,,
    \end{aligned}
  \right.
\label{eq:AuxiliaryProblem}
\end{equation}
where $g = v - 2\Delta v + \Delta^2 v$.
Indeed, recall \cref{eq:LMSolutions} and observe that $- \Delta \lambda = -\Delta f - \Delta e = \Delta (\Delta v)  + (v + 2f) =  v - 2\Delta v + \Delta^2 v$. Here, we have also used that $-\Delta v = f$ and $-\Delta e = v + 2f$.
In case (i), $g = (1+4\pi^2-4\pi^4)\sin(\pi x)\sin(\pi y)$, meanwhile, in case (ii), $g = 1$.
Notice that $g\in C^\infty(\overline{\Omega})$ in both cases.

In the first case, $\lambda$ can easily be shown to be a constant scalar multiple of $\sin(\pi x)\sin(\pi y)$ and so $\lambda \in C^\infty(\overline{\Omega})$ is infinitely smooth.
Therefore, by \cref{cor:DPG*PoissonConvergenceRates}, the convergence rate of the DPG* method under uniform $h$-refinement will be limited only by the underlying de Rham sequence polynomial order $p$.
Indeed, \cref{fig:hUnifSinA} demonstrates the convergence of the corresponding discrete solution $\bmv_h = (\bmp_h,v_h)$ to the exact solution, $\bmv = (\grad v,v)$, measured in the full test norm above, starting with an single-element mesh with (isotropic) polynomial order $p_K=q_K = p\in\{1,2,3,4\}$.
\cref{fig:hUnifSinB} presents the convergence of only the solution variable $v_h$, measured in the $L^2(\Omega)$-norm.
Although both figures correspond only to a test space enrichment of $\dd\!p=1$, similar results were observed for each choice $\dd\!p\in\{0,1,2\}$.

\begin{figure}[ht!]
  \centering
  \begin{subfigure}[b]{0.5\textwidth}
    \centering
    \includegraphics[width=7cm]{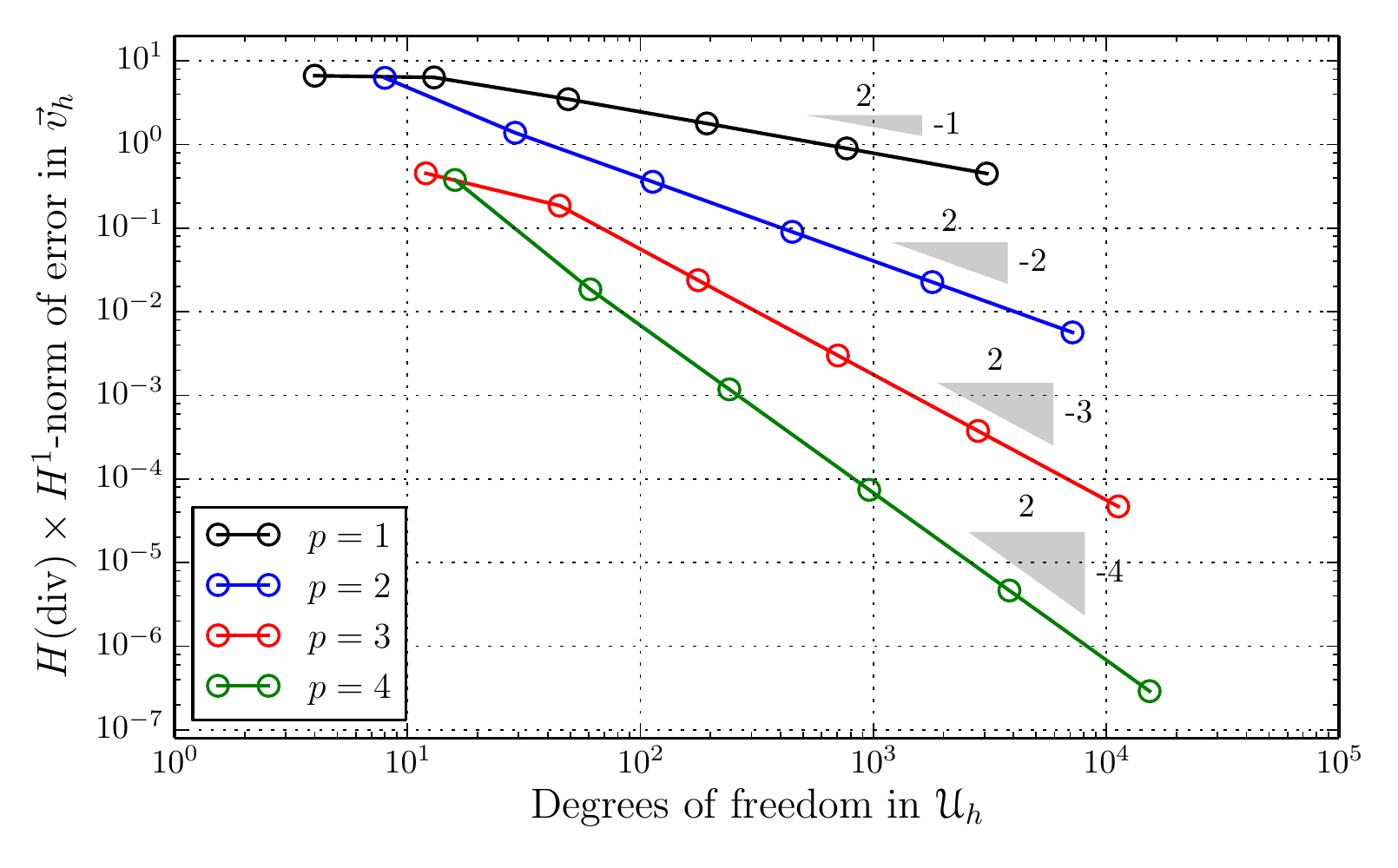}
    \caption{Optimal rates.}
  \label{fig:hUnifSinA}
  \end{subfigure}%
  ~ 
  \begin{subfigure}[b]{0.5\textwidth}
    \centering
    \includegraphics[width=7cm]{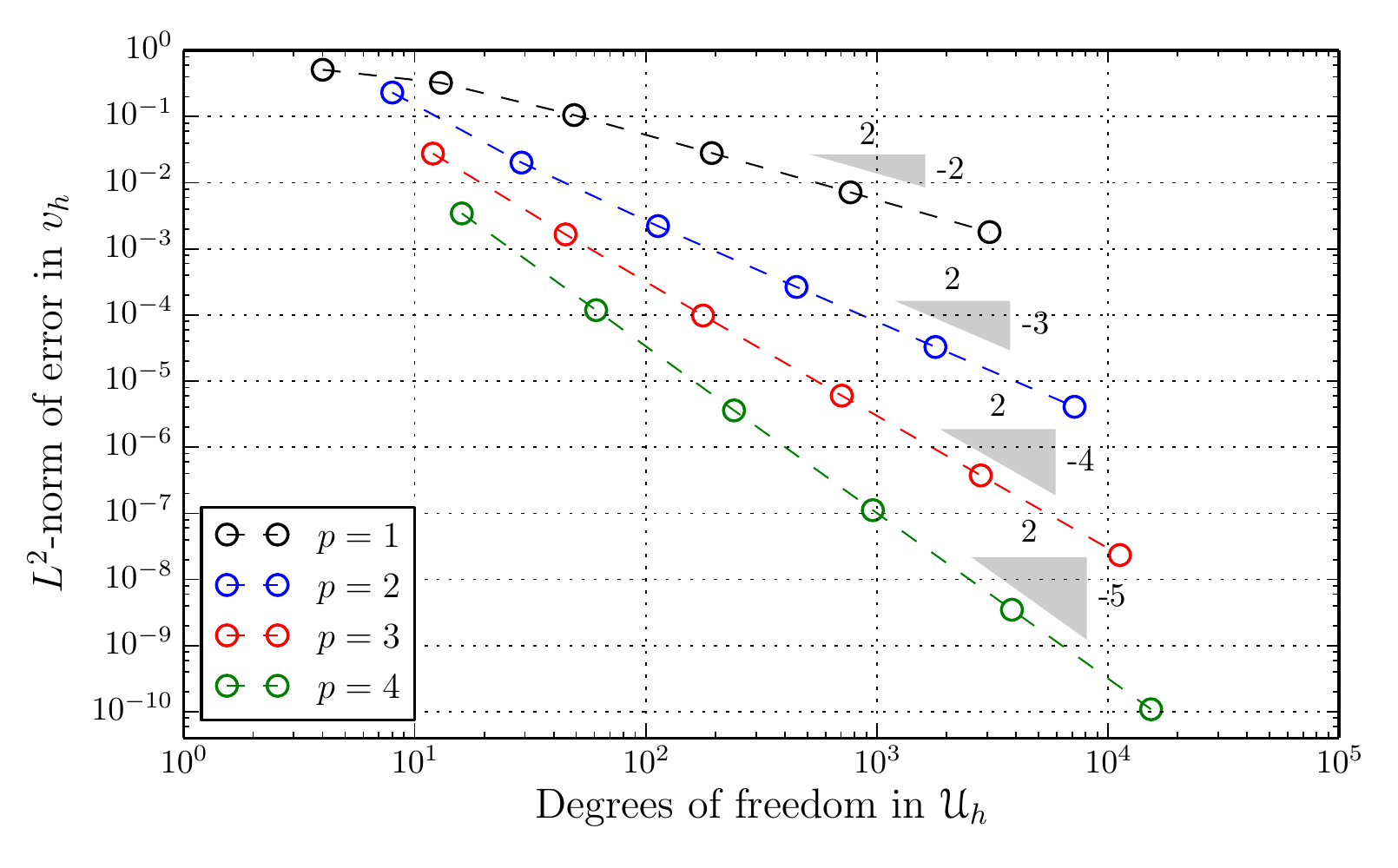}
    \caption{Optimal rates in the $L^2$-norm.}
  \label{fig:hUnifSinB}
  \end{subfigure}
  \caption{Convergence under $h$-uniform mesh refinements with the manufactured solution $v(x,y) = \sin(\pi x)\sin(\pi y)$. (Here, $\dd\!p=1$.)}
  \label{fig:hUnifSin}
\end{figure}

In the second case, because of the irregularity of the boundary, the Lagrange multiplier $\lambda \notin C^\infty(\overline{\Omega})$ is \emph{not} infinitely smooth.\footnote{Standard elliptic regularity theory can be used to show that $\lambda \in C^\infty(\Omega)$ is, however, infinitely smooth in the interior of the domain \cite{MR2597943}.}
Therefore, with this problem, the DPG* method experiences rate-limited convergence under uniform $h$-refinements.
This is evidenced by \Cref{fig:hUnifOne}.
However, as demonstrated by \Cref{fig:hAdaptOne}, using the greedy $h$-refinement strategy from \cref{sub:adaptive_mesh_refinement}, optimal convergence rates can still be recovered through adaptive mesh refinement.
See \cref{fig:SquareAdaptiveMesh} for a visual depiction of the solution of the corresponding auxiliary problem~\cref{eq:AuxiliaryProblem} as well as the corresponding adaptively refined mesh.

\begin{figure}
  \centering
  \begin{subfigure}[b]{0.5\textwidth}
    \centering
    \includegraphics[width=7cm]{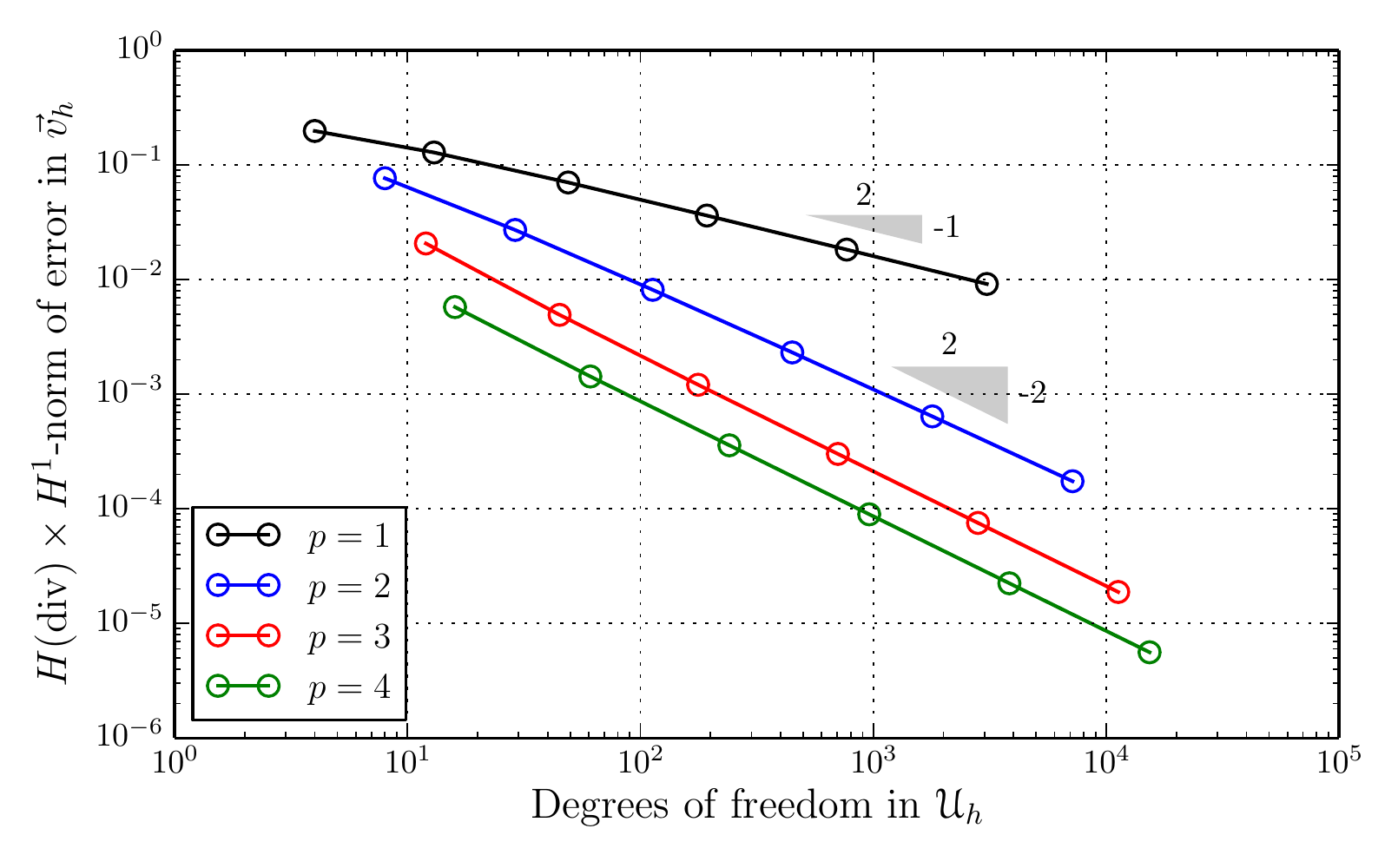}
    \caption{Limited optimal rates.}
    \label{fig:hUnifOneA}
  \end{subfigure}%
  ~ 
  \begin{subfigure}[b]{0.5\textwidth}
    \centering
    \includegraphics[width=7cm]{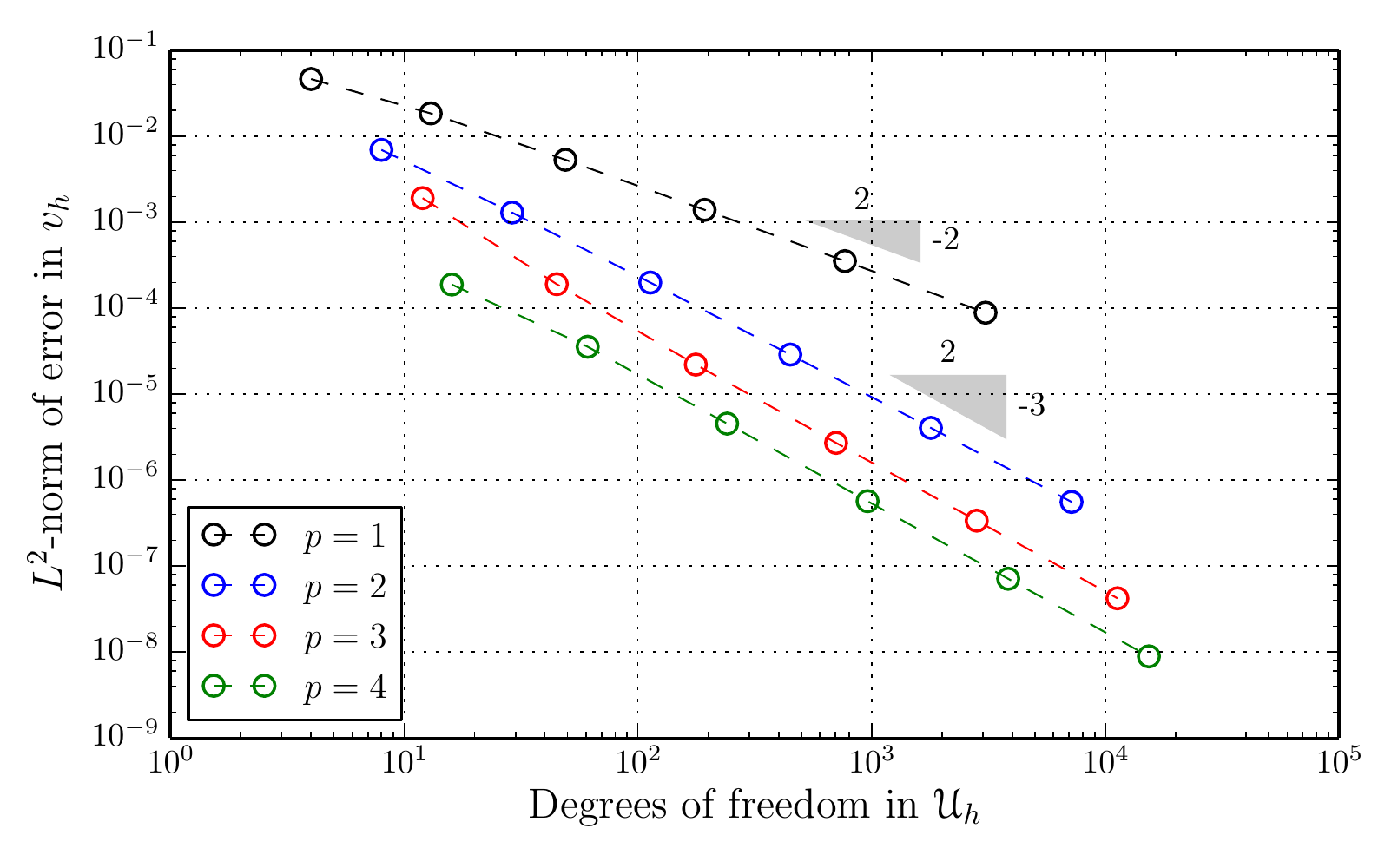}
    \caption{Limited optimal rates in the $L^2$-norm.}
    \label{fig:hUnifOneB}
  \end{subfigure}
  \caption{Convergence under $h$-uniform mesh refinements with the manufactured solution $v(x,y) = 1$. (Here, $\dd\!p=1$.)}
  \label{fig:hUnifOne}
\end{figure}

\begin{figure}
  \centering
  \begin{subfigure}[b]{0.5\textwidth}
    \centering
    \includegraphics[width=7cm]{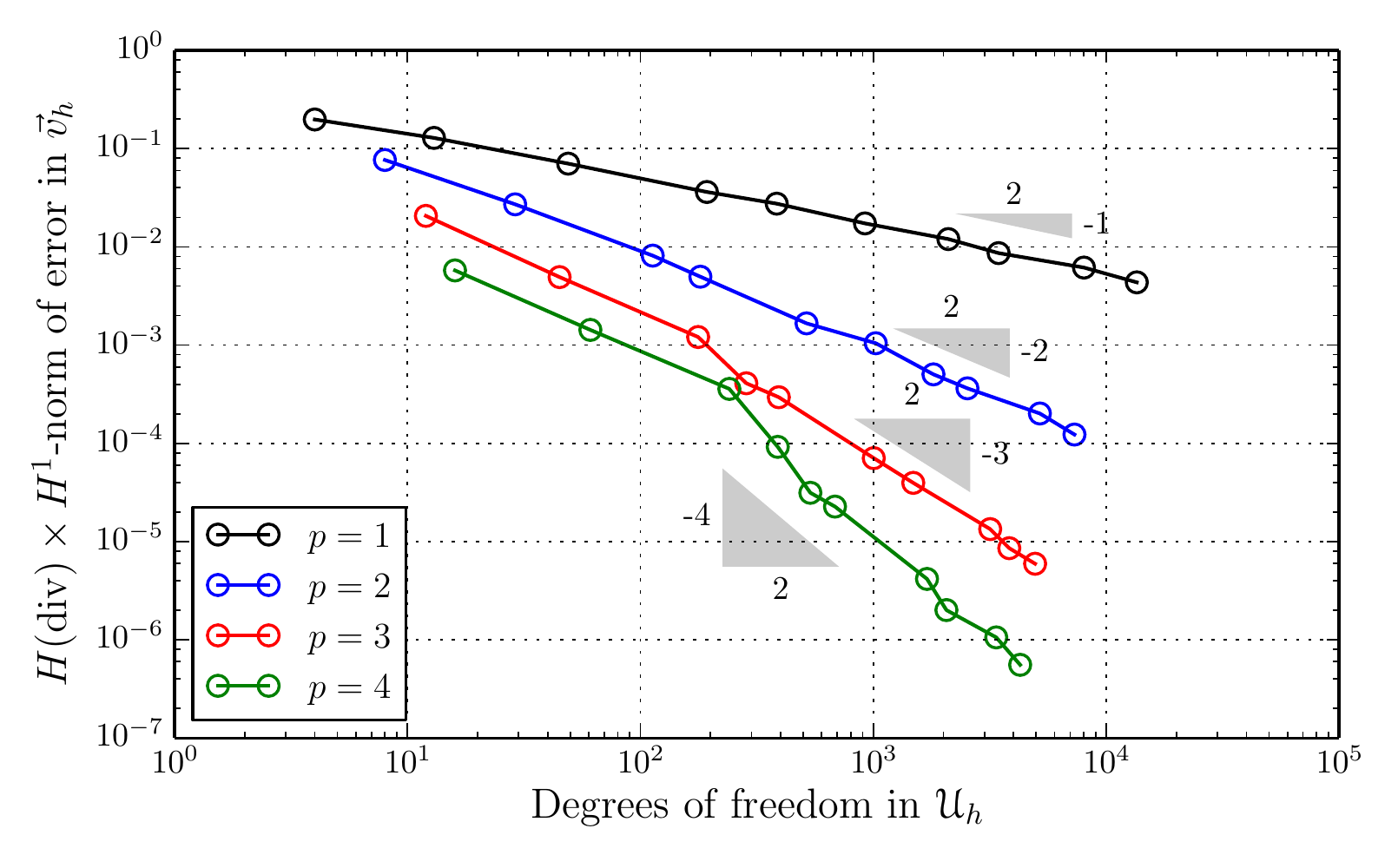}
    \caption{Recovered optimal rates.}
  \end{subfigure}%
  ~ 
  \begin{subfigure}[b]{0.5\textwidth}
    \centering
    \includegraphics[width=7cm]{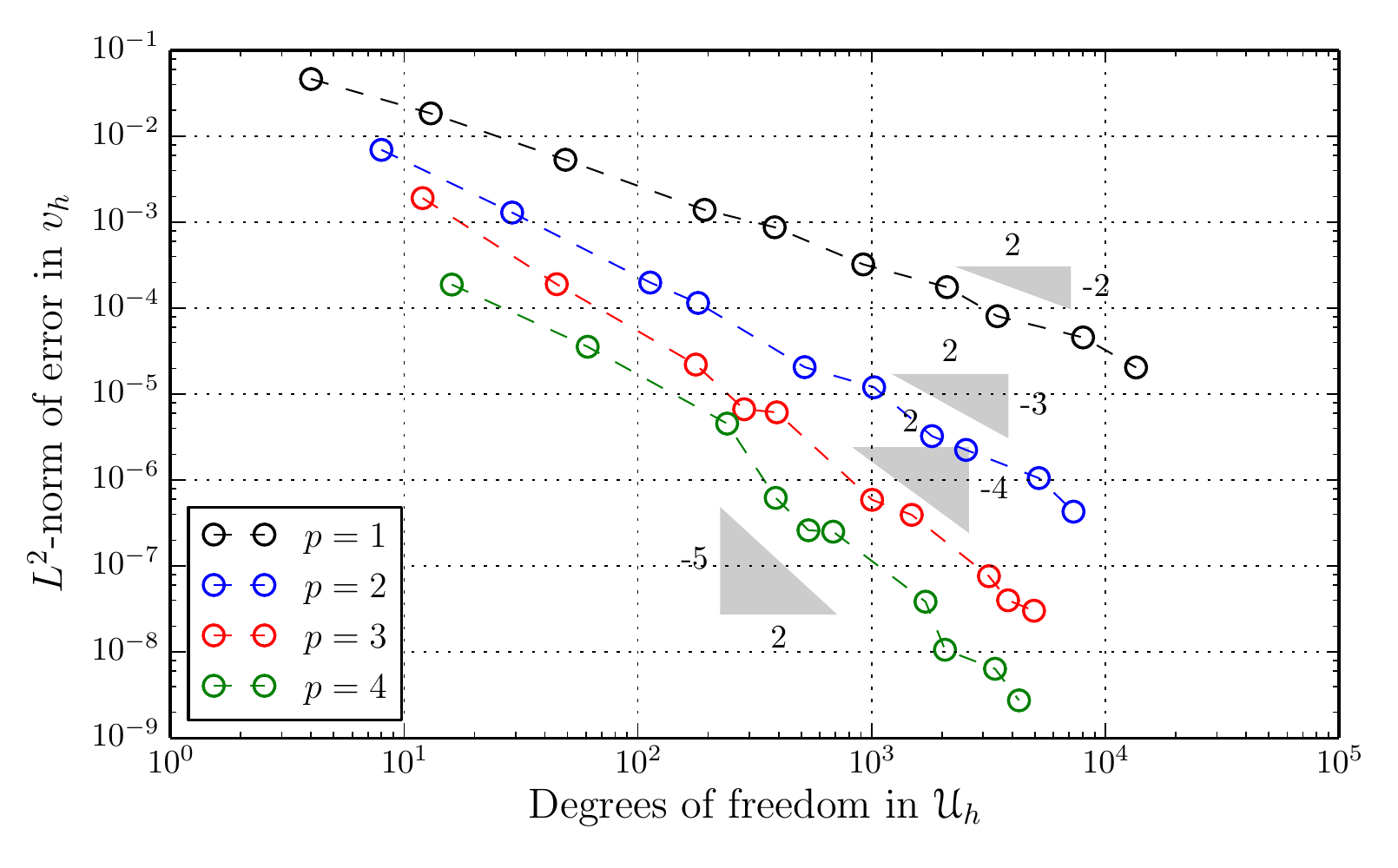}
    \caption{Recovered optimal rates in the $L^2$-norm.}
  \end{subfigure}
  \caption{Convergence under $h$-adaptive mesh refinements with the manufactured solution $v(x,y) = 1$. (Here, $\dd\!p=1$.)}
  \label{fig:hAdaptOne}
\end{figure}

\begin{figure}
  \centering
  \includegraphics[width=11cm]{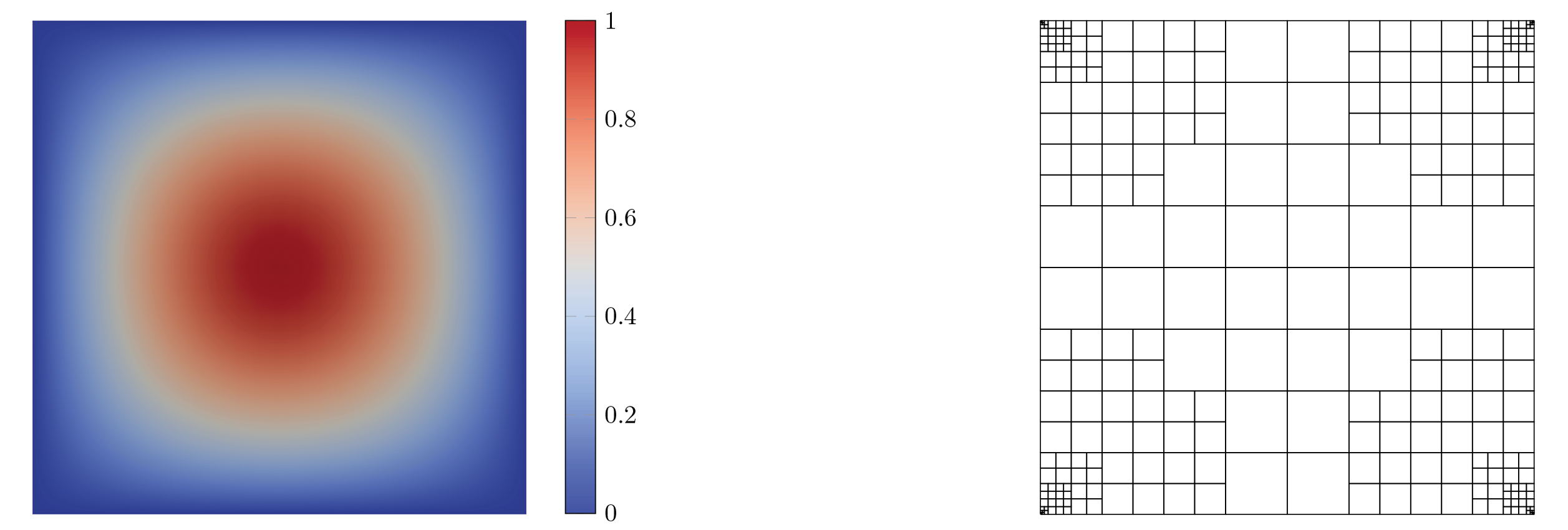}
  \caption{Left: The DPG* solution variable $\lambda$ when $v=1$. (Color scale represents the solution values.) Right: The corresponding quadtree mesh coming from the $h$-adaptive algorithm after ten refinements. (Here, $p=4$ and $\dd\!p=1$.)}
  \label{fig:SquareAdaptiveMesh}
\end{figure}

\begin{remark}
  
  Previously in this subsection, we remarked that similar results were observed for each test space enrichment parameter $\dd\!p\in\{0,1,2\}$ that we chose in our numerical experiments analyzing case (i).
  Similary, in case (ii), the behaviour documented above was nearly indistinguishable for each $\dd\!p\ge 1$.
  However, when $\dd\!p=0$ we observed unexpected effects which we repeatedly verified with independent implementations of the method.
  Indeed, starting from a mesh consisting of a single square element of order $p$ and subsequently performing uniform $h$-refinements, the exact solution $v = 1$ was repeatedly reproduced up to machine zero, no matter the polynomial order $p\in\{1,2,3,4\}$ considered.
  In testing more complicated manufactured solutions (not documented here) which also feature a singular Lagrange multiplier $\lambda$, we discovered superconvergence effects from this choice of enrichment parameter.
  Indeed, in our numerous additional verification experiments with $\dd\!p$ set to zero, the method overcame the rate-limited behavior illustrated in \cref{fig:hUnifOne}.
  This peculiar superconvergence artifact can not be explained by the theory presented in this paper.
  Notably, this artifact also agrees with previous results seen with a DPG* method for acoustics which can be found in the original technical report on the method \cite{Keith2017DPGstar} (which portions of this text are based off of) and clearly warrants further analysis.

\end{remark}

\subsection{Mixed boundary conditions on an L-shaped domain} 
\label{sub:l_shaped_domain}

Again, recall \cref{eq:ModelProblem}.
In this final example, set $\Omega = \Omega_{\Lshaped{}}$, $\Gamma_\Dirichlet = [0,1)\times\{0\} \cup \{0\}\times[0,-1)$, and $v_0 = 0$.
Additionally, set $p_n$ to be the normal derivative of the exact solution $v(r,\theta) = r^{2/3}\sin(\frac{2}{3}\theta)$.
For this problem, it is well known that the solution $v \in H^{1+s}(\Omega)$, for all $s<2/3$.

In each of our experiments, we began with a single three-element mesh composed of congruent squares and uniform order $p_K = q_K =2$ and $\dd\!p=1$ in all three elements.
\cref{fig:hpAdaptivity} demonstrates the convergence of the solution error we witnessed under $h$-uniform, $h$-adaptive (as described above), and $hp$-adaptive refinements using a flagging strategy where all marked element adjacent to the origin (i.e. the singular point) are $h$-refined and all other marked elements are $p$-refined.
As shown in \cref{fig:hpError}, the error estimator $\eta(\bmv_h)$ generally overestimated the solution error and the dependence upon $\dd\! p$ was not seen to be roundly significant.
\cref{fig:LshapedDomain} depicts both the computed solution and the $hp$ mesh mesh after fifteen refinement steps.



\begin{figure}[ht!]
  \centering
  \begin{subfigure}[b]{0.5\textwidth}
    \centering
    \includegraphics[width=7cm]{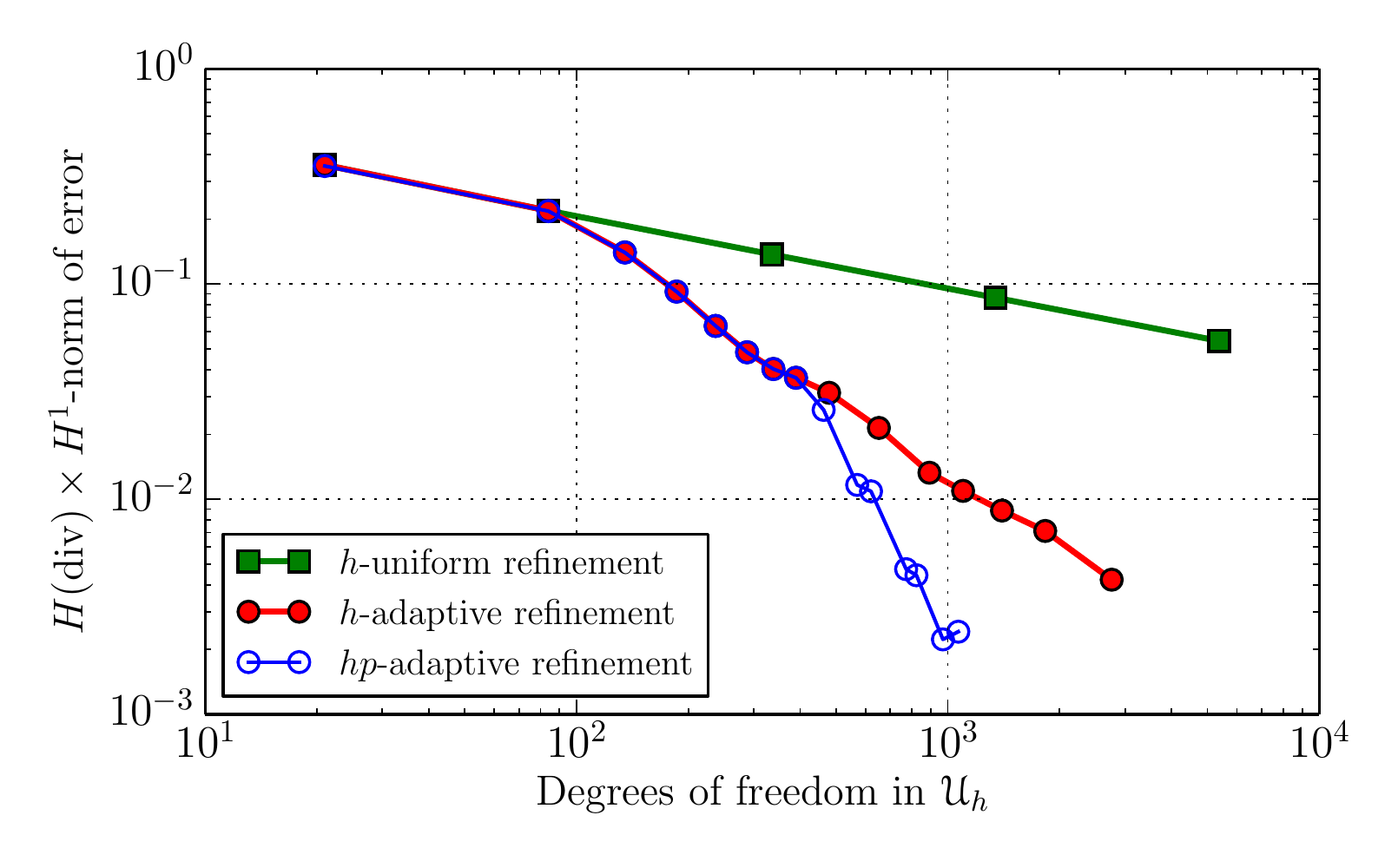}
    \caption{Effect of refinement strategies.}
    \label{fig:hpAdaptivity}
  \end{subfigure}%
  ~ 
  \begin{subfigure}[b]{0.5\textwidth}
    \centering
    \includegraphics[width=7cm]{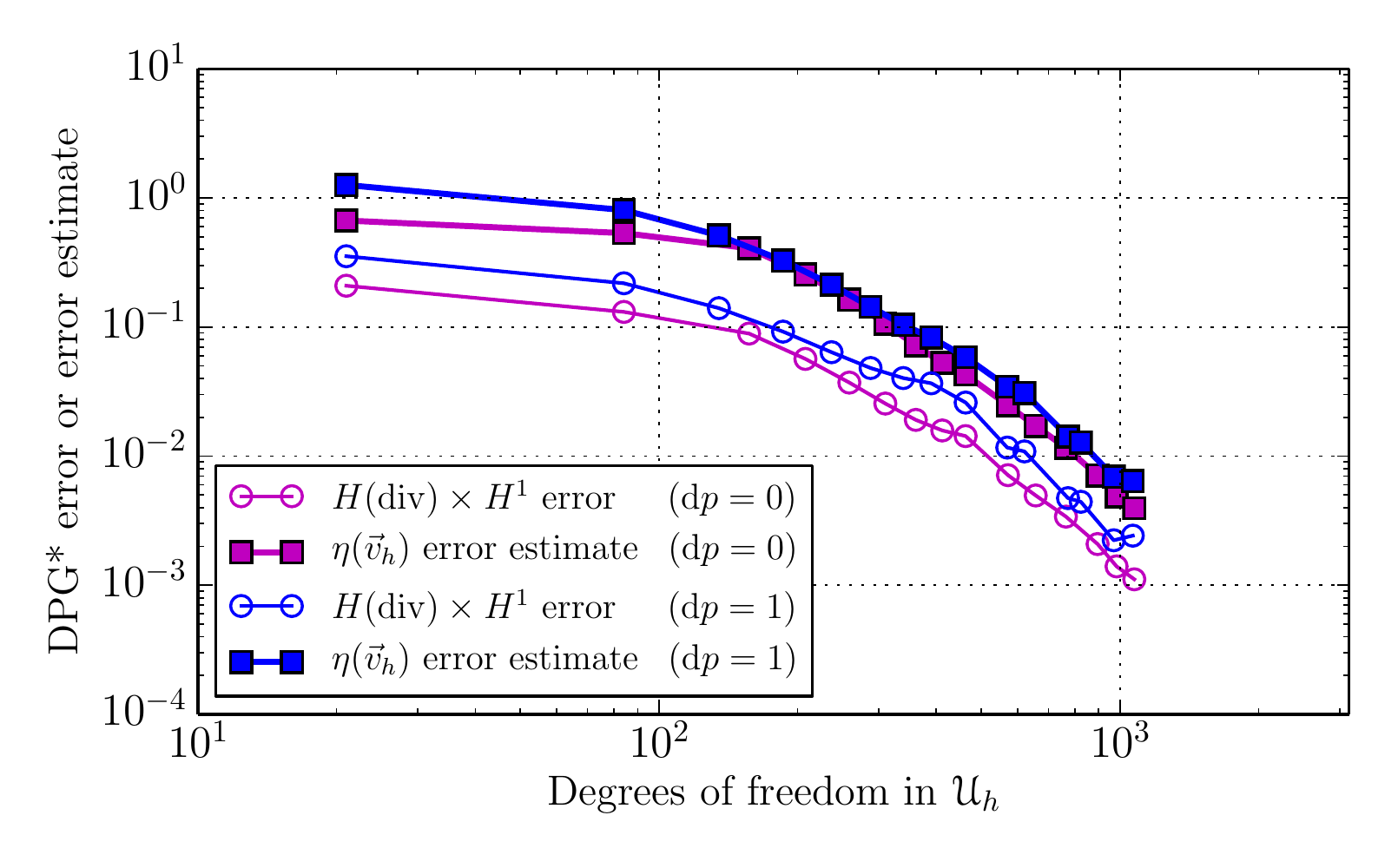}
    \caption{Effect of $\dd\!p$.}
    \label{fig:hpError}
  \end{subfigure}
  \caption{ \subref{fig:hpAdaptivity}\,: Comparison of the convergence of various refinement strategies when $\dd\!p =1$. \subref{fig:hpError}\,: Convergence of the error in the DPG* solution variable $\bmv_h$ and the error estimator $\eta(\bmv_h)$ for two values of $\dd\!p$ with the $hp$-adaptive algorithm.}
  \label{fig:hpAdaptOne}
\end{figure}

\begin{figure}[ht!]
    \centering
    \includegraphics[width=11cm]{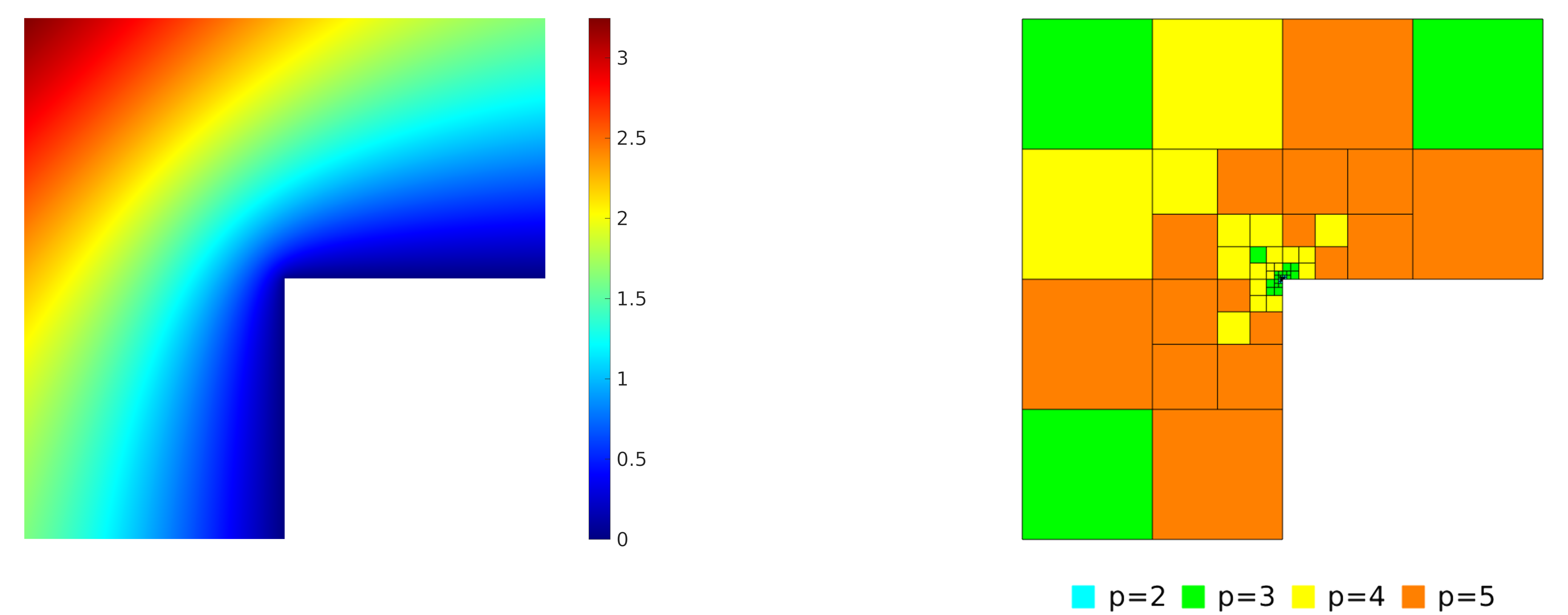}
    \caption{Left: The DPG* solution $v$. (Color scale represents solution values.) Right: The corresponding $hp$ quadtree mesh found by the $hp$-adaptive algorithm after fifteen refinements. (Colors represent polynomial degrees $p$.)}
    \label{fig:LshapedDomain}
\end{figure}


\phantomsection
\bibliographystyle{siam}
\bibliography{main}

\end{document}

%% file: packages.tex
\usepackage[top=1.4in,bottom=1.275in,left=1.2in,right=1.2in]{geometry}

\usepackage[utf8]{inputenc}
\usepackage[T1]{fontenc}
\usepackage[english]{babel}
\usepackage[activate={true,nocompatibility},final,tracking=true,kerning=true,spacing=true,factor=1100,stretch=10,shrink=10,]{microtype}
\microtypecontext{spacing=nonfrench}
\usepackage{placeins}

\PassOptionsToPackage{usenames,dvipsnames}{xcolor}

\usepackage{amsmath}
\usepackage{amssymb}
\usepackage{amsthm}
\usepackage{amscd}

\usepackage{bm}
\usepackage{mathrsfs}
\usepackage{mathtools}
\usepackage{centernot}
\usepackage{xfrac}
\usepackage[only,llbracket,rrbracket]{stmaryrd}
\SetSymbolFont{stmry}{bold}{U}{stmry}{m}{n}
\usepackage{calligra}
\DeclareMathAlphabet{\mathcalligra}{T1}{calligra}{m}{n}
\DeclareFontShape{T1}{calligra}{m}{n}{<->s*[1.1]callig15}{}
\DeclareFontFamily{OT1}{pzc}{}
\DeclareFontShape{OT1}{pzc}{m}{it}{<-> s * [1.0] pzcmi7t}{}
\DeclareMathAlphabet{\mathpzc}{OT1}{pzc}{m}{it}

\usepackage[mathscr]{eucal}

\usepackage{graphicx}
\usepackage[font=small,figurewithin=section]{caption}
\usepackage[labelformat=simple]{subcaption}

\usepackage{float}
\usepackage[hypertexnames=false]{hyperref}
\usepackage[capitalize,nameinlink,noabbrev]{cleveref}

\crefname{equation}{}{}
\usepackage{autonum}
\usepackage{tikz}
\usepackage{verbatim}

\definecolor{CeruleanRef}{RGB}{12,127,172}
\hypersetup{colorlinks=true,allcolors=CeruleanRef}
\hypersetup{pdfstartview=FitB,pdfpagemode=UseNone}

\makeatletter
\renewcommand*{\eqref}[1]{%
  \hyperref[{#1}]{\textup{\tagform@{\ref*{#1}}}}%
}
\makeatother

\newtheorem{theorem}{Theorem}[section]
\newtheorem{lemma}[theorem]{Lemma}
\newtheorem{corollary}[theorem]{Corollary}
\newtheorem{proposition}[theorem]{Proposition}

\theoremstyle{definition}

\newtheorem{example}[theorem]{Example}

\theoremstyle{remark}
\newtheorem{remark}[theorem]{Remark}

\usepackage{todonotes}

%% file: preamble.tex
\let\inf\relax \DeclareMathOperator*\inf{\vphantom{p}inf}
 
\let\max\relax \DeclareMathOperator*\max{\vphantom{p}max}
\let\subset\relax \DeclareMathOperator{\subset}{\subseteq}

\def\quotient#1#2{%
    \raise0.1pt\hbox{$#1$}/\lower0.6pt\hbox{$#2$}%
}

\makeatletter
\newcommand{\normm}{\@ifstar\@normms\@normm}
\newcommand{\@normms}[1]{%
  \left|\mkern-1.5mu\left|\mkern-1.5mu\left|
   #1
  \right|\mkern-1.5mu\right|\mkern-1.5mu\right|
}
\newcommand{\@normm}[2][]{%
  \mathopen{#1|\mkern-1.5mu#1|\mkern-1.5mu#1|}
  #2
  \mathclose{#1|\mkern-1.5mu#1|\mkern-1.5mu#1|}
}
\makeatother

\newsavebox\CBox

\makeatletter
\newcommand{\leqnomode}{\tagsleft@true\let\veqno\@@leqno}
\newcommand{\reqnomode}{\tagsleft@false\let\veqno\@@eqno}
\makeatother

\newcommand{\R}{\mathbb{R}} 
\newcommand{\N}{\mathbb{N}} 

\newcommand{\inv}{^{\raisebox{.2ex}{$\scriptscriptstyle-1$}}}

\newcommand{\T}{{}^{\raisebox{.2ex}{$\scriptscriptstyle\mathsf{T}$}}}

\newcommand{\mesh}{\Omega_h} 
\DeclareMathOperator{\dd}{d} 
\newcommand{\bdry}{\partial} 
\DeclareMathOperator{\tr}{\mathrm{tr}} 
\newcommand{\selement}{{\scriptscriptstyle K}} 
\newcommand{\sOmega}{{\scriptscriptstyle \Omega}} 
\newcommand{\kernel}{\mathrm{Null}}
\newcommand{\range}{\mathrm{Range}}

\newcommand{\bzeta}{\vec{\zeta}}

\newcommand{\blambda}{\vec{\lambda}}
\newcommand{\bmmu}{\vec{\mu}}
\newcommand{\bnu}{\vec{\nu}}

\newcommand{\bsigma}{\vec{\sigma}}
\newcommand{\btau}{\vec{\tau}}

\renewcommand\Omega{\varOmega}
\renewcommand\Gamma{\varGamma}
\renewcommand\Pi{\varPi}

\DeclareMathOperator{\Riesz}{\mathscr{R}} 

\DeclareMathOperator*{\argmin}{arg\,min}

\DeclareMathOperator{\grad}{grad}
\DeclareMathOperator{\rot}{rot}
\DeclareMathOperator{\curl}{curl}
\let\div\relax 
\DeclareMathOperator{\div}{div}

\makeatletter
\newcommand{\dashto}[1][2pt]{
  \settowidth{\@tempdima}{${}\rightarrow{}$}
  \makebox[\@tempdima]{${}\rightarrow{}$}
  \makebox[-\@tempdima]{\hspace{-0.1\@tempdima}\color{white}\rule[0.5ex]{#1}{1pt}}
  \makebox[\@tempdima]{}
  }
\makeatother
\let\tilde\widetilde
\let\hat\widehat

\newcommand{\jump}[1]{\llbracket #1 \rrbracket}

\newcommand{\upperone}[2]{\raisebox{-0.2\height}{$#1{}^1\!$}}
\newcommand{\upone}{\mathpalette\upperone\relax}
\newcommand{\lowertwo}[2]{\raisebox{-0.1\height}{$#1{}_{\!2}$}}

\newcommand{\lowtwo}{\mathpalette\lowertwo\relax}

\newcommand{\frachalf}[2]{\raisebox{0.0\height}{$#1{}\upone/\lowtwo$}}

\newcommand{\onehalf}{\mathpalette\frachalf\relax}
\newcommand{\minusonehalf}{{\raisebox{.4ex}{{\protect\scalebox{0.5}{$-$}}}}{\onehalf}}

\newcommand{\mcD}{\mathcal{D}}
\newcommand{\mcE}{\mathcal{E}}

\newcommand{\mcI}{\mathcal{I}}
\newcommand{\mcJ}{\mathcal{J}}

\newcommand{\mcL}{\mathcal{L}}

\newcommand{\mcQ}{\mathcal{Q}}
\newcommand{\mcR}{\mathcal{R}}

\newcommand{\mcT}{\mathcal{T}}

\newcommand{\sff}{\mathsf{f}}
\newcommand{\sfg}{\mathsf{g}}

\newcommand{\sfv}{\mathsf{v}}
\newcommand{\sfw}{\mathsf{w}}

\newcommand{\sfB}{\mathsf{B}}

\newcommand{\sfG}{\mathsf{G}}

\newcommand{\bmn}{\vec{n}}

\newcommand{\bmp}{\vec{p}}
\newcommand{\bmq}{\vec{q}}
\newcommand{\bmr}{\vec{r}}
\newcommand{\bmu}{\vec{u}}
\newcommand{\bmv}{\vec{v}}

\newcommand{\bmH}{H}

\newcommand{\bmV}{\bm{V}}

\newcommand{\scB}{\mathscr{B}}

\newcommand{\scR}{\mathscr{R}}

\newcommand{\scU}{\mathscr{U}}
\newcommand{\scV}{\mathscr{V}}

\newcommand{\fku}{u}
\newcommand{\fkv}{v}
\newcommand{\fkw}{w}

\newcommand{\calliD}{\hspace{-2.5pt}\mathcalligra{D}\hspace{2.9pt}}
\newcommand{\Eg}{\hspace{-.5pt}\mathcalligra{E}_{\mathrm{grad}}\hspace{.5pt}}
\newcommand{\calliE}{\Eg}
\newcommand{\Ed}{\hspace{-.5pt}\mathcalligra{E}_{\mathrm{div}}\hspace{.5pt}}

\newcommand{\overbar}[1]{\mkern 1.5mu\overline{\mkern-1.5mu#1\mkern-1.5mu}\mkern 1.5mu}
\newcommand{\shoverbar}[1]{\mkern 3.5mu\overline{\mkern-3.5mu#1\mkern-1.5mu}\mkern 1.5mu}

\newcommand{\Dirichlet}{{\protect\scalebox{0.6}{$\mathrm{D}$}}}
\newcommand{\Neumann}{{\protect\scalebox{0.6}{$\mathrm{N}$}}}
\newcommand{\interior}{{\protect\scalebox{0.6}{$\mathrm{int}$}}}
\newcommand{\Eint}{\mcE_{\interior}}

\DeclareSymbolFont{sfletters}{OML}{cmbrm}{m}{it}
\DeclareMathSymbol{\sfalpha}{\mathord}{sfletters}{"0B}
\DeclareMathSymbol{\sfbeta}{\mathord}{sfletters}{"0C}
\DeclareMathSymbol{\sfgamma}{\mathord}{sfletters}{"0D}
\DeclareMathSymbol{\sfdelta}{\mathord}{sfletters}{"0E}
\DeclareMathSymbol{\sfepsilon}{\mathord}{sfletters}{"0F}
\DeclareMathSymbol{\sfzeta}{\mathord}{sfletters}{"10}
\DeclareMathSymbol{\sfeta}{\mathord}{sfletters}{"11}
\DeclareMathSymbol{\sftheta}{\mathord}{sfletters}{"12}
\DeclareMathSymbol{\sfiota}{\mathord}{sfletters}{"13}
\DeclareMathSymbol{\sfkappa}{\mathord}{sfletters}{"14}
\DeclareMathSymbol{\sflambda}{\mathord}{sfletters}{"15}
\DeclareMathSymbol{\sfmu}{\mathord}{sfletters}{"16}
\DeclareMathSymbol{\sfnu}{\mathord}{sfletters}{"17}
\DeclareMathSymbol{\sfxi}{\mathord}{sfletters}{"18}
\DeclareMathSymbol{\sfpi}{\mathord}{sfletters}{"19}
\DeclareMathSymbol{\sfrho}{\mathord}{sfletters}{"1A}
\DeclareMathSymbol{\sfsigma}{\mathord}{sfletters}{"1B}
\DeclareMathSymbol{\sftau}{\mathord}{sfletters}{"1C}
\DeclareMathSymbol{\sfupsilon}{\mathord}{sfletters}{"1D}
\DeclareMathSymbol{\sfphi}{\mathord}{sfletters}{"1E}
\DeclareMathSymbol{\sfchi}{\mathord}{sfletters}{"1F}
\DeclareMathSymbol{\sfpsi}{\mathord}{sfletters}{"20}
\DeclareMathSymbol{\sfomega}{\mathord}{sfletters}{"21}
\DeclareMathSymbol{\sfvarepsilon}{\mathord}{sfletters}{"22}
\DeclareMathSymbol{\sfvartheta}{\mathord}{sfletters}{"23}
\DeclareMathSymbol{\sfvarpi}{\mathord}{sfletters}{"24}
\DeclareMathSymbol{\sfvarrho}{\mathord}{sfletters}{"25}
\DeclareMathSymbol{\sfvarsigma}{\mathord}{sfletters}{"26}
\DeclareMathSymbol{\sfvarphi}{\mathord}{sfletters}{"27}
\DeclareSymbolFont{bsfletters}{OML}{cmbrm}{b}{it}
\DeclareMathSymbol{\bsfalpha}{\mathord}{sfletters}{"0B}
\DeclareMathSymbol{\bsfbeta}{\mathord}{sfletters}{"0C}
\DeclareMathSymbol{\bsfgamma}{\mathord}{sfletters}{"0D}
\DeclareMathSymbol{\bsfdelta}{\mathord}{sfletters}{"0E}
\DeclareMathSymbol{\bsfepsilon}{\mathord}{sfletters}{"0F}
\DeclareMathSymbol{\bsfzeta}{\mathord}{sfletters}{"10}
\DeclareMathSymbol{\bsfeta}{\mathord}{sfletters}{"11}
\DeclareMathSymbol{\bsftheta}{\mathord}{sfletters}{"12}
\DeclareMathSymbol{\bsfiota}{\mathord}{sfletters}{"13}
\DeclareMathSymbol{\bsfkappa}{\mathord}{sfletters}{"14}
\DeclareMathSymbol{\bsflambda}{\mathord}{sfletters}{"15}
\DeclareMathSymbol{\bsfmu}{\mathord}{sfletters}{"16}
\DeclareMathSymbol{\bsfnu}{\mathord}{sfletters}{"17}
\DeclareMathSymbol{\bsfxi}{\mathord}{sfletters}{"18}
\DeclareMathSymbol{\bsfpi}{\mathord}{sfletters}{"19}
\DeclareMathSymbol{\bsfrho}{\mathord}{sfletters}{"1A}
\DeclareMathSymbol{\bsfsigma}{\mathord}{sfletters}{"1B}
\DeclareMathSymbol{\bsftau}{\mathord}{sfletters}{"1C}
\DeclareMathSymbol{\bsfupsilon}{\mathord}{sfletters}{"1D}
\DeclareMathSymbol{\bsfphi}{\mathord}{sfletters}{"1E}
\DeclareMathSymbol{\bsfchi}{\mathord}{sfletters}{"1F}
\DeclareMathSymbol{\bsfpsi}{\mathord}{sfletters}{"20}
\DeclareMathSymbol{\bsfomega}{\mathord}{sfletters}{"21}
\DeclareMathSymbol{\bsfvarepsilon}{\mathord}{sfletters}{"22}
\DeclareMathSymbol{\bsfvartheta}{\mathord}{sfletters}{"23}
\DeclareMathSymbol{\bsfvarpi}{\mathord}{sfletters}{"24}
\DeclareMathSymbol{\bsfvarrho}{\mathord}{sfletters}{"25}
\DeclareMathSymbol{\bsfvarsigma}{\mathord}{sfletters}{"26}
\DeclareMathSymbol{\bsfvarphi}{\mathord}{sfletters}{"27}

\newcommand{\bdryOp}{\calliD}
\newcommand{\hbdryOp}{\calliD_h}
\newcommand{\bdryOpAdj}{\calliD^\ast}
\newcommand{\hbdryOpAdj}{\calliD^\ast_h}

\newsavebox{\foobox}

\newcommand{\Lshaped}[1]{%
\begin{tikzpicture}[#1]%
\draw (0,0) -- (0ex,1ex);%
\draw (0.5ex,0) -- (0.5ex,1ex);%
\draw (1ex,0.5ex) -- (1ex,1ex);%
\draw (0,1ex) -- (1ex,1ex);%
\draw (0,0.5ex) -- (1ex,0.5ex);%
\draw (0,0) -- (0.5ex,0ex);%
\end{tikzpicture}%
}

\newcommand{\om}{\varOmega}
\newcommand{\som}{{\scriptscriptstyle\varOmega}}
\newcommand{\veps}{\varepsilon}
\newcommand{\U}{U}
\newcommand{\V}{V}

\newcommand\bknote[2][]{\todo[inline, caption={2do}, color=red!50 #1]{
\begin{minipage}{\textwidth-4pt}\underline{BK:} #2\end{minipage}}}
\newcommand{\RRR}{\mathbb{R}}

\newcommand{\D}{\mathcal{D}}

\newcommand{\lprp}[1]{\sideset{^\perp}{}{\mathop{#1}}}
\newcommand{\ip}[1]{\langle {#1} \rangle}

\newcommand{\oh}{\om_h}
\newcommand{\ohE}{\om_{h, E}}

\newcommand{\vpi}{\varPi}

\newcommand{\LL}{\mathcal{L}}

\newcommand{\Igrad}{\mcI_\mathrm{grad}}
\newcommand{\Idiv}{\mcI_\mathrm{div}}
\newcommand{\vsig}{{\vec{\sigma}}}
\newcommand{\vpsi}{{\vec{\psi}}}

\let\d\partial